\documentclass[9pt,a4paper,leqno]{article}
\usepackage{blindtext}
\usepackage{fullpage}
\usepackage{amsmath}
\usepackage{amsthm}
\usepackage{amsfonts}
\usepackage{amssymb}
\usepackage{mathrsfs}
\usepackage{cases}
\usepackage{comment}
\usepackage{graphicx} 
\usepackage{mathtools}
\usepackage{soul}
\usepackage{mathabx}
\usepackage{xcolor}
\usepackage[hidelinks]{hyperref}
\usepackage{bbm}

\newtheorem{assumption}{Assumption}

\newtheorem{theorem}{Theorem}[section]
\newtheorem{notation}[theorem]{Notation}
\newtheorem{remark}[theorem]{Remark}
\newtheorem{definition}[theorem]{Definition}
\newtheorem{proposition}[theorem]{Proposition}
\newtheorem{lemma}[theorem]{Lemma} 
 
\newtheorem{corollary}[theorem]{Corollary}

\numberwithin{equation}{section}

\definecolor{LB}{rgb}{0,0.5,0.5}

\newcommand{\g}{\gamma}
\newcommand{\al}{\alpha}
\newcommand{\N}{\mathbb N}
\newcommand{\be}{\beta}
\newcommand{\R}{\mathbb R}
\newcommand{\lam}{\lambda}
\newcommand{\eps}{\varepsilon}
\newcommand{\E}{\mathbb E}

\def\sha{{\cal A}}
\def\shb{{\cal B}}
\def\shc{{\cal C}}
\def\shd{{\cal D}}
\def\she{{\cal E}}
\def\shf{{\cal F}}

\def\shn{{\cal N}}
\def\shl{{\cal L}}
\def\shp{{\cal P}}
\def\shs{{\cal S}}

\def\Q{\mathbb Q}
\def\P{\mathbb P}

\def\dots{...\,}

\title{Interacting particles in random irregular media}

\author{Luca Bondi\footnote{University of Torino and ENSTA, Institut Polytechnique de Paris, luca.bondi@unito.it}, Elena Issoglio\footnote{University of Torino, elena.issoglio@unito.it}, and Francesco Russo\footnote{ENSTA, Institut Polytechnique de Paris, francesco.russo@ensta.fr}}
\date{October 2026}

\begin{document}

\maketitle

\begin{abstract}
  In this paper,
  we first develop novel tools of independent interest in the field of stochastic calculus related to SDEs
  with distributional drift, namely the formulation of a quasi-Stroock-Varadhan martingale problem
  in the sense that the test functions are of class $C^{1,2}$ instead of solutions
  of Kolmogorov type PDE. This has several applications, among which,
  the time restriction and the notion of conditional
  rough martingale problem and related solution.

  A significant focus is also brought to singular systems of $d N$-dimensional particles defined by a
  $dN$-dimensional stochastic differential equation (SDE) with additive Brownian noise.
  The singularity  features in the drift coefficient, which is a distribution in a Besov space
  of negative index $-\beta$. We study the well-posedness of the system and the tightness of its associated empirical measures.

  \end{abstract}


  
\tableofcontents

\section{Introduction}\label{sc:intro}

  The focus of this paper is
  the system of singular interacting $d$--dimensional particles given by
\begin{equation}\label{eq:IPSBIR}
	\left\{
	\begin{array}{l}
		dX_t^{i,N}=F(K*\mu^N(X_t^{i,N}))b(t,X_t^{i,N})dt+dW_t^i,  \ t\in [0,T]\\
		X_0^{i,N}\sim \nu,\ 1\leq i\leq N,
	\end{array}
	\right.
\end{equation}
where $N\in\N$,  $T$ is a finite time horizon,
$F:\R\to\R$ is a non--linear differentiable function, $K: \R^d\to\R$ is a $\be$--H\"older continuous
probability kernel for $\be\in(0,1/2)$, $b$ is a time-dependent generalised function
taking values in a negative inductive Besov space of index $(-\beta)+$, namely
$b\in C([0,T];\shc^{(-\beta)+}(\R^d;\R^d))$ (see Definition \ref{def:besov} below),  $W$ is a $d$--dimensional Brownian motion and $\nu$ is a Borel probability measure on $\R^d$.
 $\mu^N=\frac1N\sum_{j=1}^N\delta_{X_t^{j,N}}$ is the so called empirical measure where $\delta$ is the Dirac delta measure.

The particle system \eqref{eq:IPSBIR} describes a collection of $d$-dimensional interacting
particles where the interactions are captured by the singular kernel $K$, which models a
pairwise interaction between the particles averaged through the empirical measure $\mu^N$:
of course the term $K*\mu^N(X^{i,N}_t)$ can be re-written as 
$$
\frac1N\sum_{j=1}^NK(X_t^{i,N}-X_t^{j,N}).
$$
At the initial time $t=0$, the system is assumed to start from an
exchangeable configuration where 
$(X_0^{i,N}),1\leq i\leq N$ are drawn from the common law $\nu$.

At this stage the system \eqref{eq:IPSBIR} is only formal, since
the drift component $b(t, \cdot)$
is not a function but a Schwartz distribution 
so that the evaluation at points does not make sense. In fact the product
$F\left(\frac{1}{N}\sum_{j=1}^N K(X_t^{i,N}-X_t^{j,N})\right) \cdot b(t,X_t^{i,N})$, is
considered as the pointwise product
in the sense of Bony, see e.g. \cite{bony}.
It is a again a well-defined distribution since
$(x^1, \ldots,x^N) \mapsto F\left(\frac{1}{N}\sum_{j=1}^N K(x^i-x^j)\right)$
belongs to $C_T \shc^\beta(\R^{dN})$.
We can then consider the system \eqref{eq:IPSBIR} within the framework of Stochastic Differential Equations (SDEs) with distributional coefficients, which are in general modelled in the literature
as suitable martingale problems, see e.g. \cite{issoglio_russoMPb}.
It is important to emphasise that a process solving \eqref{eq:IPSBIR} is generally
not a semimartingale.
For this reason it will be improper to consider  test functions $f$ of class $C^{1,2}$ describing 
the decomposition $f(t, X^1_t, \ldots, X^N_t), t \in [0,T]$, as in
the case of the Stroock-Varadhan martingale problem.

One significant question concerns the behaviour of the empirical measures $\mu^N$
when $N \rightarrow \infty$. As far as we know this is the first paper
attacking this problem of particle systems which are not semimartingales.
This paper not only investigates the first step, i.e. the tightness,
but, in order to reach that objective, it establishes a series
of new useful results in the framework of the
stochastic analysis of SDEs with distributional drift.

As we already mentioned, this paper is concerned with two connected main subjects.
\begin{enumerate}
\item How the macroscopic evolution of an interacting particle system is captured through the  behaviour of its empirical measures.
  \item The development of new stochastic analysis tools for  SDEs with distributional drift.
\end{enumerate}
We mention below some short historical elements  related to the two subjects.
\begin{enumerate}
\item  In the literature, topic 1. is very closely related to
  the so called {\it chaos propagation},
The core mathematical idea behind
is the asymptotic independence that emerges in large interacting particle systems as the number of particles $N \to \infty$. Consider $N$ interacting particles
$X_t^{i,N} \in \R^d, 1 \le i \le N$. Even though this phenomenon was observed by some
theoretical physicists starting probably with L. Boltzmann, it seems
that that property was introduced, as {\it Boltzmann property}, in \cite{kac}. On the
other hand the  terminology {\it chaos propagation} appears only later in \cite{mckeanChaos}.
According to Proposition 2. of \cite{Sznitman}, the propagation of chaos is equivalent to the fact that as $N \to \infty$, the empirical measure $\mu^N$
(as random element taking values in $\shp(\Omega_T^d)$)
converges weakly to a deterministic law $U$.
Suppose that the coupled particle system is a solution of stochastic differential equation
$$dX_t^{i,N} = \tilde b\left(t,X_t^{i,N}, \mu_t^N\right)dt + \tilde \sigma\left(t,X_t^{i,N}, \mu_t^N\right)dW_t^i, \quad 1 \le i \le N,$$
where $\tilde b:[0,T]\times \R^d\times \shp(\R^d)\to \R^d$ and $\tilde \sigma:\tilde b:[0,T]\times \R^d\times \shp(\R^d)\to \R^{d\times d}$ are two bounded smooth coefficients.
When the initial states $\{X_0^{i,N}\}_{i=1}^N$ are independent and identically distributed with law $\nu$ (meaning the joint initial law is the product measure $\nu^{\otimes N}$),
then propagation of chaos holds, so that the empirical measure converges weakly to $U$ being the 
the law of the solution  $\bar X$ satisfying 
$$d\bar{X}_t = \tilde b\left(t,\bar{X}_t, \text{Law}(\bar{X}_t)\right)dt + \tilde \sigma\left(t,\bar{X}_t, \text{Law}(\bar{X}_t)\right)dW_t,$$

One basic reference in the literature is the Saint-Flour course of
A.S. Sznitman,
 see e.g. \cite{Sznitman},
 which in particular formalises Kac's notion of chaos
 when the diffusion is  the identity matrix ($\tilde \sigma={\rm Id}$) and the the drift is
 smooth of convolution type.
 Since the convolution kernel is fixed, one speaks often
 of {\it strong} (resp. {\it weak}) interactions when the the support is
 relatively large (resp. small).
 
 In \cite{Oelschlaeger}, the author investigates a model of  interacting $N$-particle system
 with so called {\it moderate interaction}, i.e. the dependence on the interacting measure is tuned
 via a convolution with some smooth probability densities sequence $V^N$, which converges to the $\delta$-measure. 
 In particular, if $\mu^N$ is the empirical measure, then  the empirical measure at each time $t$, i.e. $\mu^N_t$
 (together with $V^N*\mu^N_t$)
converges weakly (when $N$ goes to infinity) to
 a deterministic density  law $u_t$.
 In this sense, it appears a trade-off between the
 smoothing of the kernel and the
 number of particles.

 In this case, $u_t, t \in [0,T],$  constitute the time-marginals of the law of
 a 
 solution $U$ to a McKean-Vlasov SDE
 and  $t \mapsto u_t, $
 is solution of a non-linear (Fokker-Planck) PDE.
 Later, \cite{coppoletta} extended this result,
 proving the effective convergence of the empirical measures $\mu^N$
 to $U$, whereas in \cite{Oelschlaeger} only the time-marginals
 were proved to converge.

Then, in \cite{JourMeleard}, the authors further generalised these limiting regimes to non-constant diffusion coefficients, establishing quantitative rates of convergence.  
 We also mention some empirical numerical results
 by in the case of a generalized porous media
(with possibly discontinuous  coefficients)
 equations
 confirming the propagation of chaos even in this irregular context, see e.g.
 \cite{cuvelier1, cuvelier2}.
Recent relevant references can be found for 
  classical McKean-Vlasov associated with particle systems, see Chapter 3 of \cite{BIR_Review}, and for a recent review on propagation of chaos, see \cite{chaintron1,chaintron2}.

\item

  Also in topic 2. literature is now quite rich and very active.
We refer to an SDE of the type

\begin{equation}
	\label{eq:SDEmod}
	dX_{t} = B(t,X_{t})dt + dW_{t}, \quad X_{0} \sim \nu, \quad t \in [0,T],
\end{equation}
where $W$ is a standard $d$-dimensional Brownian motion, $\nu$ is a Borel probability measure on $\R^{d}$, and
the singular drift $B \in C([0,T];\shc^{(-\beta)+}(\R^d;\R^d))$,
with $-\beta \in (-1/2,0)$. 
In some papers the  noise is often also multiplicative in the sense that
a non-identity diffusion term appears.
The presence of the Brownian motion term is crucial to  give a sense to SDEs as \eqref{eq:SDEmod}.
Indeed, the system would be inherently ill-posed without noise, whereas the presence of $W$
enables a framework of regularisation, this phenomenon is called {\em regularisation by noise}, see e.g. \cite{flandoli-book}.

Stochastic differential equations with  distributional coefficients have been investigated extensively since the early 2000s.
Foundational work in one-dimension are
\cite{frw1,frw2, rtrut}  (resp. \cite{BassChenDir}) as far as weak (resp. strong) solutions are concerned.
The solutions were Dirichlet processes, in the sense of \cite{FolDir}:
the weak solutions were expressed with the help of a suitable martingale problem,
the strong solution as limit of more regular diffusion processes. 

Later an impressive activity on the subject in the multidimensional case has been performed:
many of them implemented the stochastic sewing lemma techniques, see \cite{KhoaLe}, which generalises the deterministic sewing lemma \cite{FeyelDeLaPradelle1,gubinelli}.
Among the applications we mention
 \cite{anzelletti24, athreya2020, butkovsky.et.al2021,  goudenege23}.
 Contributions  can be found also when the noise is a fractional Brownian motion,
 see e.g. \cite{GaleatiGerencser25}.

Other techniques have continued the martingale problems formulation, see
 \cite{flandoli_et.al14}, \cite{issoglio_russoMPb}. 
 \cite{chaudru_menozzi} investigated the case of alpha-stable noise and \cite{cannizzaro}
 has considered a context beyond the so called Young regime.
For other references we refer to
 Sections 5.4 and 5.5 of \cite{BIR_Review} and \cite{BandiniRusso_DistrDrift}
which extends the one-dimensional techniques to path-dependent SDEs with jumps.
\end{enumerate}
In view of our paper results about 
item 2. we formulate our framework.
Instead of addressing \eqref{eq:SDEmod} directly, one characterises
the law under which the canonical process formally satisfies \eqref{eq:SDEmod},  via a {\it rough martingale problem},
denominated by rMP($B, \nu)$, see Definition \ref{def:rMP}, associated with its parabolic generator
\begin{equation}\label{eq:shl}
	\mathcal{L}f = \partial_{t}f + \frac{1}{2}\Delta f + \nabla f \cdot B.
      \end{equation}
      and initial condition $\nu$.
      The classical martingale problem formulation is to be attributed to Stroock and Varadhan
      \cite{stroock_varadhan}. 
Instead of directly studying the SDE
\begin{equation}\label{eq:SDE_SV}
	dX_t=b(t,X_t)dt+\sigma(t,X_t)dW_t,
\end{equation}
 Stroock and Varadhan characterised the probability measure $\P$ on the canonical space $\Omega_T^d = C([0, T], \R^d)$
 so that, the quantity
$$f(t,X_t)-f(0,X_0)-\int_0^t L f(s,X_s)\, ds$$
is a local martingale under $\P$, for every
$f \in C^{1,2}([0,T] \times \R^d)$,
where
$$Lf(t,x) = \frac{1}{2} \sum_{i,j=1}^d a_{ij}(t,x) \partial_{x_i x_j}^2 f(t,x) + \sum_{i=1}^d b_i(t,x) \partial_{x_i} f(t,x),$$
  and $a(t,x)=\sigma\sigma^\top(t,x)$. This approach has been proven to be equivalent to finding solutions in law to \eqref{eq:SDE_SV}, see e.g. Section 5.4 Proposition 4.6 of \cite{karatzasShreve}. 
  This strategy was then generalised in our context
\eqref{eq:SDEmod}  of singular drifts, see
  e.g.  \cite{issoglio_russoMPb} which has characterized the
  so called {\it virtual solution} of \cite{flandoli_et.al14}.
  \cite{cannizzaro} has generalized this to a even more singular situation
  (beyond the Young regime),
  in the framework of paracontrolled distributions.
  The main difference between the rough martingale problem
  formulation and the classical Stroock–Varadhan one is that
  $C_b^{1,2}([0,T]\times \R^d)$ cannot be a priori used since
  $f(t,X_t)$ is generally not a semimartingale but only a
  weak Dirichlet process,
  see e.g. Chapter 15 of \cite{Russo_Vallois_Book} or \cite{er2} for the explanation of the concept.
  Consequently the chosen domain to replace $C_b^{1,2}([0,T]\times \R^d)$ is constituted
by a suitable class $\shd_\shl$ of bounded functions being solutions of $\shl f= g$, see Notation
\ref{def:D} where $g$ is a function, so that $\int_0^\cdot  g(s,X_s)\, ds$
is well-defined. A probability $\P$ will be a solution if 
$f(t,X_t)-f(0,X_0)-\int_0^t \shl f(s,X_s)\, ds$ is a $\P$-local martingale
for every $f \in \shd_\shl$.
More details are given in
 Section 5.4 of \cite{BIR_Review}.

 Section 4. of this work is devoted to study several properties of the
 rough martingale problem associated with \eqref{eq:SDEmod}
 in particular introducing different domains for which the martingale
 problem is well-posed it has a unique solution, see Section
 \ref{ssc:equivdoms}.
 One of those domains is $\hat \shd_\shl$, see Section \ref{sec:LGrMP},
 which contains
 a suitable class of (this time unbounded) solutions to
 $\shl f= g$.

 The most significant ones are given in 
 Theorem \ref{thm:extendA} where we introduce domains
 of $C^{1,2}$-type, making the formulation closer to
 the Stroock-Varadhan one. As mentioned earlier for $f$ of this type
 $f(t,X_t)$ is not a semimartingale so we need to develop some
 stochastic calculus techniques to expand those processes.
A very useful technique concerns
 the local-time operator $$C_T\shc^{0+}\ni g\mapsto \int_0^\cdot g(r,X_r)dr\in \shc,$$ which is classically well-defined for bounded continuous functions, can be continuously extended to a
 broader class of Besov distributions, i.e. $C_T\shc^{(-\beta)+}$, see Proposition \ref{pr:extendA}.

 The aforementioned  quasi-Stroock-Varadhan formulation
 has four applications.

 \begin{enumerate}
\item
   First, the dynamics of the solution to \eqref{eq:SDEmod}
   is focused via genuine stochastic analysis, in the sense
   that we specify that the canonical process is indeed a solution
   of the SDE \eqref{eq:SDEmod},
   which is an immediate when $B$ is a function.
   Corollary \ref{cor:extendA} provides a precise sense of the SDE
   taking into account the fact that
   $id(x)=x$ belongs to $C_{ub}^{2+}(\R^d)$. For some aspects
   this procedure simplifies a bit the formalism
   of Section 6. of \cite{issoglio_russoMPb}.
\item 
  Secondly, we answer a natural question related
  to the solution $\P$ of the rough martingale problem
  on an interval $[0,T]$. For $t \in [0,T)$, we expect that  the restriction
  $\P^t$ of $\P$  to $\Omega^d_t$ is (the unique)
  solution of the rough martingale problem
  related to the restricted drift $B\vert{}_{[0,t]}$.
  With the original formulation of the rough martingale problem
  with domain $\shd_\shl$, it was a mathematical change to prove
  this property, see Remark \ref{rmk:difficulty}.
  With the new quasi-Stroock-Varadhan formulation
  this is almost an immediate consequence.
  \item 
     We discuss a technique to recover solutions
   to {\it conditional rough martingale problems},
   namely given $\P$ solution to a rough martingale problem on $\R^d$ we construct $\P_{\bar \g}$ from the conditional expectation of $\P$ with respect to $(Z^2=\g^2,\ldots,Z^d=\g^d)$ such that $\P_{\bar\g}$ is a solution to a rough martingale problem on $\R$.
   When $B$ is a function, the fact that the canonical process is a solution
   of a classical SDE makes the procedure quite direct.
   In the case of rough martingale procedure
   the quasi-Stroock-Varadhan formulation becomes an efficient tool
   mimicking the case when $B$ is a function.
 \item
   We  establish useful uniform $L^p$-estimates for a
   generic solution $\P$ of rMP($B,\nu$)
   of the Besov norm of the drift $B$.

\end{enumerate}
Finally we use previous tools to establish the tightness of the empirical measures
$\mu^N$ of the singular  interacting particles \eqref{eq:IPSBIR}.
By  Proposition 2.2 of \cite{Sznitman}, this property is equivalent to
the tightness of the law $\shl$aw($Z^{1,N}$) of the sequence of the first particle $Z^{1,N}$, taking into account the exchangeability of every $N$-system of
particles. Theorem \ref{thm:MomentsCond}) shows finally the tightness
of the laws of the $(Z^{1,N})$ making use of items 3. and 4. above.
The main result is stated in Corollary \ref{cor:MomentsCond}.

\medskip

 The remainder of this paper is organised as follows.
 Section \ref{sc:prelim} introduces the topological setting and provides preliminary Schauder and Bernstein estimates on Besov and inductive spaces,
 in particular for functions belonging to unbounded
 Besov-H\"older funtions.
Section \ref{sec:AR} contains new analytical results
on Besov distributions, in particular  the rigorous embedding of $d$-dimensional distributions into $dN$-dimensions
ones, see Section \ref{ssc:wellpos}.
In the same section, Section \ref{sec:Kolmogorov} provides some
useful technical results on
Kolmogorov PDEs involving Besov distributions.

Section \ref{sec:LG} is focused on the rough martingale problem: Section \ref{ssc:domain} recalls basic definitions and existing results on the rough martingale problem, Section \ref{ssc:equivdoms} compares the different domains used in the literature, Section \ref{sec:LGrMP} defines the linear growth
rough martingale problem, and Section \ref{sec:ext-op} recovers the quasi-Stroock–Varadhan formulation and the dynamics of the SDE \eqref{eq:SDEmod}. Section \ref{sec:extension-alia} examines the structural properties of the
rough martingale problem, such as {\it time restriction} and
moments.
Next, Section \ref{sec:CRMP} addresses the conditional rough martingale problem. Finally, Section \ref{sec:IPS} applies these tools to establish the
     tightness of empirical measures for the interacting particle system \eqref{eq:IPSBIR}.

     \section{Topological setting: definitions and recalls }

     \label{sc:prelim}
 
\subsection{Basic preliminaries}

 \label{sc:basic-prelim}

 We give first some definitions and preliminary results on spaces and functions which will be used later on. $T >0$ will be a fixed terminal time. By convention, any vector $x\in\R^d$ is a column vector.
 Nevertheless the scalar product of two vectors $x$ and $y$
 will simply be denoted $x y$. Given $x \in \R^d$ (resp. a finite-dimensional matrix $M\in \R^{d\times d}$)
     $\vert x\vert $ (resp. $\vert M\vert $) will denote the Euclidean 
     (resp. Frobenius) norm of  $x$ (resp. $M$). 
If $F$ is a finite dimensional  normed space  and $E$ is a given set,
for   $f:E \rightarrow F$ 
we denote $\|f\|_{\infty}:= \sup_{x\in E} \vert f(x)\vert  $.
$C_b^2(\R^d)$ (resp. $C^{1,2}([0,T] \times \R^d)$)
will be the linear subspace
of $C^2(\R^d)$ (resp. $C_b^{1,2}([0,T] \times \R^d)$) constituted by real-valued bounded
functions with bounded derivatives.
This is also a Banach space 
equipped with a natural norm which sums up 
the usual sup-norm the function and all the derivatives.

The space $C^{1,2}_{ubuc}([0,T]\times\R^d)$ will be the space of (not necessarily
bounded functions) $f$ such that $\partial_tf,\  \nabla f,\ {\rm Hess}(f)$ are bounded and ${\rm Hess}(f)$ is uniformly continuous.
The space $C_{ub}^{1,2+}([0,T]\times\R^d)$ (resp. $C_{ub}^{2+}(\R^d)$) will be the set of
(possibly unbounded) functions
in  $C^{1,2}([0,T] \times \R^d)$ (resp. $C^{2}(\R^d)$)
with bounded derivatives,
such that there exists $\epsilon>0$ so that its second derivatives w.r.t. $x$ are 
$\epsilon$-H\"older continuous (uniformly with respect to $t$).
We will also denote by
 $C_{b}^{1,2+}([0,T]\times\R^d)$ (resp. $C_{b}^{2+}(\R^d)$)
 the linear space of bounded functions belonging to
  $C_{ub}^{1,2+}([0,T]\times\R^d)$ (resp. $C_{ub}^{2+}(\R^d)$).
  We observe that we have the inclusion of spaces $C_{ub}^{1,2+}([0,T]\times\R^d)\subseteq C^{1,2}_{ubuc}([0,T]\times\R^d)$ holds.
  Finally the set of functions $C_c^{1,2+}([0,T]\times \R^d)$ will denote the set
  of $C_{ub}^{1,2+}([0,T]\times \R^d) $ functions with compact support.
 
  $C^{0,1}([0,T] \times \R^d)$ will be the Fréchet space of
  real valued continuous functions on $[0,T] \times \R^d$
  such that the first order derivatives are continuous.

We denote  by $\shs:=\shs(\R^d) $, the space of real-valued Schwartz functions on
$\R^d$ and by $\shs':=\shs'(\R^d)$     the space of tempered distributions.
We denote by $\mathcal{F}$ and $\mathcal{F}^{-1}$ the Fourier transform on $\shs$ and
its inverse, see Definition 1.20 in Section 1.2.3 of \cite{sawano}, which are extended on $\mathcal{S}'$ in the usual way, see Definition 1.21 in the same section of \cite{sawano}. The dual pairing on $\shs$ and $\shs'$ will be indicated with $\langle\cdot,\cdot\rangle$. $L^1:= L^1 (\R^d)$ will denote the space of real-valued Lebesgue integrable functions.
Given a topological linear space $E$, $C_TE$ (resp. $B_TE$)  will denote the space of $E$-valued continuous (resp. bounded Borel) functions defined on $[0,T]$
and $\shp(E)$  will denote the set of all Borel  probability measures on $E$.

Throughout the paper $c$ will denote a (universal, but possibly dependent on the
dimension $d$) positive constant, while $C$ will denote a positive  constant which will depend on one or more parameters  (possibly indicated with subscripts), that will be specified case by case. These constants can vary from line to line.
In particular $C_T$ will always indicate a strictly positive increasing function
$T \mapsto C_T$.
Given a function $f:[0,T] \times \R^d \rightarrow \R$, we will often denote $f(s):= f(s,\cdot)$.

We introduce a small discussion about heat semigroups.
\begin{definition}\label{def:heatkern}
  Let $P_t$ denote the semigroup generated by $\frac12\Delta$ on $\shs$, in particular for all $\Phi\in\shs$ we define
  $P_t \Phi = \int_{\R^d} p_t(x-y) \Phi(y) dy, t \ge 0$,
 where $p$ is the heat kernel
\begin{equation}\label{eq:heatkern}
p_t(z):=\frac1{(2\pi t)^{d/2}} \exp^{-\frac{|z|^2}t}.
\end{equation}
It is easy to see that $P_t:\shs\to\shs$. 
Moreover we can extend it to $\shs'$ by dual pairing (and we denote it with the same notation
for simplicity). 
One has $\langle P_t\Psi, \Phi\rangle:=\langle \Psi,P_t\Phi, \rangle$ for all $\Psi\in\shs'$ and $\Phi\in\shs$, the equality holds without changing the sign, because $p$ is symmetric.
\end{definition}
In the whole paper $(P_t)$ will denote the heat semigroup introduced above.

We conclude this  subsection introducing some definitions and conventions related to probability concepts.
Given a probability space $(\Omega, \she, P)$ and a filtration $\shf :=(\shf_t)$ fulfilling the usual conditions
we will denote by $\shc^{\shf, P}$, or simply by $\shc^\P$, when the underlying filtration and probability 
are easily identifiable, the $\shf^X$- adapted and progressively measurable processes defined on $[0,T]$ taking values in $\R^d$. It is an $F$-type space when equipped with a metric
associated with the u.c.p. convergence with respect to $P$.
Throughout the whole paper we will denote by $\Omega_T^d$ the canonical space $C([0,T];\R^d)$ and with $Z$ the canonical process and we will denote by $\shb(\Omega)$ the Borel $\sigma$-algebra on $\Omega$ .

\subsection{Besov-H\"older spaces and inequalities}.

Besov spaces are a fundamental element of this study, and we give a definition of a particular class of those spaces.
\begin{definition}\label{def:besov}
  The Besov space $=\shc^{\g}:=\shc^{\g}(\R^d)$ is defined by set of real-valued functions
    \begin{equation*}
        \shc^{\g}:=\Bigg\{f\in\mathcal{S}':\sup_{j\in\N}2^{j\g}\|\mathcal{F}^{-1}(\varphi_j\cdot\mathcal{F}f)\|_{\infty}<\infty\Bigg\},
    \end{equation*}
    where $(\varphi_j)_{j\in\N}$ is a dyadic partition of unity.
 A classical general reference is \cite{sawano},
 Section 2.1, which uses the denomination Nikolskii–Besov spaces
 or  \cite{bahouri}, Section 1.3.
By Proposition 2.3 in \cite{sawano}, we have that $\shc^{\g'}\subseteq \shc^{\g}$ for all $\g,\g'\in\R$ such that $\g<\g'$.
By Theorem 2.7 in  \cite{sawano}, we know that if $\g\in(0,1)$, then the Besov Space $\shc^\g$ is equivalent to the
H\"older-Zygmund space, which we will still denote $\shc^\g$ by a slight abuse of notation. This is endowed with the norm $\|\cdot\|_{\shc^\g}$, defined for all $h\in\shc^\g$ as

Furthermore, we recall that, if $\g\in\R_+\setminus \N$, then the Besov space $\shc^{\g}$ coincides with the H\"older space, namely the space of bounded functions with bounded derivatives up to order $\lfloor\g\rfloor$ and such that the $\lfloor\g\rfloor$-th derivative is $(\g-\lfloor\g\rfloor)$--H\"older continuous.

\begin{equation}\label{eq:holdnorm}
  \|h\|_{\shc^\g}:=\|h\|_{\infty}+\|h\|_{\g},
\end{equation}
where
$$\|h\|_{\g}:=
\sup_{x,y\in\R^d, x\not=y}\frac{|h(x)-h(y)|}{|x-y|^\g}.$$
If $\g\in(1,2)$
\begin{equation}\label{eq:holdnorm1}
 \|h\|_{\shc^\g}:=\Vert h \Vert_\infty + \Vert\nabla h \Vert_\infty +\sup_{x,y\in\R^d, x\not=y}\frac{|h(x)-h(y)|}{|x-y|^{\g-1}},
 \end{equation}
see Definition 1.49 of \cite{bahouri}. 
 Furthermore, we define the inductive (resp. projective) Besov Spaces $\shc^{\g+}$ (resp. $\shc^{\g-}$)as 
    \begin{equation}\label{eq:indBesov}
        \shc^{\g+}:=\bigcup_{\al>\g}\shc^{\al}\quad,\quad \shc^{\g-}:=\bigcap_{\al<\g}\shc^{\al}.
    \end{equation}
  \end{definition}
  We will denote by $C_c^\gamma$ (resp. $C_c^{\gamma+}, C_c^\infty$)
  the subset of functions in $C^\gamma$ (resp. $C^{\gamma+}, C^\infty(\R^d)$)
  with compact support.
We recall that Lemma B.2 of \cite{issoglio_russoMK} guarantees that $C_T\shc^{\g+}=\bigcup_{\al>\g}C_T\shc^{\al}$. This constitutes a characterisation of these inductive spaces of distributions.

\begin{remark}\label{rm:equivnorm}
We want to point out that the norm introduced in \eqref{eq:holdnorm} is equivalent to the one defined by
\begin{equation}\label{eq:holdnorm2}
\Vert h\Vert_\infty +\sup_{|x-y|<1,\ x\not=y}\frac{|h(x)-h(y)|}{|x-y|^\g},\ \forall h\in\shc^\g.
\end{equation}
Indeed, it holds that 
$$\frac13\|h\|_{\shc^\g}\leq \Vert h \Vert_\infty+\sup_{|x-y|<1,\ x\not=y}\frac{|h(x)-h(y)|}{|x-y|^\g}\leq \|h\|_{\shc^\g},\ \forall h\in\shc^\g.$$
Thanks to this equivalence, we will still denote \eqref{eq:holdnorm2} by $\|h\|_{\shc^\g}$ if there will be no ambiguity of notation. It is easy to verify that this equivalence holds also for \eqref{eq:holdnorm1}.
\end{remark}

In the whole paper, for a given $\gamma \in \R$, we will denote by $\shc^\gamma(\R^d;\R^d)$
(resp. $\shc^\gamma(\R^{dN};\R^{dN})$) any vector whose components are elements of $\shc^\gamma(\R^d)$ (resp.  $\shc^\gamma(\R^{dN})$). For $ \Gamma = (\Gamma^1, \ldots, \Gamma^d)^\top \in \shc^\gamma(\R^d;\R^d)$,
we denote
\begin{equation}\label{eq:normvec}
	 \Vert \Gamma \Vert_{\shc^\gamma(\R^d;\R^d)} := \sup_i \Vert \Gamma^i \Vert_{\shc^\gamma(\R^d)}.
	\end{equation}
A similar convention can be used replacing $d$ with $Nd$.

  The estimates below are known as {\em Schauder's estimates}:  for the proof we refer to \cite[Lemma 2.5]{catellier_chouk}, see also \cite{gubinelli_imkeller_perkowski} for similar results.
%

\begin{lemma}\label{lm:schauder}
	Let $\gamma \in \R$. Then  there exists a constant $c$ such that the following hold.
	 Let $f\in \shc^{\g}$. Then, for any $\theta\geq 0$,
			\begin{equation}\label{eq:schauder1}
			\|P_tf\|_{\shc^{2\theta+\g}}\leq c t^{-\theta} \|f\|_{\shc^\g},        
		\end{equation}
	for all $t>0$. 
	Moreover for  all $\theta \in (0,1)$, if  $f\in\shc^{2\theta+\g}$  we have
	\begin{equation}\label{eq:schauder2}
		\|P_tf-f\|_{\shc^\g}\leq ct^{\theta} \|f\|_{\shc^{2\theta+\g}}.
              \end{equation}

      \end{lemma}
Another classical inequality is stated below.
\begin{lemma}(Bernstein's inequality).\label{lm:bern}
Let $\g\in\R$ then there exists $c>0$ such that 
\begin{equation*}
    \|\nabla g\|_{\shc^\g}\leq c\|g\|_{\shc^{\g+1}},
\end{equation*}
for all $g\in\shc^{\g+1}$.
\end{lemma}
The following is an important estimate, which allows to define the so-called {\em pointwise product} of distributions.
Indeed let $\al,\gamma>0$
with $\al-\gamma>0$.
There exists a constant $c>0$ such that
\begin{equation}\label{eq:bony}
    \|f \cdot g\|_{\shc^{-\gamma}} \leq c\|f\|_{\shc^{\al}}\|g\|_{\shc^{-\gamma}},
  \end{equation}
  for all $f \in \shc^{\al}$ and $g\in \shc^{-\gamma}$: in particular
  the pointwise product $fg$ is a well-defined element of $\shc^{-\gamma}$.
  Furthermore,
for all $f\in C_T\shc^{\al}$ and $g\in C_T\shc^{-\gamma}$, then $ f g \in C_T\shc^{-\gamma}$ and
\begin{equation}\label{eq:bonyt}
    \|f \cdot g\|_{C_T\shc^{-\gamma}} \leq c\|f\|_{C_T\shc^{\al}}\|g\|_{C_T\shc^{-\gamma}}, 
  \end{equation}
  see (2.9) and (2.10) of \cite{issoglio_russoMK}.
  Given an $\R^d$-valued vector function (resp. distribution)
  $f = (f_1, \ldots,f_d)$ (resp. $g = (g_1,\ldots, g_d)$)
  $ f \cdot g$ will denote the distribution
  $\sum_{i=1}^d f_i \cdot g_i$.

We focus now on some spaces of unbounded H\"older  functions, firstly introduced in Section 3 of \cite{issoglio_russoPDEa}, which will play a central role throughout the whole paper.
In particular this set will include functions with linear growth such as the component of the identity function $id(x)=x$.
This set will include  unbounded test functions in the domain of the martingale problem,
see Section \ref{sec:LGrMP}.
\begin{notation}\label{def:DC}
Given $\g\in(0,1)$, we define the space $ D\shc^\g$ as
\begin{equation}
 D\shc^\g:=\{h:\R^d\to\R \text{ differentiable function s.t. }\nabla h\in\shc^\g(\R^d;\R^d)\}.
\end{equation}
On $ D\shc^\g$, we introduce  the norm
\begin{equation}\label{eq:DC}
\|h\|_{ D\shc^\g}:=|h(0)|+\|\nabla h\|_{\shc^\g}.
\end{equation}
\end{notation}
\begin{remark} \label{rmk:DCgamma+}
  We notice immediately that for all $\g\in(0,1)$ the inclusion $\shc^{1+\g}\subseteq  D\shc^\g$ holds.
  Furthermore, from \cite[Lemma 3.1]{issoglio_russoPDEa} the space $( D\shc^\g,\|\cdot\|_{ D\shc^\g})$ is a Banach space. 
  The definition of $D\shc^{\g+}$
(resp. $D\shc^{\g-}$)
  is analogous to the one given in \eqref{eq:indBesov} for $\shc^{\g+}$.
\end{remark}

The lemma below  constitutes an extension of Schauder's estimates
for the aforementioned spaces. Its proof is postponed to Section \ref{app:B1} of the Appendix.

\begin{lemma}\label{lm:schauderDC}
Let $\gamma \in (0,1)$.
Then,  there
is  constants $ C_T > 0$, such that $T \mapsto C_T$ is increasing, such that
for all $\theta\ge 0$ such that $2\theta + \gamma <1$
\begin{equation}\label{eq:schaud1DC}
     \|P_tf\|_{D\shc^{2\theta+\g}}\le C_T (1 + t^{-\theta})\|f\|_{D\shc^\g}, t \in (0,T],
  \end{equation}
  for all $f\in D\shc^\g$.
  Moreover, there is a universal $c$ such that, for all $\theta \in (0,1)$, we have the following.
		For all $f\in D\shc^{\g+2\theta}$,  it yields
		\begin{equation}\label{eq:schaud2DC}
			\|P_tf-f\|_{D\shc^\g}\leq ct^{\frac{2\theta\wedge 1}2} \|f\|_{D\shc^{2\theta+\g}}.
		\end{equation}

\end{lemma}



In the sequel we will make use of the following.

\begin{assumption}\label{ass:beta}
 Let  $0<\be<\frac12$.
\end{assumption}

\begin{assumption}\label{ass:B}
	$B = (B^1,\ldots, B^d)^\top \in C_T\shc^{(-\beta)+}(\R^d;\R^d)$.
\end{assumption}

We recall here a result from \cite{issoglio_russoMK}, which will be used later.
\begin{lemma}(\cite[Proposition 2.4]{issoglio_russoMK}).\label{lm:Xncont}
Let $g\in C_T\shc^{(-\beta)+}$.  Let us consider the sequence $(g^n)$ defined by
\begin{equation}\label{eq:molbn}
g^n(t,·) :=p_{\frac1n} \ast g(t,\cdot),
\end{equation}
for any fixed $t \in[0,T ]$, $n \geq1$, where $p$ is defined in \eqref{eq:heatkern}. Then
we have the following.
\begin{itemize}
\item[(i)] For each $n$, $g^n$ is globally bounded, together with all its space derivatives.
\item[(ii)] For each $n$, $t\mapsto g^n(t,\cdot)$ is continuous in $\shc^\g$ for all $\g>0$. In particular $g^n\in C_T\shc^{(-\beta)+}$.
\item[(iii)] We have the convergence $g^n \to g $ in $C_T\shc^{(-\beta)+}$.
\end{itemize}

\end{lemma}

 We prove here a modification of \cite[Lemma 3.2]{issoglio19}, for the case of the backward heat semigroup.

 The proof of the lemma below is postponed to Section \ref{app:B2} of the Appendix.

  \begin{lemma} \label{lm:issoglio}
	If  $h\in B_T \mathcal C^{(-\beta )+}$ then $\int_\cdot^T P_{s-\cdot} h(s) ds \in C_T\shc^{1+\be}$.
      \end{lemma}

\section{New analytical results on Besov distributions}

\label{sec:AR}

Throughout the section Assumption \ref{ass:beta} will hold.

\subsection{The Kolmogorov equation and related estimates in $D \shc^{\beta+}$}

\label{sec:Kolmogorov}

In this section we suppose the validity of Assumptions \ref{ass:beta} and \ref{ass:B} and
we focus on various analytical results concerning the parabolic
Kolmogorov PDE with distributional drift
\begin{equation} \label{eq:KolmPDE}
	\partial_t u+\frac12\Delta u+ \nabla u  \cdot B = g, \ u(T,\cdot) = u_T,
\end{equation}
in the
space of unbounded H\"older functions introduced in
Notation \ref{def:DC}.
Indeed here
$u_T \in   D\shc^{\beta+}$ and
$g \in C_T\shc^{(-\beta)+}$.

Solutions of \eqref{eq:KolmPDE} are considered 
in the weak and mild sense in $C_T  D\shc^{\beta}$, see Definitions 4.1 and 4.2.
in \cite{issoglio_russoPDEa},
which are shown to be equivalent in Proposition 4.5 therein.
We recall here the mild formulation, which is the one used throughout the paper.
\begin{definition}\label{def:mild}
	Let  $u\in C_T   D\shc^{\beta}  $. Then u is said to be a mild solution of \eqref{eq:KolmPDE} if it satisfies
	\begin{equation}\label{eq:mild}
		u(t) = P_{T-t} u_T  + \int_t^T P_{s-t} \left(\nabla u(s)\cdot B(s) \right) d s - \int_t^T P_{s-t}  g(s)d s, \ \forall t\in[0,T].
	\end{equation}
	
\end{definition}

\begin{remark} \label{rmk:componentwise}
	Often $g, u_T$ take value $\R^d$. In that case  the
	solution $u$ will also take values in $\R^d$.  In this case $u$  will be said to be a solution of \eqref{eq:KolmPDE}
	if it holds componentwise.
\end{remark}
\begin{theorem}\label{thm:PDEU}
	There exists a unique solution $u$ to \eqref{eq:KolmPDE} in
	$C_T  D\shc^{\beta}$. If moreover $u_T\in \shc^{1+\beta}$, then $u\in C_T\shc^{1+\beta}$.
\end{theorem}
	\begin{remark} \label{rmk:PDEU}

In fact, the aforementioned unique solution also belongs to 
$C_T  D\shc^{\beta+}$. If, moreover, $u_T\in \shc^{(1+\beta)+}$, then $u\in C_T\shc^{(1+\beta)+}$.
\end{remark}
\begin{proof}[Proof of Theorem \ref{thm:PDEU}]
  This is the content of Theorem 4.7 and Remark 4.8 in \cite{issoglio_russoPDEa} with
  $ b=B$ and $\lambda=0$.
\end{proof}

The theorem below provides upper bounds and time modulus of continuity of
the solution $u$ to \eqref{eq:KolmPDE}. These will turn to be fundamental tools for proving $L^p$ estimates in Section \ref{sec:Lpestim}.
The proof  is postponed to Section \ref{app:C1} of the Appendix.


\begin{theorem}\label{thm:unDC}
Let $0 < \alpha < \beta$.
  There are functions $(t,x,y,z) \mapsto R_{\al,\beta}(t, x,y,z)$
  and $(t,x,y,z) \mapsto R^0_{\al,\beta} (t,x,y,z)$,
  increasing in all variables,
such that the following holds.
Given  $B\in C_T\shc^{-\al}(\R^d;\R^d)$,
$g\in C_T\shc^{-\al}$, $u_T\in D\shc^{\beta+}$,
the unique solution of \eqref{eq:KolmPDE}
$u$ fulfils the properties below.
	\begin{itemize}
	\item[i)] \begin{equation} \label{eq:nablaunDC}
		\| u\|_{C_TD\shc^\beta} \leq  R_{\al,\beta}(T, \|B\|_{C_T\shc^{-\al}},\|g\|_{C_T\shc^{-\al}},\|u_T\|_{ D\shc^{\beta}}).
              \end{equation}
              
	\item[ii)] For $0<\eps<\be$, we have 
	\begin{equation}\label{eq:holdunDC}
          \| u  (t,\cdot)- u  (s,\cdot)\|_{D \shc^\eps}\leq
          \vert t-s\vert^{\frac{\be-\eps}2}
		R^0_{\al,\beta}(T,\|B\|_{C_T\shc^{-\al}},\|g\|_{C_T\shc^{-\al}},\|u_T\|_{ D\shc^\beta}),\quad \forall\ t,s\leq T.
	\end{equation}
      \item[iii)] We suppose now $u_T\in\shc^{1+\beta}$.
        For $0<\eps<\be$, we have
	\begin{equation}\label{eq:holdun}
          \|u  (t,\cdot)-u  (s,\cdot)\|_{\shc^{1+\eps}}\leq
          \vert t-s\vert^{\frac{\be-\eps}2}
		R^0_{\al,\beta} (T, \| B \|_{C_T\shc^{-\al}}, \| g \|_{C_T\shc^{-\al}}, \|u_T\|_{\shc^{1+\beta}}),\quad \forall\ t,s\leq T.
	\end{equation}
	\end{itemize}
	
      \end{theorem}
      \begin{remark} \label{rmk:unDC}
        The explicit form of $R_{\al,\beta}$ (resp. $R^0_{\al,\beta}$) is given in \eqref{eq:Ralpha}
        (resp. \eqref{eq:R0}). Taking $u_T = 0$, $g= 0$, it is obvious that the unique solution of
        \eqref{eq:KolmPDE} is $u \equiv 0$. Since
        $R(T, x, 0, 0) = R^0(T,x,0,0) \equiv 0$,
        we observe that the upper bounds \eqref{eq:nablaunDC}, \eqref{eq:holdunDC} and \eqref{eq:holdun}
are optimal.
        \end{remark}

\begin{corollary} \label{cor:unDC}
  Let $(u_T^n)$ be a bounded sequence in $D\shc^{\beta}$.
Let $0 < \al<\be$.
  We consider a bounded sequence $(B^n)$ (resp. $(g^n)$) 
 in $C_T\shc^{(-\alpha}(\R^d; \R^d)$ 
 (resp. $C_T\shc^{(-\alpha}(\R^d)$).
For each $n$, let $u^n$ be the
  unique solution in
	 $ C_TD\shc^\be$
        to the Kolmogorov equation \eqref{eq:KolmPDE}
	with $B=B^n$, $g=g^n$ and $u_T=u^n_T$.
        We have the following.
	\begin{itemize}
		\item[(i)] The sequence $(\| u^n\|_{C_T  D\shc^\beta})$  is bounded.
		\item[(ii)] Let $\eps \in (0,\beta)$.  The sequence $u^n:[0,T] \rightarrow D\shc^{\epsilon}$  is $\frac{\be-\eps}{2}$-equi H\"older continuous.
       \item[(iii)] If $(u^n_T)$ is a bounded sequence in $\shc^{1+\beta}$ and $\eps\in(0,\beta)$, then the sequence $u^n:[0,T] \rightarrow \shc^{1+\epsilon}$
         is $\frac{\be-\eps}{2}$-equi H\"older continuous.
	\end{itemize}

\end{corollary}
\begin{proof}
  We prove item {\it (i),(ii)} because, for item {\it (iii)}, one just needs to repeat the same argument but using Theorem \ref{thm:unDC} item {\it iii)}. 
	By assumption we have
	$$\sup_{n\in\N}\|B^n\|_{C_T\shc^{-\al}}<\infty,\ \sup_{n\in\N}\|g^n\|_{C_T\shc^{-\al}}<\infty\ \text{ and } \sup_{n\in\N}\|u^n_T\|_{D\shc^{\beta}}<\infty.$$
By  Theorem \ref{thm:unDC} (i) and (ii), taking in particular into account that    $R_{\al,\beta}$ and
	$R^0_{\al,\beta}$ are increasing in all variables,  we have
	\begin{equation}\label{eq:nablaunDCbis}
		\| u^n\|_{C_TD\shc^\beta}\leq R_{\al,\beta}(T,\sup_{n\in\N}\|B^n\|_{C_T\shc^{-\al}},\sup_{n\in\N}\|g^n\|_{C_T\shc^{-\al}},\|u^n_T\|_{ D\shc^\beta})
	\end{equation}
	and
	\begin{equation}\label{eq:holdunDCbis}
		\| u^n(t)- u^n(s)\|_{D\shc^{\eps}}\leq (t-s)^{\frac{\be-\eps}2}R^0_{\al,\beta}(T,\sup_{n\in\N}\|B^n\|_{C_T\shc^{-\al}},\sup_{n\in\N}\|g^n\|_{C_T\shc^{-\al}},\|u^n_T\|_{ D\shc^\beta}).
	\end{equation}
	Thus, we conclude.
\end{proof}
We now state and prove a useful result that will be used throughout the whole paper.

\begin{corollary}\label{cor:unDCBis}    
  Let $g\in C_T\shc^{(-\beta)+}$ and $u_T\in D\shc^{\beta+}$. Let $u
  \in C_TD\shc^{\beta}$ 
  the solution to \eqref{eq:KolmPDE}. Let also $(B^n)\subseteq C_T\shc^{(-\beta)+}(\R^d;\R^d)$ (resp. $(g^n)\subseteq C_T\shc^{(-\beta)+} $)
  converging to $B$ in $C_T\shc^{(-\beta)+}$ (resp. to $g$ in $C_T\shc^{(-\beta)+}$).
  Let $(u^n_T)$ be  a sequence in $ D\shc^{\beta+}$ converging to $u_T$ in $ D\shc^{\beta+}$.

  For each $n$, let us denote by
   $u^n \in C_T D\shc^{\beta}$ be the solution of
\begin{equation}\label{eq:shln}
		\shl^nu^n=g^n,\ u^n(T)=u^n_T,
	\end{equation} 
where $\shl^n u = \partial_t u +\frac12\Delta u+\nabla u \cdot B^n$.

Then the sequence  $(u^n)$ converges to
 $u$ in $C_TD\shc^{\beta+}$.
\end{corollary}

\begin{remark} \label{rmk:unDCBis}
If in previous statement we replace $C_TD\shc^{\beta+}$ (resp. $D\shc^{\beta+}$)
with $C_T\shc^{(1+\beta)+}$ (resp. $\shc^{(1+\beta)+}$)
then, by analogous argument we have that  $(u^n)$ converges 
 $u$ in $C_T\shc^{(1+\beta)+}$.
  \end{remark}

\begin{proof} [Proof of Corollary \ref{cor:unDCBis}]
  Since $u^n$ is the solution to \eqref{eq:shln},
  then, by Remark 4.4 \cite{issoglio_russoPDEa}, $u^n\in C_TD\shc^{\beta}\cap C^1([0,T;\shs'])$, so we can evaluate $\shl u^n$ to get
  $$\shl u^n=\shl^nu^n+\nabla u^n\cdot(B-B^n)=g^n+\nabla u^n\cdot(B-B^n)\in C_T\shc^{(-\beta)+}.$$
  By Corollary \ref{cor:unDC},
  there exists some $C>0$ such that $\|\nabla u^n\|_{\shc^{\beta}}\leq C$ for all $n$.
By assumption, there 
  is  $\al<\be$, such that both  $B^n\to B$ and $g^n\to g$ in $C_T\shc^{-\al}$. Consequently,  taking also into account \eqref{eq:bonyt},  we have that 
  $$ \|g^n+\nabla u^n\cdot(B-B^n)-g\|_{C_T\shc^{-\al}}\leq  \|g^n-g\|_{C_T\shc^{-\al}}+\|\nabla u^n\|_{C_T\shc^{\beta}}\|(B-B^n)\|_{C_T\shc^{-\al}}\leq\|g^n-g\|_{C_T\shc^{-\al}}
  +C\|(B-B^n)\|_{C_T\shc^{-\al}},$$
  which converges to $0$.
 Thus,  Remark 4.18 after Lemma 4.17 in \cite{issoglio_russoPDEa} yields that $u^n\to u$ in $C_TD\shc^{\beta+}$, and this concludes the proof.
\end{proof}

The proposition below relates to the case when $B$ belongs to $C_T\shc^{0+}(\R^d;\R^d)$ and the terminal condition is smooth.
\begin{proposition}\label{pr:lunardi}
	Let $h\in C_T\shc^{0+}(\R^d;\R^d)$, $g\in C_T\shc^{0+}$ and $u_T\in C_b^{2+}(\R^d)$. Then the unique solution $u\in C_T\shc^{(1+\beta)+}$ of 
	\begin{equation}\label{eq:lunardi}
		\partial_t u+\frac12 \Delta u+\nabla u h=g , \ u(T)=u_T.
	\end{equation}
	provided by Theorem \ref{thm:PDEU} is a classical solution, in particular $u\in C_b^{1,2+}([0,T] \times \R^d)$.
\end{proposition}
\begin{proof}
	By considering a time reversal argument and thanks to Theorem 5.1.9 of \cite{lunardi95}, we know that there exists a classical solution $\phi\in C_b^{1,2+}$ to 
	$$\shl \phi=g ,\phi(T)=f_T.$$
	In particular $\phi \in C_T\shc^{1+\beta}$. Since the pointwise product extends the usual product,
and the fact that the PDE admits a unique solution in $C_T\shc^{1+\beta}$ (by Theorem \ref{thm:PDEU})
we have that $f=\phi$.
\end{proof}

\subsection{Immersion of $d$-Dimensional Besov Spaces into $dN$-Dimensions}


 \label{ssc:wellpos}
 In this whole subsection, we fix two positive integers $d$ and $N$.
 As mentioned in the title, we introduce the
 notion of
 immersion of  Schwartz distributions on $\R^d$ into Schwartz distributions in $\R^{dN}$  and related properties.
 In particular, given a Schwartz distribution $h \in \shc^{-\beta}(\R^d)$, $1 \le k \le N$,
$\ell:\R^{dN}\to\R^d$ is some regular function
we need to give  a meaning to the product 
 of the type  $$\ell(x_1,\dots,x_N)h(x_k), \text{ for } x_1,\ldots,x_N\in\R^d, 1 \le k \le N.$$
The issue is that the dimensions $dN$ and $d$ are different
so that, a priori, the product is not a classical pointwise (Bony) product.
 In order to make the dimensions compatible, 
the second factor $x_k \mapsto h(x_k)$
needs to be formally considered as an extension of
$ (x_1,\ldots,x_N) \mapsto h(x_k),$
when $h$ in $\shc^{0+}(\R^d)$.
For this, we construct an embedding map $h \mapsto T^k_h$ from  $\shc^{-\beta}(\R^d)$ into $\shc^{-\beta}(\R^{dN})$
such that, when $h \in \shc^{0+}(\R^d)$ then  $T^k_h$ is the function
$(x_1,\ldots, x_N) \mapsto h(x_k)$.
At this point $ (x_1,\ldots,x_N)  \mapsto h(x_k),$
is replaced rigorously by $(x_1,\ldots,x_N)  \mapsto T^k_{h}(x_1, \ldots, x_N)$ since we construct $T$ so that $T^k_h\in\shs'(\R^{dN})$. 
For simplicity of notation in the sequel of the section $k$ will be fixed to $1$ and $T$ will simply denote $T^1$.
However similar considerations will apply to $T^k, 1 \le k \le N$.

\begin{definition}\label{def:marginal}
	Given a function $\varphi\in \shs(\R^{dN}) (\subseteq L^1(\R^{dN}))$ we denote by  $\Tilde{\varphi}:\R^d \to\R$ the \emph{marginalisation} of $\varphi$,
	that is
	\begin{equation}\label{eq:marginal}
		\Tilde{\varphi}(x_1):=\int_{\R^{d(N-1)}}\varphi(x_1,\dots,x_N)dx_2\dots dx_N.
	\end{equation}
\end{definition}
\begin{definition}\label{def:T}
	For all $h\in\shc^{(-\beta)+}(\R^d)$ we define $T:\shc^{-\beta}(\R^d)\to\shs'(\R^{dN})$ to be the map that associates $h$ to
	\begin{equation}\label{eq:T}
		\langle T_h,\varphi\rangle:=\langle h, \Tilde{\varphi}\rangle,\quad\forall\varphi\in\shs(\R^{dN}).
	\end{equation}
\end{definition}
We immediately get the following.
\begin{lemma}\label{lm:fourier}
	If $\varphi \in \mathcal{S}(\R^d)$ then
	
	\begin{equation}\label{eq:fourier}
		\int_{\R^d} \mathcal{F}^{-1}\varphi(y) dy =\varphi(0).
	\end{equation}
	\begin{proof}
		\eqref{eq:fourier} follows from the definition of extension of Fourier Transform on $\shs'$  which yields the  easy calculation 
		\begin{equation*}
			\varphi(0)=\langle \delta_0,\varphi\rangle=\langle \underbrace{\shf\delta_0}_\text{=$1\in\shs'$},\shf^{-1}\varphi\rangle=\int_{\R^d} \shf^{-1}\varphi(y) dy.\qedhere
		\end{equation*}
	\end{proof}
\end{lemma}

\begin{proposition} \label{prop:Fmarg}
	If $\varphi\in\
	C_c^{\infty}(\R^{dN})\subseteq
	\mathcal{S}(\R^{dN})$, then we have 
	\begin{equation}\label{eq:fourierp}
		\widetilde{\mathcal{F}^{-1}\varphi}(\cdot)=\mathcal{F}^{-1}[\varphi(\cdot,0,\dots,0)].
	\end{equation}
\end{proposition}

\begin{remark} \label{rmk:Fmarg}
	We point out that the Fourier transform on the left-hand side in (\ref{eq:fourierp}) is meant on $\R^{dN}$ whilst the one on the right-hand side is meant on $\R^d$. 
\end{remark}

\begin{proof} [Proof of Proposition \ref{prop:Fmarg}]
	We first prove \eqref{eq:fourierp} in the particular case of functions
	\begin{equation} \label{eq:productBis}
		\varphi(y_1,\dots,y_N):=\prod_{j=1}^N \varphi_j(y_j)\in \mathcal{S}(\R^{dN}).
	\end{equation}
	Then, the general case 
	will follow by a density argument together with
        the continuity of the Fourier transform.
        Indeed the projective tensor product space
        $\bigotimes_{i=1}^N \shs(\R^d)$
        and $\shs(\R^{dN})$ are isomorphic, see Corollary
        of Theorem 51.5 of \cite{treves}. See also
        Chapter 50 of \cite{treves} for the characterisation of nuclear space.
        Since the Fourier transform is multiplicative on the class of functions of the type in \eqref{eq:productBis}, i.e.
	\begin{equation} \label{eq:product}
		\mathcal{F}^{-1}\varphi(y_1,\ldots,y_N) =   \prod_{j=1}^N \mathcal{F}^{-1}  \varphi_j(y_j),
	\end{equation}
	we have
	\begin{equation} \label{eq:productFourier}
		\widetilde{\mathcal{F}^{-1}\varphi}(\cdot)
		=\mathcal{F}^{-1}\varphi_1(\cdot)\prod_{j=2}^{N}\int_{\R^d}\mathcal{F}^{-1}\varphi_j(y_j)dy_j
		=\mathcal{F}^{-1}\varphi_1(\cdot)\prod_{j=2}^{N}\varphi_j(0)=\mathcal{F}^{-1}
		\varphi(\cdot,0,\dots,0),
	\end{equation}
	where we have used Lemma \ref{lm:fourier} in the second equality.

      \end{proof}

			
		\begin{proposition}\label{pr:besovT}
                  The map $T:\shc^{-\beta}(\R^d)\ni h\mapsto T_h\in \mathcal{S'}(\R^{dN})$
                  is linear and it
                  maps into $\shc^{-\beta}(\R^{dN})$
			and it is an isometry in the sense that
\begin{equation} \label{eq:iso}
  \Vert T_h \Vert_{\shc^{-\beta}(\R^{dN})} =  \Vert h \Vert_{\shc^{-\beta}(\R^{d})},
  \end{equation}
			for every $h \in \shc^{-\beta}(\R^d)$.
                        In particular $T$ is injective.
		\end{proposition}
		
		\begin{proof}
	The linearity follows by the 
		linearity of the dual pairing on $\shs(\R^{dN}),\shs'(\R^{dN})$.
                We prove know   \eqref{eq:iso}.
			Let $h \in \shc^{-\beta}(\R^d)$. We need to evaluate 
			the
                        Besov quantity
                        $\|T_h\|_{\shc^{-\beta}(\R^{dN})}$, i.e.
			\begin{equation} \label{eq:BesovNorm}
				\sup_{j \ge -1}2^{-j\beta}\|T_h \ast \mathcal{F}^{-1}\varphi_j\|_{\infty},
			\end{equation}
			where $\{\varphi_j\}_{j \ge -1}$ is a partition of unity in $\R^{dN}$.
			Indeed, for $x \in \R^d$, using Proposition \ref{prop:Fmarg}
			in the first equality, we get
			\begin{align} \label{eq:hTh}
		\notag T_h \ast \mathcal{F}^{-1}\varphi_j(x))  &= 
  \langle T_h, \mathcal{F}^{-1}\varphi_j(x - \cdot)\rangle =
				\langle h,  \widetilde{\mathcal{F}^{-1}\varphi_j}(x_1 - \cdot)\rangle \\
				\notag  &= \langle h, \mathcal{F}^{-1}\varphi_j(x_1 - \cdot, 0, \ldots,0)\rangle
				= \langle h, \mathcal{F}^{-1}\psi_j(x_1 - \cdot)\rangle \\
				&= h \ast \mathcal{F}^{-1}\psi_j (x_1),
			\end{align}
			where
			$$ \psi_j(x_1) := \varphi_j(x_1,0 \ldots 0),\  x_1 \in \R^d.$$
		We now know that Besov norms defined through different partitions of unity are 
			equivalent (see Theorem 2.1 in \cite{sawano}) and
			it can be immediately  verified that
				$\{\psi_j\}_{j\geq -1}$ is a partition of unity for $\R^d$. 
			Finally from \eqref{eq:BesovNorm} and \eqref{eq:hTh} it comes out 
			\begin{equation*}
				\|T_h\|_{\shc^{-\beta}(\R^{dN})}=\sup_{j \ge -1}2^{-j\beta}\|T_h \ast \mathcal{F}^{-1}(\varphi_j)\|_{\infty}=
				\sup_{j \ge -1}2^{-j\beta}\|h \ast \mathcal{F}^{-1}(\psi_j)\|_{\infty}=\|h\|_{\shc^{-\beta}(\R^d)}, 
			\end{equation*}
			which concludes the proof.
		\end{proof}

\section{The quasi-Stroock-Varadhan martingale problem and its dynamics}

\label{sec:LG}

In the whole section Assumptions \ref{ass:beta}
and \ref{ass:B}
will be in force.
This section is devoted to some structural properties
of the solution to rough martingale problems,
introduced in \cite{issoglio_russoMPb}.
In Section \ref{ssc:domain} we recall the notion of rough martingale problem,
which relies on a specific domain of test functions, see e.g. Notation \ref{def:D} in Definition
\ref{def:rMP}. Section \ref{ssc:equivdoms} and Section \ref{sec:LGrMP}  compare previous definition with
alternative possible rough martingale problem formulation
replacing the domain in Notation \ref{def:D}
some other domains, see e.g. Notation \ref{def:domains} or \eqref{eq:domain-hat}.

Nevertheless the main results are included in Section \ref{sec:ext-op}.
The first significant step  consists in showing
that, given a solution $\P$, the {\it local time} operator
\begin{equation*}
  g \mapsto A(g):= \int_0^\cdot g(s,Z_s) ds,
  \end{equation*}
can be extended continuously 
from $C_T\shc^{0+}$ to a linear
map $A^\P: C_T\shc^{(-\beta)+} \rightarrow \shc^\P$, see Proposition
\ref{pr:extendA}.

The second step consists in proving that $\P$
fulfils the so called {\it quasi-Stroock-Varadhan} martingale problem since the domain is
the set of
$C^{1,2+}$-functions whose derivatives are all bounded, see Theorem
\ref{thm:extendA}.
The price to pay is that, if $Z$ is the canonical process,  $f(\cdot, Z)$ is not necessary
a semimartingale any more.
Finally, Corollary \ref{cr:Dynamics} allows to
identify $Z$ as a solution (under $\P)$
of a genuine SDE.  

\subsection{The rough  martingale problem}

\label{ssc:domain}

In this section we assume the validity of Assumption \ref{ass:beta}, and we recall the notion of {\em rough martingale problem} (for shortness rMP) which is used to give a meaning to SDEs with distributional drifts of the form
\begin{equation}\label{eq:dSDE}
	\left\{
	\begin{array}{l}
		dX_t=B(t,X_t)dt+dW_t\\
		X _0\sim \nu,
	\end{array}
	\right.
\end{equation}
where  $B\in C_T\shc^{(-\beta)+}(\R^d; \R^d) $ and  $\nu\in\shp(\R^d)$ is given. 
From now on throughout the paper $ \shl$ will denote the PDE operator
\eqref{eq:shl}.
As defined in Section 2 of \cite{issoglio_russoMPb}, for any $\gamma \in \R$, we denote by $\bar{\shc}_c^\gamma = \bar{\shc}_c^\gamma (\R^d)$ the closure of the space of compactly supported functions
$\shc_c^\gamma$ in $\shc^\gamma$, defined explicitly as

\[\bar{\shc}_c^\gamma := \{f \in \shc^\gamma : \exists (f_n) \subset \shc_c^\gamma \text{ such that } f_n \to f \text{ in } \shc^\gamma\}.
\]
We remark that those are separable Banach spaces, which property is used in particular
in Lemma \ref{lm:DLb}.

Analogously to \eqref{eq:indBesov},
we denote the corresponding inductive spaces  
\[
\bar {\mathcal C}_c^{\gamma+}:= \bigcup_{\alpha >\gamma} \bar{\mathcal C}_c^{\alpha} .
\]
The domain of the rough martingale problem introduced in (3.1) of \cite{issoglio_russoMPb}
is
\begin{notation}\label{def:D}
	\begin{equation*}
		\shd_\shl:=\{f\in C_T\shc^{(1+\beta)+}: \exists g\in C_T\bar\shc_c^{0+} \text{ s.t. $f$ is a solution of } \mathcal{L} f=g \text{ and }f(T)\in\bar \shc_c^{(1+\beta)+} \},
	\end{equation*}
\end{notation}

\begin{definition}[rough martingale problem]\label{def:rMP}
	Let $\Omega_T^d$ be the canonical space  $C([0,T];\R^d)$ equipped with
	its Borel $\sigma$-field $\shf_T$ and $Z$ be the canonical process.  $\P\in\shp(\Omega_T)$ is a solution to the rough martingale problem with distributional drift $B$, initial condition
	$\nu$ 
	if and only if for every $f\in\shd_\shl$ 
$$f(t,Z_t)-\int_0^t\mathcal{L}f(s,Z_s)ds$$
	is a local martingale under $\P$, and $\P\circ Z_0^{-1}=\nu$, i.e. $\mathcal{L}aw(Z_0)=\nu$. In the sequel of this section we will denote  the rough martingale problem with distributional drift $B$ and initial condition
	$\nu$ by  rMP$(B,\nu)$.
      \end{definition}

      


  %

\subsection{About the rough  martingale problems with different domain}

\label{ssc:equivdoms}

In the literature, see  \cite{issoglio_et.al24, issoglio_russoMK, issoglio_russoMPb}, appear  different domains of test functions present characterising a martingale problem.
We will see that the unique solution of the rough martingale problem rMP$(B,\nu)$ provided by Proposition \ref{pr:rMP}
will be the same also the unique solution if we replace  $\shd_\shl$ with one of the domains listed below.

\begin{definition}\label{def:rMPD}
	Let $\Omega_T^d$ be the canonical space  $C([0,T];\R^d)$ equipped with
	its Borel $\sigma$-field $\shf_T$ and $Z$ be the canonical process.  $\P\in\shp(\Omega_T)$ is a solution to the rough martingale problem with distributional drift $B$, initial condition
	$\nu$ and domain $\shd$ if and only if for every $f\in\shd$ 
	$$f(t,Z_t)-f(0,Z_0)-\int_0^t\mathcal{L}f(s,Z_s)ds$$
	is a local martingale under $\P$, and $\P\circ Z_0^{-1}=\nu$, i.e. $\mathcal{L}aw(Z_0)=\nu$. In the sequel of this section we will denote  the rough martingale problem with distributional drift $B$, initial condition
	$\nu$ and domain $\shd$ by  rMP$^{\shd}(B,\nu)$.
\end{definition}
%
%
We emphasise that rMP$(B,\nu)$ coincides with rMP$^{\shd_\shl}(B,\nu)$.
Below we introduce the list of domains.
\begin{notation}\label{def:domains}
  \begin{itemize}
  \item[1.] 
\begin{equation}\label{def:Dtilde}
    \tilde \shd_\shl:=\{f\in C_T\shc^{(1+\beta)+}: \exists g\in C_T\shc^{0+} \text{ s.t. $f$ is a solution of }
    \mathcal{L} f=g \text{ and }f(T)\in \shc^{(1+\beta)+} \}.
    \end{equation}
\item[2.]	\begin{equation}\label{def:DS}
    \shd_\shl^\shs:=\{f\in C_T\shc^{(1+\beta)+}: \exists g\in C_T\shs \text{ s.t. $f$ is a  solution of}
    \mathcal{L} f=g \text{ and }f(T)\in\shs \}
	\end{equation}
	\item[3.]
	\begin{equation}\label{def:D_0}
		\shd_0:=\{f\in C_T\shc^{(1+\beta)+}: \exists g\in C_T\shs \text{ s.t. $f$ is a solution of } \mathcal{L} f=g \text{ and }f(T)=0 \}.
              \end{equation}
            \item[4.]
\begin{equation} \label{eq:domain-hat}
	\hat\shd_\shl:=\{f\in C_TD\shc^{\beta+} \vert \exists g\in C_T\shc^{0+} \text{ s.t. } \shl f=g\}.
\end{equation}
            \end{itemize}

            \end{notation}

We continue some basic relations.

\begin{remark}\label{pr:dom_red}
	\begin{enumerate}
        \item The inclusion
          \begin{equation*}
	\shd_0 \subseteq \shd_\shl^\shs \subseteq \shd_\shl \subseteq {\hat \shd}_\shl
\end{equation*}
holds. In particular a solution to  rMP$^{{\hat \shd}_\shl}(B,\nu)$
(that we will also call {\it linear growth martingale problem}
is a solution to  rMP$(B,\nu)$.
 		\item Suppose $B \in C_T\shc^{0+}$. According to Lemma 4.1 of \cite{issoglio_russoMPb}, 
		the martingale problem   rMP$^{\shd_\shl}(B,\nu)$ is equivalent to the Stroock-Varadhan martingale problem.
		\item Since $\shd_0 \subset \shd_\shl$, a   solution to rMP$^{\shd_\shl}(B,\nu)$ is a solution to  rMP$^{\shd_0}(B,\nu)$.
		Nevertheless the converse may
		not be true a priori. Indeed, let us consider the case
		when $B$ is bounded and smooth.
              A smooth compactly supported function $f:\R^d \rightarrow \R$ 
              which does not depend on time. It  belongs to $\shd_\shl$,
              but it cannot be approximated
t 		uniformly on compact sets with solutions $f$ of \eqref{eq:KolmPDE} such that $f(T) = 0$.
   \item
    Nevertheless, by the proof of Theorem 5.11
                of \cite{issoglio_et.al24}
                  the rMP$^{\shd_0}(B,\nu)$
                admits at most one solution.
                Taking into account Proposition
                \ref{pr:rMP} and the fact that
                $\shd_0 \subset \shd_\shl$
                we know (a posteriori) that
            rMP$^{\shd_0}(B,\nu)$ admits existence and uniqueness.    
	\end{enumerate}
\end{remark}

\begin{proposition}\label{pr:equivMP} Let $\nu \in \shp(\R^d)$.
	The martingale problems rMP$^{\shd_\shl}(B,\nu)$ and rMP$^{\shd_\shl^\shs}(B,\nu)$ are equivalent.
\end{proposition}
\begin{proof} 
	Since $\shd_\shl^\shs\subseteq\shd_\shl$, it is trivial that
	a solution $\P$ of rMP$^{\shd_\shl}(B,\nu)$ is also a solution of  rMP$^{\shd_\shl^\shs}(B,\nu)$.
	Vice versa, let us suppose that $\P$ solves rMP$^{\shd_\shl^\shs}(B,\nu)$ and let $f \in \shd_\shl$
	with $f(T,\cdot) = f_T \in {\bar \shc}_c^{(1+\beta)+}$ and
	$g = \shl f \in C_T{\bar \shc}_c^{0+}.$
	We need to prove that
	$$ M^f_t := f(t,Z_t)-f(0,Z_0)-\int_0^t g (s,Z_s)ds,\  t\in [0,T],$$
	is a local martingale under $\P$.
	
	Since $C_T\shs$ is dense in $C_T\bar\shc_c^{\g+} $  for every $\gamma \in \R$, by Corollary 5.8 in \cite{issoglio_russoPDEa} then
	$\shs$ is also dense in $\bar\shc_c^{\g+} $.
	Let  $(g^n)$ (resp.  $(f_T^n)$) be a sequence in  $ C_T\shs\subseteq C_T\bar\shc_c^{0+}$
	(resp. $\shs\subseteq\bar \shc_c^{(1+\beta)+}$) converging
	in $  C_T\bar\shc_c^{0+}$ (resp. in $  \bar \shc_c^{(1+\beta)+}$)
	to $g$ (resp. $f_T$). Let $f^n$ be the solutions
	of $\shl f^n = g^n$ with $f^n(T,\cdot) = f^n_T$.
        From the continuity results on the Kolmogorov PDE, see
        Corollary \ref{cor:unDCBis} and Remark \ref{rmk:unDCBis} (with $B^n = B$),
 we have that $f^n\to f$
	in  $  C_T\shc^{(1+\beta) +}$.
	Since each $f^n$ belongs to $\shd_\shl^\shs$  then
	$M^{f^n} = f^n(\cdot,Z)-f^n(0,Z_0)-\int_0^\cdot g^n (s,Z_s)ds$
	is a local martingale.
	But  $f^n \rightarrow f$ and $g^n \rightarrow g$ uniformly so that
	$M^{f^n}$ converge u.c.p. to $M^f$.
	Since the linear space of the local martingales equipped with the
	metric topology related to the u.c.p. convergence is closed, 
	then $M^f$ is a local martingale and the result is established.
      \end{proof}

\subsection{The linear growth rough martingale problem}

\label{sec:LGrMP}


The martingale problem
rMP$^{\hat\shd_\shl}(B,\nu)$,
where $\hat\shd_\shl$ was defined in \eqref{eq:domain-hat}
is a domain possibly unbounded functions of Besov-H\"older.

\begin{theorem}\label{thm:rMPext}
  Let $\P$ be a solution to rMP$(B,\nu)$.
Then $\P$ is a solution to
  rMP$^{\hat\shd_\shl}(B,\nu)$.
Moreover,  for all  $f\in\hat\shd_\shl$
     %
we have that
        \begin{equation}\label{eq:rMPext2}
		M_t^{f^2}:=
		f^2(t,Z_t)
		-\int_0^t (2fg+|\nabla f|^2)(r,Z_r)dr
	\end{equation}
	is a local martingale under $\P$ 
      so that
        \begin{equation} \label{eq:MfBracket}
          [M^f] = \int_0^\cdot |\nabla f|^2(r,Z_r)dr.
          \end{equation}
\end{theorem}
\begin{corollary}\label{cor:rMPext}
Let $\P$ be a solution to rMP$(B,\nu)$. Let $f\in C_TD\shc^{\beta+}$
such that $\shl f = 0$.
Then 
		$f(t,Z_t)$
                is a local martingale under $\P$ with quadratic variation given by
                $\int_0^t |\nabla f|^2(r,Z_r)dr$.
\end{corollary}

\begin{proof}[Proof of Theorem \ref{thm:rMPext}.]

  Let $f\in \hat\shd_\shl$ and let us denote $g:= \shl f$.
  Let us also consider
  \begin{equation} \label{eq:phi_lambda}
    \phi(t,x):=\phi_{\lam}(t,x), \ \ell=\lam(\phi_\lam-id).
\end{equation}
    By Remark \ref{rmk:philam}
$\phi$ satisfies Assumption \ref{ass:phiell}.
   We define $\tilde g:=g\circ \psi$ and $\tilde f_T=f_T\circ \psi_\lam$, where $\psi(t,\cdot)$ 
	is the inverse function of  $\phi(t,\cdot)$, for all $t\in[0,T]$. We notice that $\tilde f_T$ is differentiable with 
	$$\nabla \tilde f_T(x)=\nabla f_T(\psi(T,x))\nabla\psi(T,x)$$
	of class $\shc^{\beta+}$, thus bounded.
	Now, using the heat kernel $p_{\frac1n}$ defined in \eqref{eq:heatkern}, we regularise $\tilde f_T $ to get
	\begin{equation}\label{eq:fnp1n}
		\tilde f^n_T(x):=\tilde f_T*p_{\frac1n}(x).
	\end{equation}
	We observe that, for every $n\in\N$,  $(\tilde f_T^n)$ are possibly unbounded functions.
        Nevertheless the  derivatives of every order are bounded: in fact, this comes from knowing that
        $\tilde f_T\in D\shc^{\beta+}$ together with $p_{\frac1n}\in\shs$, which means that $p_\frac1n$ has integrable derivatives of every order.

	 We now define $\tilde\shl$ as 
         $$\tilde \shl \varphi := \partial_t \varphi+\frac12 {\rm Tr}
         \left[(\nabla \phi \circ \psi)^\top{\rm Hess}\varphi (\nabla \phi \circ \psi)\right]+ \nabla \varphi \cdot  \ell,$$
	and $\bar g^n$ as
	$$\bar g^n:=\tilde g-\frac12 {\rm Tr}
        \left[(\nabla \phi \circ \psi)^\top{\rm Hess}f_T^n ( \nabla \phi \circ \psi)\right]- \nabla f_T^n \cdot \ell.$$
        Notice that both the coefficients of operator $\tilde\shl$ and $\bar g^n$ have bounded continuous
        coefficients: indeed $\ell\in C_T\shc^{0+}(\R^d;\R^d)$ by Assumption \ref{ass:phiell}
        and $\nabla \phi\circ	\psi$ is bounded and continuous, since $\nabla \phi\in\shc^{\beta+}(\R^d;\R^d)$, while for $\bar g^n$, it is in $C_T\shc^{0+}$ because it is a sum of terms in $C_T\shc^{0+}$. By using a time reversal argument we can apply Theorem 5.1.9. from \cite{lunardi95}, to get that there exists a unique solution $h^n\in C_b^{1,2+}([0,T]\times \R^d)$ to 
	 	\begin{equation}\label{eq:hn}
	 			\tilde\shl h^n=\bar g^n, \quad h^n_T=0.
	 		\end{equation}
	At this point, we define $\tilde f^n:=h^n+\tilde f_T^n\in C_{ub}^{1,2+}(\R^d)$, since $h^n\in C_b^{1,2+}([0,T]\times \R^d)$ and
        $\tilde f_T^n\in C_{ub}^{2+}(\R^d)$, therefore we have $\tilde f^n\in
        C_{ub}^{1,2+} ([0,T]\times \R^d)\subseteq C_{ubuc}^{1,2}([0,T]\times \R^d)$, see e.g. Section \ref{sc:basic-prelim}. Moreover, by linearity of $\tilde\shl$, one gets that $\tilde f^n$ solves $\tilde \shl\tilde f^n=\tilde g, \tilde f^n(T,\cdot)=\tilde f^n_T$.   
 
	We define now $Y_t:=\phi(t,Z_t)$.
        For any $\tilde f \in C_c^\infty $ we define $f:= \tilde f \circ
        \phi $, where $\phi$
was defined in \eqref{eq:phi_lambda}.
 By Lemma 3.7 of \cite{issoglio_russoMPb}, we know that
         $f \in \mathcal D_{\mathcal L}$. 
Since $\P$ is  a solution of rMP($\shd_\shl,\nu)$,
  we have that
\[
\tilde f (Y_t) -\tilde f (Y_0) - \int_0^t (\tilde{\mathcal L} \tilde f) (s, Y_s) \mathrm ds =  f(t, X_t) - f(0,X_0) - \int_0^t ( \mathcal L f)(s, X_s) \mathrm ds , 
\]
which is a $\P$-local martingale.
 By the same arguments as the proof of Proposition 4.6, Chapter 5
        of \cite{karatzasShreve},
        it is not difficult to show
        that 
        $[Y, Y] = \int_0^\cdot (\nabla \phi)^\top(\nabla \phi)(r,\psi(r,Y_r)) dr$.
  Since previous matrix covariation is non-degenerate,
  under $\P$, there is   a standard $\P$-$d$-dimensional Brownian motion
$W$,  such that $Y$  is a solution to
	\[
	Y_t=Y_0+\int_0^t \ell(t,Y_r)dr+\int_0^t\nabla \phi(r,\psi(r,Y_t))dW_r.
      \]
      We then apply
	 It\^o's formula 
	 to  $(\tilde f^n(t,Y_t))$ 
	 and $((\tilde f^n)^2(t,Y_t))$
	  to obtain
	$$\tilde f^n(t,Y_t)=\tilde f^n(0,Y_0)+\int_0^t  \tilde g(r,Y_r)dr+\int_0^t\nabla \tilde f^n(r,Y_r)\nabla \phi(r,\psi(r,Y_t))dW_r$$
	and
	$$(\tilde f^n)^2(t,Y_t)=(\tilde f^n)^2(0,Y_0)+\int_0^t (2\tilde f^n\tilde g(r,Y_r)+|\nabla \tilde f^n|^2)(r,Y_r)dr+2\int_0^t (\tilde f^n\nabla \tilde f^n) (r,Y_r)\nabla \phi(r,\psi(r,Y_t))dW_r  .$$ 
Thus, we have that 
	
$$\tilde f^n(t,Y_t) -\tilde f^n(0,Y_0) -\int_0^t  \tilde g(r,Y_r)dr =
\int_0^t \nabla \tilde f^n(r,Y_r)dW_r$$
	and
	$$ ({\tilde f}^n)^2(t,Y_t)  - ({\tilde f}^n)^2(0,Y_0) -
        \int_0^t (2 \tilde f^n \tilde g+|\nabla \tilde f^n|^2)(r,Y_r)dr 
        = 2\int_0^t(\tilde f^n \nabla \tilde f^n)(r,Y_r)\nabla \phi(r,\psi(r,Y_t))dW_r$$
	are  local martingales under $\P$. Defining  $f^n:=\tilde f^n\circ \phi$ together with the definition of  $Y$ and $\tilde g$, the processes
	\begin{equation}\label{eq:MPfn}
          f^n(t,Z_t)-f^n(0,Z_0) - \int_0^t   g(r,Z_r)dr =\int_0^t \nabla  f^n(r,Z_r)dW_r
	\end{equation}
	and 
	\begin{equation}\label{eq:MPfn2}
    (f^n)^2(t,Z_t) - (f^n)^2(0,Z_0)-\int_0^t  (2f^ng+|\nabla  f^n|^2)(r,Z_r)dr =  \int_0^t( f^n \nabla  f^n)(r,Z_r)dW_r
\end{equation}
	are 
 local martingales under $\P$. Applying Lemma \ref{lm:gen3.7} we obtain
	$$g=\tilde g\circ \phi=(\tilde \shl\tilde f^n)\circ \phi=\shl f^n,$$
	thus $(f^n)$ are solutions to $\shl f^n=g,\ f^n(T,\cdot)=f^n_T$.
        By \eqref{eq:schaud2DC} in Lemma \ref{lm:schauderDC} we can show that
        $f^n_T\to f_T$ in $D\shc^{\beta+}$.
Consequently, by Corollary \ref{cor:unDCBis} with $B^n = B$,
we have that $f^n\to f$ in $C_TD\shc^{\beta+}$, the solution to $\shl f=g,\ f(T)=f_T$.
Taking   the limit as $n$ approaches infinity  in both \eqref{eq:MPfn} 
	and  \eqref{eq:MPfn2} 
	we get that
	\begin{equation}\label{eq:MPf}
		M_t^f:=f(t,Z_t)  -\int_0^t  g(r,Z_r)dr
	\end{equation}
	and 
	\begin{equation}\label{eq:rMPf2}
		M_t^{f^2}:=f^2(t,Z_t)-2\int_0^t fg(r,Z_r)dr-\int_0^t |\nabla f|^2(r,Z_r)dr
	\end{equation}
	are local martingale are under $\P$, because the linear space of local martingales vanishing at zero is closed  under the u.c.p. topology.
            \eqref{eq:MfBracket} follows applying integration by parts
	to the semimartingale $(f^2(t,Z_t))$, \eqref{eq:rMPf2}, taking into account
        the uniqueness of the decomposition of a semimartingale.
\end{proof}


\begin{remark} \label{rmk:denALPHA}
		\begin{enumerate}
                \item
                  The extension of the domain of test functions
                  discussed in the theorem above allows to 
                  avoid the closure of compactly supported functions
                  characterizing the domain $\shd_\shl$, see Notation
                  \ref{def:D}.
                \item  For $\alpha\in (0,1)$ it holds that
                  the closure of $\shc^{1+\al}$
                into ${D\shc^\al}$ is not the full space
              ${D\shc^\al}$ as we explain below.
            If it were the case, by making use of the domain
            $\tilde \shd_\shl$, see \eqref{def:Dtilde},
            the proof of Theorem \ref{thm:rMPext} would be an
            easy consequence of density arguments.
            
                  Indeed,  consider the function $f:\R^d\to\R$ defined by $f(x)=x_1$. This function is such that $f\in D\shc^{\al}$ and its gradient is $\nabla f(x)=(1,0,\dots,0)$.
		  Suppose there exists a sequence $(f_n)\subseteq\shc^{1+\al}$ such that
                  $f_n(0)\to f(0)$ and 
		  \begin{equation}\label{eq:unconvDC}
			\|f_n-f\|_{D\shc^{\al}}\to 0.
		  \end{equation}
		 Let us fix $\eps=\frac12$ and consider $n_0\in\N$ such that
		\[
		|f_{n_0}(0)| < \frac{1}{4}, \quad \|\nabla f_{n_0}-(1,0,\dots,0)\|_{\shc^{\al}}< \frac{1}{2}.
		\]
		We can write
		\[
		f_{n_0}(x)=f_{n_0}(0)+x\cdot\int_0^1 \nabla f_{n_0}(ax)da, \ \forall x\in\R^d.
		\]
		By taking the absolute value and using the inverse  triangle inequality we get
		\begin{align*}
			|f_{n_0}(x)|=&\left|f_{n_0}(0)+x\cdot\int_0^1 \nabla f_{n_0}(ax)da\right|\\
			\geq& \left|x\cdot\int_0^1 \nabla f_{n_0}(ax)da\right|-|f_{n_0}(0)|\\
			\geq& \frac{|x|}2-\frac14 ,\ \forall x\in\R^d.
		\end{align*}
		Thus $f_{n_0}$ cannot be bounded. 
		\item Previous considerations show in particular that the component of the identity function  $id:\R^d\to\R^d$, which belongs to $D\shc^\be$, cannot be approximated by functions in $C_T\shc^{1+\beta}$. 
	\end{enumerate}
\end{remark}	

\begin{remark}
From Theorem \ref{thm:rMPext} we also  have $[M^f]_t=\int_0^t |\nabla f|^2(r,Z_r)dr\leq \|\nabla f\|_{\infty}^2$, then $M^f_t-f(0,Z_0)$ is a genuine  martingale, i.e. $M_t^f-f(0,Z_0)\in L^p$ for all $p\ge 1$.
\end{remark}

\subsection{Extending the local time operator}

\label{sec:ext-op}

One natural question consists in identifying the {\it stochastic differential
  equation} solved by the canonical process $Z$ under the solution $\P$ of the
rMP($B,\nu$). In other words one needs somehow to evaluate
$f(t, Z_t)$ for $f=id$, the 
 identity function
$id(x):=x, x \in \R^d$. 
The problem is that the components  $id^i$ of $id$ do not belong to $\hat \shd_\shl$.
This section clarifies and extends Section 6. of
\cite{issoglio_russoMPb}.
We start by recalling the notion of {\em weak Dirichlet process}.

\begin{definition}\label{def:wDpr}
	Let $X$ be a continuous stochastic process on some probability space $(\Omega,\shf,P)$ and let $\shf^X$ be the canonical filtration.
	\begin{itemize}
        \item[a)] 	A process $\sha$ is said to be an $\shf^X$-martingale orthogonal process
(or simply martingale orthogonal process)
          if $[\sha,N]=0$ for every $\shf^X$-continuous local martingale $N$.
		\item[b)] The process $X$ is said to be $\shf^X$-weak Dirichlet
(or simply weak Dirichlet)
                  if it is the sum of an $\shf^X$-local martingale $M$ and an $\shf^X$-martingale orthogonal
		process $\sha$  such that $\sha_0=0$.
	\end{itemize}
	
\end{definition}

\begin{remark} \label{rmk:bandini-uniq}
  We point out that such decomposition is unique for (even càdlàg) weak Dirichlet processes, see
  Proposition 3.2 of \cite{BandiniRusso_RevisedWeakDir}.
\end{remark}

We need a tool which is a direct implication of Propositions 5.3 and 5.11 and Corollary 5.9 of \cite{issoglio_russoMPb}.
\begin{lemma} \label{lem:dec-op}
  Let $\P$ be a solution of rMP($B,\nu$) and $Z$ be the canonical process.
  There are maps $f \mapsto M^f$, $f \mapsto \sha(f)$, 
  from $C^{0,1}([0,T]\times \R^d)$ to $\shc:=\shc^{\P}$  such that the following holds.
  \begin{enumerate}
  \item For every  $f\in C^{0,1}([0,T]\times \R^d)$,
    $f(t, Z_t)$ is an $\mathcal F^Z$-weak Dirichlet process
 with canonical decomposition $M^f_t + \sha_t(f), t \in [0,T]$. 
 In particular, for every $1 \le i \le d$,  $Z^i$ is an $\mathcal F^Z$-weak Dirichlet process
 with decomposition $M^{id^i} + \sha(id^i)$.
\item 
The map $f\mapsto \sha(f)$ is continuous (and linear) from $C^{0,1}$ to $\shc$.
\item Denoting $M^{id} := (M^{id^1},\ldots, M^{id^d})^\top$,
 for $f\in C^{0,1}([0,T] \times \R^d)$ we have 
\[
M^f_t = f(0, Z_0) + \int_0^t \nabla f(s, Z_s) \cdot  d M^{id}_s.
\]
\end{enumerate}
    \end{lemma}

  
Let us define now one more domain, which plays a central role for the next results,
in relation to the solution of the rMP($B,\nu$).
\begin{notation} \label{def:extendA}
	$\shd_\shl^{-\beta}:=\{f\in C_TD\shc^{\beta+}|\exists g\in C_T\shc^{(-\beta)+}\text{ s.t. }\shl f=g\}$.
\end{notation}
The proposition below permits the extension of the local time operator to
the space of distributions $C_T\shc^{(-\beta)+}$.

\begin{proposition}\label{pr:extendA}
	Let  $\P$ be a solution to the rMP($B,\nu$). Then the map  
\begin{equation} \label{eq:opA}
  A:C_T\shc^{0+}(\R^d)\ni g\mapsto \int_0^\cdot g(r,Z_r)dr \in\shc
  \end{equation}
	 admits a unique  continuous extension to  $ C_T\shc^{(-\beta)+}\ni g\mapsto A^{\P}(g) \in \shc^\P$. 
	 Moreover  $\forall f\in\shd_\shl^{-\beta}$
		\begin{equation}\label{eq:extend}
		M^f:= f(t,Z_t)-A^{\P}_t(\shl f), t \in [0,T]
		\end{equation}
		is a local martingale under $\P$ with quadratic variation $\int_0^\cdot |\nabla f|^2(r,Z_r)dr$.
              \end{proposition}
              \begin{remark} \label{rmk:density}
                We remark that, $\shc^{0+}$ is dense in $\shc^{(-\beta)+}$, because of
\eqref{eq:schauder2} in
Lemma \ref{lm:schauder}.

                \end{remark}
\begin{corollary} \label{cor:extendA}
Let  $\P$ be a solution to the rMP($B,\nu$).
For all $ f^1, f^2 \in\shd_\shl^{-\beta}$, for $t \in [0,T]$,  we have
		\begin{equation}\label{eq:extendCor}
		[f^1(\cdot,Z_\cdot), f^2(\cdot,Z_\cdot) ]_t =
     \int_0^t \langle \nabla f^1, \nabla f^2\rangle (r,Z_r)  dr, t \in [0,T].
     \end{equation}           
   \end{corollary}
  \begin{proof} 
    We apply Proposition \ref{pr:extendA} to $f = f^1, f^2, f^1 + f^2$
    and we concluding by the fact that
    $$ [f^1(\cdot,Z_\cdot), f^2(\cdot,Z_\cdot) ]_t =
    \frac{1}{2} [(f^1+ f^2)(\cdot,Z_\cdot), (f^1 + f^2)(\cdot,Z_\cdot) ]_t
    - \frac{1}{2} [f^1(\cdot,Z_\cdot), f^1(\cdot,Z_\cdot) ]_t
 - \frac{1}{2} [f^2(\cdot,Z_\cdot), f^2(\cdot,Z_\cdot) ]_t.
  $$  
\end{proof}
\begin{remark} \label{rmk:extendA}
  The operator $A^{\P}$ fulfils the following compatibility property.
  Let $t \in (0,T)$.
  For every $g, \tilde g \in C_T\shc^{(-\beta)+}$ such that
  their restrictions to $[0,t]$ coincide,
  we have
\begin{equation} \label{eq:AZgtilde}
  A^{\P}_s(g) = A^{\P}_s(\tilde g), s \in [0,t].
  \end{equation}

  Indeed, obviously \eqref{eq:AZgtilde} holds
  if $g, \tilde g \in  C_T\shc^{0+}$.
  Now we define
  $ g_n:= P_{\frac{1}{n}} g$ and $ \tilde g_n:= P_{\frac{1}{n}} \tilde g,$
  whose restrictions to $[0,t]$ coincide, i.e.
  $ A^{\P}_s(g_n) = A^{\P}_s(\tilde g_n), s \in [0,t].$
    By Lemma \ref{lm:Xncont}, $g_n, \tilde g_n$ belong to
  $C_T\shc^{0+}$ and they converge respectively to  $g, \tilde g$ 
  in $C_T\shc^{(-\beta)+}$.
  Since $A^{\P}$ is continuous with respect to  $C_T\shc^{(-\beta)+}$
   \eqref{eq:AZgtilde} holds.
\end{remark}

\begin{proof}[Proof of Proposition \ref{pr:extendA}]

  Let $v \in \shd_\shl^{-\beta}$ be the unique solution of $\shl v = g, v(T) = 0$.
  We set $\hat A(g) := \sha(v)$,
  where $\sha$ was defined in Lemma \ref{lem:dec-op}.
  The  map $\hat A$ is continuous.
    Indeed  the map $g \mapsto v$ is continuous from
    $C_T\shc^{(-\beta)+}$ to $ C_T\shc^{1+\beta}$
    by Lemma 4.17 and Remark 4.18 of \cite{issoglio_russoPDEa},
    and consequently from  $C_T\shc^{(-\beta)+}$ to $ C^{0,1}([0,T] \times \R^d).$
    Lemma \ref{lem:dec-op} states that $\sha$ is a continuous operator from  $ C^{0,1}([0,T] \times \R^d)$ taking values in  $\shc$ so that
    $\hat A$ is continuous
    by composition of continuous operators.

    We prove below that $\hat A = A^\P$.
   We fix for the moment $g \in C_T\shc^{0+}(\R^d)$.
Since $\hat A$ is continuous it is enough to prove 
          \begin{equation} \label{eq:compat}
        \hat A(g) = \int_0^\cdot
        g(r,Z_r) dr.
        \end{equation}
      Let $f\in C_T\shc^{(1+\beta)+}$ solution of
      $\shl f = g, f(T)=0$.
      Since $f \in C_T\shc^{(1+\beta)+}$ it belongs to $C^{0,1}([0,T]\times \R^d)$, so
      by Lemma \ref{lem:dec-op} 
      $f(t,Z_t)$ is a weak Dirichlet process
      with unique decomposition
\begin{equation} \label{eq:DirHat}
  f(t,Z_t) = M^f_t + \sha(f)  = M^f_t + \hat A(g), t \in [0,T].
  \end{equation}
  Since $\P$ is a solution of rMP$(B,\nu)$,
  Theorem \ref{thm:rMPext} states that $\P$
  is also a solution to rMP$^{\hat \shd_\shl}(B,\nu)$,
 therefore     there is a local martingale $\tilde M^f$ such that
\begin{equation} \label{eq:DirMP}
  f(t,Z_t) = \tilde M^f_t + \int_0^t g(r,Z_r) dr, t \in [0,T].
  \end{equation}
  Identifying \eqref{eq:DirHat} and \eqref{eq:DirMP}, 
by uniqueness of the decomposition of weak Dirichlet processes,
we deduce \eqref{eq:compat} and therefore $\hat A = A^\P$.
   
    It remains to prove the statement related to \eqref{eq:extend}.
Let  $f  \in\shd_\shl^{-\beta}$ and set 
  $g:= \shl f$. We can decompose $f : = v_f + w_f$, where $v_f$ (resp. $w_f$) is  the unique solution to
  $$\begin{cases}\shl v_f = g\\ v_f(T) = 0\end{cases} \text{ resp. } \begin{cases}\shl w_f = 0\\ w_f= f(T).
  \end{cases}$$
  By \eqref{eq:DirHat} and the fact that $\hat A = A^\P$,
  the process
  \begin{equation} \label{eq:vf}
v_f(t,Z_t)-A^{\P}_t(\shl v_f) = v_f(t,Z_t)-A^{\P}_t(g) 
		\end{equation}
                is a local martingale under $\P$.
By Corollary \ref{cor:rMPext} $w_f(t,Z_t)$
                is a local martingale.
      Finally  
$$ f(t, Z_t)- A_t^{\P}(g) = v_f(t,Z_t)-A^{\P}_t(g) + w_f(t,Z_t),$$
         is a local martingale under $\P$.

       We prove now that for all $f\in \shd^{-\beta}_\shl$ the quantity
\begin{equation} \label{eq:MfBrack}
  [M^f] = \int_0^\cdot |\nabla f|^2(r,Z_r)dr.
\end{equation}
  Let $g := \shl f$. We define 
         $g^n:=P_{\frac{1}{n}}g$ and $f^n$ is the unique solution to
         $\shl f^n=g^n, f^n(T)=f(T)$.
         By  \eqref{eq:schauder2} in Lemma \ref{lm:schauder}
$g^n$ converges to $g$ in $C_T\shc^{(-\beta)+}$. 
Since $f^n\in\hat\shd_\shl$ for every $n\in\N$,
by Theorem \ref{thm:rMPext} we have that
$$ M^{f^n}_t = f^n(t,Z_t)-\int_0^t g^n(r,Z_r)dr  =
f^n(t,Z_t)- A^\P(g^n),$$
is a local martingale under $\P$ with quadratic variation given
by
\begin{equation} \label{eq:MfnBrack}
  [M^{f^n}] = \int_0^\cdot |\nabla f^n|^2(r,Z_r)dr,
  \end{equation}
By Corollary \ref{cor:unDCBis} with $B^n = B$,
$ f^n\to  f$ in $C_T\shc^{(1+\beta)+}$ hence in $C^{0,1}([0,T]\times\R^d)$.
Since $A^{P}$ is continuous,
$M^{f^n}$ converges u.c.p. to $M^f$ defined in
\eqref{eq:extend} which is a local martingale.
By Proposition A.4 of \cite{BandiniRusso_RevisedWeakDir}
and \eqref{eq:MfnBrack}
$$ [M^f] = \lim_{n \rightarrow \infty} [M^{f^n}] =  \int_0^\cdot |\nabla f|^2(r,Z_r)dr,$$
u.c.p., which concludes the proof of \eqref{eq:MfBrack}.

\end{proof}

\subsection{Different equivalent formulations
  of the rough martingale problem}

The key tool in this section is Proposition \ref{pr:extendA}
which, given  a solution $\P$ to rMP($B,\nu$),
one can properly extend the
local time operator.

\begin{theorem}\label{thm:extendA}
		Let $\P$ be a probability on the canonical space. Then the following are equivalent.
		\begin{enumerate}
                \item $\P$ is a solution to the rMP($B,\nu$).
			  \item The map  $A:C_T\shc^{0+}\to  \shc$
        admits a continuous extension $A^\P$ to the space  $ C_T\shc^{(-\beta)+}$ and $\forall f\in\shd_\shl^{-\beta}$
		 \[
		M^f_t:= f(t,Z_t)-A^{\P}_t(\shl f)
		 \] 
		 is a local martingale under $\P$ with quadratic variation $\int_0^t |\nabla f|^2(r,Z_r)dr$.
               \item The map  $A:C_T\shc^{0+}\to  \shc$ admits a continuous extension
              $A^\P$   to the space  $ C_T\shc^{(-\beta)+}$ and $\forall f\in C_{ub}^{1,2+}$
		 	\[
		 M^f_t:=	f(t,Z_t)- A^{\P}_t(\shl f)
		 	\] 
		 	is a local martingale under $\P$ with quadratic variation $\int_0^t |\nabla f|^2(r,Z_r)dr$.
             \item The map  $A:C_T\shc^{0+}\to  \shc$ admits a continuous extension
            $A^\P$     to the space  $ C_T\shc^{(-\beta)+}$ and $\forall f\in C_{b}^{1,2+}$
		 	\[
		 M^f_t:=	f(t,Z_t)- A^{\P}_t(\shl f)
		 	\] 
		 	is a  martingale under $\P$ with quadratic variation $\int_0^t |\nabla f|^2(r,Z_r)dr$.		 
			\end{enumerate}
\end{theorem}

\begin{proof}
  {\em $(1. \rightarrow 2.)$ } is
  a direct consequence of Proposition  \ref{pr:extendA}.
  

  {\em $(2. \rightarrow 3.)$} This follows immediately
  since $C_{ub}^{1,2+}\subseteq \shd_\shl^{-\beta}$. Indeed for every $f\in C_{ub}^{1,2+}$ we can evaluate $\shl f$, which belongs to
  $C_T\shc^{(-\beta)+}$:
  in particular $\partial_t f $,
  $\Delta f $ belong to $C_T\shc^{0+}$ and the pointwise product
  $b\cdot \nabla f $ belongs to $C_T\shc^{(-\beta)+}$.
 
{\em $(3. \rightarrow 4.)$ }
This follows directly from
  $C_{b}^{1,2+}\subseteq C_{ub}^{1,2+}$.
  In fact $M^f$ is even a martingale since
  $f(0,Z_0)$ and $[M^f]$ are bounded.

 {\em $(4. \rightarrow 1.)$ }  Let  $f\in C_T\shc^{(1+\beta)+}$ be a solution to 
 $$\shl f=g, \ f(T)=f_T,$$
 where $g\in C_T\shc^{0+}$ and $f_T\in \shc^{(1+\beta)+}$. 
 We need to show that
 \begin{equation} \label{eq:Mart1}
    M^f:= f(t,Z_t)- \int_0^t g(r,Z_r) dr =   f(t,Z_t)-A^{\P}(g),
   \end{equation}
is a local martingale under $\P$.
 Let us now define $B^n:=P_{\frac1n}B\in C_T\shc^{0+}(\R^d;\R^d)$ and $f^n_T=P_{\frac1n}f_T\in C^{2+}_b$.
By  item {\it (iii)} of Lemma \ref{lm:Xncont} $B^n\to B$ 
  in  $C_T\shc^{(-\beta)+}$ thanks to item {\it (ii)} of Lemma \ref{lm:Xncont} and  $f^n_T\to f_T$ in $\shc^{(1+\beta)+}$ thanks to
  \eqref{eq:schauder2} in Lemma  \ref{lm:schauder}.
  
  For each $n$, we consider the operator $\shl^n:=\partial_t+\frac12\Delta+\nabla_\cdot B^n$.
  We consider the solution $f^n \in C_T\shc^{(1+\beta)+}$ to $\shl^nf^n=g, f^n(T) = f^n_T$
  provided by Theorem \ref{thm:PDEU} with $B=B^n$.
  By Corollary \ref{cor:unDCBis} and Remark \ref{rmk:unDCBis}, with
  $u_T^n=f^n_T$ and $u^n=f^n$, we have
  $f^n\to f$ in $C_T\shc^{(1+\beta)+}$, consequently also uniformly
  so that $f^n(t,Z_t) $ converges u.c.p. to $f(t,Z_t)$.
 Furthermore, since $B^n$ is bounded continuous
 and $f^n_T\in C_b^{2+}$, by Proposition \ref{pr:lunardi},  $f^n\in C_b^{1,2+}$,
 for every $n\in\N$. Therefore, by {\em item 4.}, we get  that, for every $n\in\N$,
  $$f^n(t,Z_t)-A^{\P}(\shl f^n)$$
  is a local martingale under $\P$, where
$\shl f^n= g+\nabla f^n\cdot (B-B^n).$
We  observe then,  that $f^n$ is the unique solution to
    \begin{equation} 
  	 	 	\shl f^n= g+\nabla f^n\cdot (B-B^n),\  f^n(T)=f^n_T.
  	 	 \end{equation} 
   Thanks to Corollary \ref{cor:unDC} (i),
 $C:= \sup_n \|\nabla f^n\|_{C_T\shc^{\beta}} $ is finite. For a suitable $\al<\be$,
  by \eqref{eq:bonyt},  we have that 
  \begin{equation}
 	 	\|g^n - g\|_{C_T\shc^{-\al}}
 	 		\leq \|\nabla f^n\|_{C_T\shc^{\beta}}\|B-B^n\|_{C_T\shc^{-\al}}
 	 		\leq C \|B-B^n\|_{C_T\shc^{-\al}}.
 	              \end{equation}
     so that
      $g^n$ converges to $g$ in $C_T\shc^{(-\beta)+}$.
       By the assumption on $A^{\P}$,  $A^{\P}(g^n)$ converges u.c.p. to $A^{\P}(g)$. 
       Finally, by additivity, $f^n(t,Z_t)-A^{\P}(g^n)$ converges u.c.p.
       to \eqref{eq:Mart1} which is forced to be a local martingale since the linear space
  of local martingales vanishing at zero is closed under the u.c.p. convergence. This concludes the proof.
 \end{proof}

An important implication of Theorem \ref{thm:extendA}  
is the following.

\begin{corollary} \label{cor:extendABis}
  rMP$^{\shd_\shl}(B,\nu)$ and  rMP$^{{\tilde \shd}_\shl}(B,\nu)$.
  are equivalent, in the sense that the solutions are the same.
\end{corollary}
\begin{proof}

  The converse implication follows immediately
  since $ \shd_\shl \subset {\tilde \shd}_\shl  $.
Concerning the   direct implication,
 suppose $\P$ fulfils rMP$^{\shd_\shl}(B,\nu)$.
 Then  Theorem \ref{thm:extendA} item 2.,
 guarantees that the map 
 $g \mapsto \int_0^\cdot g(r,Z_r) dr$ defined in $C_T \shc^{0+}$, admits a continuous
 extension $A^{\P}$ to
 $ C_T\shc^{(-\beta)+}$ to $\shc^\P$.
 Moreover, if
  $f \in \shd_\shl^{-\beta}$ (see Notation \ref{def:extendA}) and $\shl f = g$, the process
		 \[
		f(t,Z_t)  = M^f_t + A^{\P}_t(g)
		 \] 
		 is a local martingale.
                 The result follows since ${\tilde \shd}_\shl$ is a subset of
                 $\shd_\shl^{-\beta}$.

\end{proof}

For technical reasons
  we have to introduce a supplementary domain
  in relation to the closure of some
  Besov spaces of functions
  with compact support.
    \begin{notation}  \label{not:closureDom}
    We denote by 
    $$\tilde\shd_\shl^{-\beta}:=\{f \in C_T\shc^{(1+\beta)+}\vert \exists g\in
    C_T\bar\shc_c^{(-\beta)+} \text{ s.t. } \shl f=g, f_T\in\bar\shc_c^{(1+\beta)+}\}.$$
	\end{notation}
    
        Again, for technical reasons related to separability,
        we need to formulate and prove an analogous theorem
        to Theorem \ref{thm:extendA} related to the
        domain introduced in Notation \ref{not:closureDom}.

\begin{theorem}\label{thm:extABis}
	Let $\P$ be a probability measure on the canonical space $\Omega_T^d$. Then the following are equivalent.
	\begin{enumerate}
		\item$\P$ is the solution to the rMP($B,\nu$).
		\item The map $A:C_T\shc_c^{0+}\ni g\to \int_0^t g(r,Z_r)dr\in\shc^{\P}$ admits a unique continuous extension $A^\P$ to $C_T\bar\shc^{(-\beta)+}$ and for every $f\in\tilde\shd_\shl^{-\beta}$,
			\[
		M^f_t:=	f(t,Z_t)- A^{\P}_t(\shl f)
		\] 
		is a  martingale under $\P$ with quadratic variation $\int_0^t |\nabla f|^2(r,Z_r)dr$.		
		\item The map $A:C_T\shc_c^{0+}\ni g\to \int_0^t g(r,Z_r)dr\in\shc^\P$ admits a continuous extension $A^\P$ to $C_T\bar\shc^{(-\beta)+}$ and for every 
		$f\in C_c^{1,2+}$,
		\[
		M^f_t:=	f(t,Z_t)- A^{\P}_t(\shl f)
		\] 
		is a  martingale under $\P$.
              \end{enumerate}
\end{theorem}
\begin{proof}
  $(1.\implies 2.)$ We observe that
  $C_T\shc_c^{0+}$ is dense in $C_T\bar\shc^{(-\beta)+}$.
  The result
  follows directly from
   $(1.\implies 2.)$
 of Theorem \ref{thm:extendA} since
  $ \tilde\shd_\shl^{-\beta} \subset \shd_\shl^{-\beta}$ and
  $ C_T\shc_c^{0+} \subset C_T\shc^{0+}$.
The process $M^f$ is a true martingale since $\nabla f$ and therefore $[M^f]$ is bounded.
  \\
 $(2.\implies 3.)$
This follows immediately
  since $C_{c}^{1,2+}\subseteq \tilde \shd_\shl^{-\beta}$. Indeed for every $f\in C_{c}^{1,2+}$ we can evaluate $\shl f$, which belongs to
  $C_T\shc^{(-\beta)+}$:
  in particular $\partial_t f $,
  $\Delta f $ belong to $C_T\shc_c^{0+}$ and the pointwise product
  $b\cdot \nabla f $ belongs to $C_T {\bar \shc}^{(-\beta)+}$ because of Remark \ref{rmk:density}.

	$(3.\implies 1.)$ Let $g\in C_T\bar\shc_c^{0+}(\R^d)$, $f_T\in\bar\shc_c^{(1+\beta)+}(\R^d)$ and consider the solution $f\in C_T\shc^{(1+\beta)+}$ to
	\begin{equation}\label{eq:PDE31}
		\left\{
		\begin{array}{l}
		\shl f=g\\ f(T)=f_T.
		\end{array}
		\right.
	\end{equation}
	We need to prove that $M^f$
as defined in \eqref{eq:Mart1}
is a $\P$-local martingale. By the density of $C_T\shc_c^{0+}(\R^d)$ in $C_T\bar\shc_c^{0+}(\R^d)$, thanks to the continuity of \eqref{eq:PDE31},
due
Corollary \ref{cor:unDCBis} and Remark \ref{rmk:unDCBis}
together with the fact that the space of local martingales vanishing at zero is closed under u.c.p. convergence topology, it is sufficient to suppose both $g$ and $f_T$ to have compact support. A similar reasoning is done at the beginning of the proof of Proposition 
\ref{pr:equivMP}.

	Let us consider $\rho\in C_c^{\infty}$ and define 
	$$\rho_n(x)=n^d\rho(nx)\ \text{ and } f^n_T:=f_T\ast \rho_n.$$
	It is well-known that there exists a single compact set $K$ such that, for all $n\in\N$, ${\rm supp}(f_T^n)\subseteq K$. Obviously
        $f^n_T\in C_b^{2+}$.
	For all $n$, we define $B^n:=P_{\frac1n}B\in C_T\shc^{0+}(\R^d;\R^d)$, set $\shl^n f=\partial_t f+\frac12 \Delta f+\nabla fB^n$ and consider the solution
        $f^n$ solutions to 
	\begin{equation}
		\shl^n f^n=g,\ f^n(T)=f^n_T.
	\end{equation}
	By Proposition \ref{pr:lunardi} $f^n \in C_b^{1,2+}([0,T]\times \R^d)$.
        By Corollary \ref{cor:unDCBis} we have that
        \begin{equation} \label{eq:fnConv}
          f^n \rightarrow f  \ {\rm in} \ C_T\shc^{(1+\beta)+}.
\end{equation}
     	Let us now consider   the sequence $(\chi_n)$
        defined in Lemma \ref{lm:chi_n} and we set
        $\tilde f^n:=\chi_n f^n\in C_c^{1,2+}([0,T]\times \R^d)$. By evaluation, we get 
	$$\shl^n \tilde f^n= g^n+ f^n\nabla \chi_n B^n+\nabla f^n\nabla \chi_n+\frac12 f^n\Delta \chi_n,   $$
        where $g^n:=\chi_n g$.

We observe that $\tilde f^n $ solves
\begin{align*}
  \shl \tilde f^n&=g^n + f^n\nabla \chi_n B^n+\nabla f^n\nabla \chi_n
                   +\frac12 f^n\Delta \chi_n   +\nabla \tilde f^n(B-B^n)  \\
                 &= g^n + f^n\nabla \chi_n B^n+\nabla f^n\nabla \chi_n +
                   \frac12 f^n\Delta \chi_n
                   + (\nabla f^n \chi_n + f^n \nabla \chi_n) \cdot (B-B^n)\\
  &= g^n + l^n,
\end{align*}
with the notation
$$ l^n :=  \nabla f^n\nabla \chi_n + \frac12 f^n\Delta \chi_n
                   + \chi_n \nabla f^n \cdot (B-B^n) + f^n \nabla \chi_n \cdot B$$
and terminal condition $ \tilde f^n(T)=\tilde f^n_T$.

Since $g$ has compact support, we can find some $R\in\N$ such that ${\rm supp}(g)\subseteq B(0,R)$ of radius $R$.
Consequently for $n > R$, we have that $g\chi_n\equiv g$.
A similar phenomenon takes place  for $\tilde f^n_T$.

By item {\it 3.} and the fact that $A^\P$ is linear, the fact that
for $n$ large enough it holds that $g \in C_T\shc_c^{0+}$ 
\begin{equation}\label{eq:tildefnM}
	\tilde f^n(t,Z_t)-A^{\P}(g) + A^\P(l^n) = \tilde f^n(t,Z_t)- \int_0^t g(r,Z_r) dr + A^\P(l^n)
\end{equation}
is a $\P$-local martingale.

The idea is now to take the limit u.c.p. in \eqref{eq:tildefnM} when $n$ goes to infinity.
Taking into account \eqref{eq:fnConv}
it is clear that $ \tilde f^n$ goes to $f$ uniformly on compact sets.

So $\tilde f^n(t,Z_t)$ goes u.c.p. to $f(t,Z_t)$. To prove that
$$f(t,Z_t)-A^{\P}(g)=f(t,Z_t)-\int_0^t g(r,Z_r)dr,$$
and therefore conclude the proof of item 1.,
we only need to prove that $l^n$ converges to zero in $C_T\shc^{(-\beta)+}$.
Indeed $A^\P$ is continuous.


Given $\al<\be$ together with \eqref{eq:bonyt} we have
\begin{align*}
  \Vert l^n \Vert_{C_T\shc^{-\al}} &\le
   \Vert  \nabla f^n\nabla \chi_n\Vert_{C_T\shc^{-\al}}+
                                     \frac12
                                    \Vert f^n\Delta \chi_n \Vert_{C_T\shc^{-\al}}
      + \Vert \chi_n \nabla f^n \cdot (B-B^n) \Vert_{C_T\shc^{-\al}} +
        \Vert f^n \nabla \chi_n \cdot B \Vert_{C_T\shc^{-\al}} \\
  \leq&  \|\nabla f^n\|_{C_T\shc^{\beta}}\|\nabla \chi_n\|_{\shc^{-\al}}
        +\frac12 \| f^n\|_{C_T\shc^{1+\beta}}\|\Delta \chi_n\|_{\shc^{-\al}}
+ \|\nabla  f^n\|_{C_T\shc^\beta}\|\chi_n\|_{\shc^\beta} \|B-B^n\|_{C_T\shc^{-\al}}.
  \\
   &+ \|f^n\|_{C_T\shc^{1+\beta}}\|\nabla \chi_n\|_{\shc^\beta}\|B\|_{C_T\shc^{-\al}}.
           \end{align*}
Lemma \ref{lm:chi_n}  states that the sequence
  $(\Vert \chi_n \Vert_{\shc^\beta})$ is  bounded,
   $\nabla \chi_n\to 0$ in  $\shc^{\beta+}$ and $\Delta\chi_n\to 0$ in  $\shc^{(-\beta)+}$.
  Previous considerations,  \eqref{eq:fnConv}
and  $\|\cdot\|_{C_T\shc^{\beta}}\leq\|\cdot\|_{C_T\shc^{1+\beta}}$
allow to finally establish that $l^n \rightarrow 0$
 in $C_T\shc^{(-\beta)+}$.
 \end{proof}







\subsection{Well-posedness and dynamics on the rough martingale problem.}

A direct consequence of     Theorem 4.5 in \cite{issoglio_russoMPb} is the following.
\begin{proposition} \label{pr:rMP} Let us suppose
Assumptions  \ref{ass:beta} and \ref{ass:B}.
	Let $\nu$ be a Borel probability measure on $\R^d$.
	The rough martingale problem rMP$(B,\nu)$
	in the sense of Definition \ref{def:rMP}
	admits existence and uniqueness.
      \end{proposition}
      An immediate consequence is that the ''linear growth martingale problem''
      is also well-posed.
      \begin{corollary} \label{cor:rMP}
        There is a unique solution to rMP$^{\hat \shd_\shl}(B,\nu)$.
    \end{corollary}
    \begin{proof}
      Since $\shd_\shl \subseteq {\hat \shd}_\shl$
      any solution $\P$ to rMP$^{\hat \shd_\shl}(B,\nu)$ is a solution
    to rMP$(B,\nu)$. The converse is a consequence
    of Theorem \ref{thm:rMPext}.
  \end{proof}
  The result below characterizes the dynamics of the canonical process $Z$
  under the unique solution of the rough martingale problem.

\begin{corollary}\label{cr:Dynamics}
	Let $\P$ be the solution to the rMP($B,\nu$). Then $Z$ is a weak Dirichlet process with  decomposition  
\begin{equation} \label{eq:BMZ0}
  Z=A^{\P}(B)+W+Z_0,
  \end{equation}
	where $W$ is a $d$-dimensional Brownian motion, and $A^{\P}(B)=(A^{\P}(B^1),\dots, A^{\P}(B^d))$.
\end{corollary}
\begin{proof}
  By Lemma \ref{lem:dec-op} we have that $Z^i_t$ is a weak Dirichlet Process
  for all $i=1,\ldots,d$.
  Thanks to item 3. of Theorem \ref{thm:extendA}
  $Z^i = M^{id^i} + A^{\P}(B^i),  $ where $M^{id^i}$ is
  a local martingale, for all $1 \le i \le d$.
  Moreover, by Corollary
  \ref{cor:extendA}, for $1 \le i,j \le d$, we have
  $$[M^{id^i}, M^{id^j}]_t=\int_0^t  \nabla id^i \cdot \nabla id^j (Z_r)dr= \delta_{ij}t,$$
  where $\delta_{ij}$ is the Kronecker symbol.
  Thus, by L\'evy's characterisation theorem, we have that $M^{id}=W+Z_0$,
  where $W$ is a $d$-dimensional Brownian motion. This concludes the proof.
\end{proof}

\section{Structural Features of the rough Martingale Problem}
\label{sec:extension-alia}

In this section
Assumptions \ref{ass:beta} and \ref{ass:B} will be in force and $ \nu$ will be a fixed (initial) Borel probability 
on $\R^d$. The objective of this section is to give key properties of the rMP($B,\nu$) such as time restriction, $L^p$ estimates satisfied by the canonical process under the solution $\P$ of rMP($B,\nu$) and conditional martingale problem, in order to have a deeper understanding of the mathematical object. These results will be used in Section \ref{sec:IPS}.

\subsection{Time restriction of the rMP($B,\nu$)}\label{sc:restrMP}

Thinking to the classical obvious properties fulfilled by ordinary
differential equations,
 we expect that, given a solution $\P$ of rMP($B,\nu$),
its ''restriction'' to $[0,t]$ for any $ t \in (0,T)$
is still a solution of  rMP($B\vert_{[0,t]},\nu$).
This is not obvious with the formalism introduced at
Definition \ref{def:rMP}.

\begin{notation} \label{not:restriction}
  Let $E$ be a topological vector space.
	Given any $g:[0,t]\to E$ for any $t\in[0,T]$, we define $g^t:[0,T]\to E$
	as $g^t(s):=g(t\wedge s), \ s\in[0,T]$.
\end{notation}

Compatibly with the end of Section \ref{sc:prelim}, if $ t \in (0,T)$, we denote $\Omega_t^d:= C([0,t];\R^d)$.

\begin{theorem}\label{thm:rMPrestr}
  Let $\P$ be the solution of rMP$(B,\nu)$. Let $t\in[0,T)$ and define $\P^t$ (with related expectation $\E^t$) as the probability on
$(\Omega_t^d,\shb(\Omega_t^d))$ defined by 
	\begin{equation}\label{eq:rMPrestr}
		\E^t\left[F(Z_r;r\leq t)\right]:=\E\left[{}^t F(Z_r;r\leq T)\right]
	\end{equation}
	for any  $F:\Omega_t^d\to \R$ bounded and Borel-measurable,	where
        ${}^{t}F:\Omega_T^d\to \R^d$ is defined by ${}^tF(\omega):=F(\omega^t)$. Then $\P^t$ is a solution to  rMP$(B|_{[0,t]},\nu)$.
\end{theorem}
\begin{remark} \label{rmk:difficulty}
  Let us consider the map
  \begin{equation} \label{eq:Lt}
    {}^t\shl \varphi:=\partial_s \varphi+\frac12 \Delta \varphi+\nabla \varphi\cdot B|_{[0,t]}.
\end{equation}
As we mentioned earlier, the statement of Theorem \ref{thm:rMPrestr} is
not immediate: in fact, the rMP($B,\nu$) is defined through the solutions $f$ of
  the backward Kolmogorov PDE $\shl f=g, f(T,\cdot)=f_T$.
  Given a solution $f \in C_t\shc^{ (1+\beta)+}$ of ${}^t \shl f=g$, with $g \in  C_t\shc^{0+}$.
  It looks problematic to define an extension $\hat f \in C_T\shc^{(1+\beta)+}$ of $f$
(resp. $\hat g \in C_T\shc^{0+}$ of $g$), such that
  $ \shl \hat f= \hat g$.
\end{remark}

\begin{proof} [Proof of Theorem \ref{thm:rMPrestr}.]
  
  Since $\P$ is a solution of rMP($B,\nu$), owing to Theorem \ref{thm:extendA}, we know the existence of the  linear and continuous operator
  $$ C_T\shc^{(-\beta)+}\ni g\mapsto A^{\P}(g) \in \shc^\P([0,T]),$$
  such that its restriction to $C_T\shc^{0+}$
  is 
  the map  $A(g) =  \int_0^\cdot g(r,Z_r)dr$.
We define, for all $g\in C_t\shc^{(-\beta)+}$, the operator
	 \begin{equation}\label{eq:extendLT}
	 A^{\P^t}_s(g):=A^{\P}_s(g^t), \ s \in [0,t],
	 \end{equation}
         where $ g^t$  has been defined in Notation \ref{not:restriction}.
         Since every $\shf_t$-measurable $\P$-null set is a
     $\P^t$-null set,  $A^{\P^t}$ is also defined up to $\P^t$-indistinguishability.  
    Moreover    $ A^{\P^t}$ fulfils the following properties.
  \begin{enumerate}
  \item	 
    For all $g\in C_t\shc^{0+}$ 
	 \[
          A_s(g^t)=\int_0^s g(t\wedge r,Z_r)dr=\int_0^s g(r,Z_r),\  s \le t.
	 \]
       \item	 $A^{\P^t}$ is continuous with respect to $C_t\shc^{(-\beta)+}$
         with values in $\shc([0,t])$.
         Indeed, since it is linear it is enough to prove its continuity at zero.
         Let $(g_n)$ be a converging sequence in $C_t\shc^{(-\beta)+}$ towards $0$, then there exists $\al<\be$ such that $g_n\to 0$ in $C_t\shc^{-\al}$. From this we deduce immediately that $(g_n^t)$ converges in $C_T\shc^{-\al}$,
         since $\|g_n\|_{C_t\shc^{-\al}}=\|g_n^t\|_{C_T\shc^{-\al}}$.
         In particular $g_n^t$ converges to zero in $C_T\shc^{(-\beta)+}$
         which allows to conclude that $A^{\P^t}(g_n) \rightarrow 0$
         $\P^t$-u.c.p.
         \end{enumerate}
	          
       Let us consider $\varphi \in C_b^{1,2+}([0,t]\times \R^d)$. In view of Theorem \ref{thm:extendA} 4., it is enough to prove that 
      $$M^\varphi_s:=\varphi(s,Z_s)-\varphi(0,Z_0)-A^{\P^t}_s({}^t\shl \varphi), s \in [0,t],$$
	is a martingale under $\P^t$, where ${}^t\shl$ was defined in \eqref{eq:Lt}.
        By Lemma \ref{lm:LA}, applied with $T= t$ and $f=\varphi$,
        there is a function 
        $s\mapsto\hat\varphi(s)$
which 
belongs to $C_b^{1,2+}(\R_+\times\R^d)$ and
in particular its restriction   to $[0,T]\times\R^d$
(still denoted by $\hat\varphi$)
belongs to
$C_b^{1,2+}([0,T]\times\R^d)$.
Consequently by Theorem \ref{thm:extendA} 4. we have
$$
M^{\hat \varphi}_s:={\hat \varphi}(s,Z_s)- {\hat \varphi}(0,Z_0)
-A^{\P}_s({}\shl {\hat \varphi}), s \in [0,T],$$
is a local martingale under $\P$ with quadratic variation $\int_0^t |\nabla f|^2(r,Z_r)dr$.
Consequently it is a martingale, being $\nabla f$ bounded.

Now we observe that, for $s\le t $ we have $\hat \varphi(s)=\varphi(s)$
and $\shl {\hat \varphi}(s) = {}^t\shl(\varphi)(s),$ for $s < t$.

Owing now to Remark \ref{rmk:extendA}, together with  the fact that  
$ ({}^t\shl (\varphi))^t = \shl {\hat \varphi}$ on $[0,t]$
we get
$$ A_s^{\P}(\shl {\hat \varphi}) =  A_s^{\P}(({}^t\shl{\varphi})^t)
=  A_s^{\P^t}({}^t\shl \varphi),\  \text{ for } s \in [0,t],$$
where the latter equality is a consequence of \eqref{eq:extendLT}.
Consequently for $s \in [0,t]$,
$ M^{\varphi}_s = M^{\hat \varphi}_s $, so that
$(M^{\varphi})_{s \in [0,t]}$ is a martingale under $\P$.
In particular, let $ 0 \le s_0 <  s \le t.$  For every continuous bounded functional
$G: C([s,t];\R^d) \rightarrow \R$ we have
$$ \E^\P( (M^{\varphi}_s - M^{\varphi}_{s_0}) G((Z_r)_{s\le r \le t})) = 0,$$
which also implies that  previous equality holds when substituting
$\P$ with $\P^t$.
This concludes the proof.
\end{proof}

\subsection{$L^p$ estimates of solutions to rMP}

\label{sec:Lpestim}

This subsection is devoted to the study of $L^p$-estimates
for the canonical process $Z$ under $\P$ solution to the rMP($B,\nu$). These estimates are useful in order to establish the regularity of the paths of $Z$ and they will be used in Section \ref{sec:IPS} for the study of the tightness of a singular $N$-particle system.

\begin{proposition}\label{pr:phi_tight0} 
        Let   $ 0 < \alpha < \beta$ and  $0 < \g < \frac{\beta}{2}$.
        There is a function $\tilde R_{\al,\beta}:\R_+\times\R^2\to\R$,
        increasing in all variables, such that the following holds.
 The solution $f$  of
	\begin{equation} \label{eq:tight1}
		\shl f=g,\quad f(T,\cdot)=0,
	\end{equation}
  with $B \in C_T\shc^{-\alpha}(\R^d;\R^d),  g\in C_T\shc^{-\alpha}$, 
fulfils the estimate
 
\begin{equation} \label{tight2}
  |f(T,y)-f(s,x)|\leq \tilde R_{\al,\beta}(T,\|B\|_{C_T\shc^{-\alpha}},\|g\|_{C_T\shc^{-\alpha}}) |T-s|^\g,\ \forall x,y\in\R^d,\ \forall s\in[0,T].
  	\end{equation}
        
\end{proposition}
\begin{proof}

  Since $f$ is a solution to $\shl f=g,\ f(T,\cdot)=0$, we have that $f(T,y)=0$ for all $y\in\R^d$ in particular $f(T,y)=f(T,x)=0$. Hence, by \eqref{eq:holdun} in Theorem \ref{thm:unDC} with
$\eps:= \beta- 2\gamma$ and $f=u$
  we have
       \begin{align*}
         |f(T,y)-f(s,x)| &=|f(T,x)-f(s,x)|\leq
\Vert f(T,\cdot)-f(s,\cdot) \Vert_\infty \le \Vert f(T,\cdot)-f(s,\cdot) \Vert_{\shc^{1+ \eps}} \\
&\le
         R_{\al,\beta}(T, \|B\|_{C_T\shc^{-\alpha}},\|g\|_{C_T\shc^{-\alpha}},\|0\|_{D\shc^\beta}) |T-s|^\g.
  \end{align*}
  Now by defining, for all $x,y\in\R^d$, $\tilde R_{\al,\beta}(T, x,y):=R^0_{\al,\beta}(T, x,y,0)$ we conclude.
			\end{proof}

\begin{proposition}\label{pr:phi_tight}
  Let   $ 0 < \alpha < \beta$ and $p \ge 1$.
  Then there exists an increasing function (in both arguments)
   $\hat R_{\al,\beta}:\R\times \R_+\to\R$ 
  such that the following holds.

Given $B \in C_T\shc^{-\alpha}(\R^d;\R^d)$, a solution $\P$
   to the rMP($B,\nu$) and $\phi= (\phi^1,\ldots,\phi^d)^\top \in C_TD\shc^{\beta}$  the solution to 
	\begin{equation*}
		\shl \phi=0,\ \phi_T=id,
	\end{equation*}
we have 
 \begin{equation} \label{eq:BoundsPhi}
		\E^\P[|\phi(t,Z_t)-\phi(s,Z_s)|^p]\leq \hat R_{\al,\be, p}(T,\|B\|_{C_T\shc^{-\al}})|t-s|^{\frac p2},\ 0\leq s \leq t\leq T.
	\end{equation}
\end{proposition}

\begin{proof}
  We consider the function $R_{\al,\beta}:\R_+\times\R^3\to\R$ provided by 
  Theorem \ref{thm:unDC}.
Let $\P$ and $\phi$ as above related to a given $B$.
  Since $\phi^i\in \hat\shd_\shl $ with $g^i=0$ for all $i=1,\dots,d$,  Corollary \ref{cor:rMPext} guarantees that
  $\phi^i(\cdot,Z_\cdot)$ are local martingales.
  By Corollary \ref{cor:extendA} we have
	\[
	[\phi^i(\cdot,Z_\cdot),\phi^j(\cdot,Z_\cdot)]=\int_0^t (\nabla \phi^i\cdot \nabla\phi^j)(r,Z_r)dr,\ \forall i,j=1,\dots,d.
	\]
	Therefore using BDG inequality we get
	\begin{equation} \label{eq:BDGMoments}
	\E^\P[|\phi(t,Z_t)-\phi(s,Z_s)|^p]\leq  C_p\E^\P\left[\left|\int_s^t|\nabla\phi|^2(r,Z_r)dr\right|^{\frac p2}\right].
	\end{equation}
        Thanks to Theorem \ref{thm:unDC} {\it i)}, we have
	\[
	\|\nabla \phi^i\|_{C_T\shc^{\beta}}\leq \|\phi^i\|_{C_TD\shc^\beta}\leq  R_{\al,\beta}(T, \|B^i\|_{C_T\shc^{-\al}},\|0\|_{C_T\shc^{-\al}},\|id^i\|_{D\shc^{\beta}}),\ i=1,\dots,d.
	\]
	Since $\|B^i\|_{C_T\shc^{-\al}}\leq \|B\|_{C_T\shc^{-\al}}$ and $\|id^i\|_{D\shc^{\beta}}=1$ together with the fact that $R_{\al,\beta}$ is an increasing function, we have
	\begin{equation}\label{eq:nablauR}
		\|\nabla \phi\|_{C_T\shc^{\beta}}\leq  R_{\al,\beta}(T, \|B\|_{C_T\shc^{-\al}},\|0\|_{C_T\shc^{-\al}},1).
	\end{equation}
	Substituting \eqref{eq:nablauR} in \eqref{eq:BDGMoments} we obtain 
	\begin{equation}\label{eq:5.5final}
		\E^\P[|\phi(t,Z_t)-\phi(s,Z_s)|^p]\leq  C_pd^{\frac p2}R_{\al,\beta}(T, \|B\|_{C_T\shc^{-\al}},\|0\|_{C_T\shc^{-\al}},1)\left|t-s\right|^{\frac p2}.
	\end{equation}
       Defining now  $\hat R_{\al,\be,p}(T,x):=C_p R_{\al,\beta}(T,x,0,1))^{\frac p2}$, which is
       obviously an increasing function (in both arguments)
       thanks to Theorem \ref{thm:unDC},  \eqref{eq:5.5final} implies
       the validity of \eqref{eq:BoundsPhi}.


\end{proof}

\begin{theorem}\label{thm:LPestim}
  Let $\al<\be$ and $0<\lambda<\frac\be 2$. Then there exists a function
  $\bar R_{\al,\be,p}:\R_+\times \R_+\to\R$, which is increasing in all variables, such that the following holds.
Let $B\in C_T\shc^{-\al}$ and $\P$ be the solution to the rMP($B,\nu$). Then for all $p\geq 1$
		\begin{equation}\label{eq:LPestim}
                  \E^\P[|Z_t-Z_s|^p]\leq \bar R_{\al,\be,p}(T, \|B\|_{C_T\shc^{-\al}})|t-s|^{\lambda p}, \forall s,t \in [0,T].
	\end{equation}

\end{theorem}

\begin{proof}
  We recall the notation     ${}^t B:=  B|_{[0,t]}$.
  Without restriction of generality we can prove \eqref{eq:LPestim} supposing $s < t$.
  For every $t\in[0,T]$, we consider the decomposition $id  = \phi + u$
  for  the identity function
  $id(x)=x$, where $\phi $ and $ u$ are solutions in $C_tD\shc^{\beta}$ to
		\begin{equation*}
		{}^t\shl \phi= 0, \  \phi(t)=id, \text{ and }
		{}^t\shl u=-{}^tB,\ u(t)=0,
	\end{equation*}
		respectively.
	For all $r\le t$ we write $Z_r=  (\phi - u)(r,Z_r)$ and, for all $s \in[0,t]$, we can obtain
		\begin{equation}\label{eq:PhiU}
                  	\E^\P[|Z_t-Z_s|^p]
    \le 2^{p-1} \E^\P[| \phi(t,Z_t) - \phi(s,Z_s) |^p] +2^{p-1}\E^\P[|u(s,Z_s) - u(t,Z_t)|^p] =: I^1 + I^2.
	\end{equation}
 Thanks to Theorem \ref{thm:rMPrestr}, since $\P$ is a solution to rMP($B,\nu$), then $\P^t$ is a solution to rMP(${}^tB,\nu$).
 We analyse the two summands $I^1, I^2$ of \eqref{eq:PhiU} separately. 
 By Proposition \ref{pr:phi_tight}, with $t$ replacing $T$,
we get
 	\begin{equation}\label{eq:Phi}
 I^1 =   2^{p-1}\E^{\P^t}[| \phi(t,Z_t) - \phi(s,Z_s) |^p]\le  2^{p-1}\hat R_{\al,\be,p}(t, \|{}^t B\|_{C_t\shc^{-\al}})|t-s|^{\frac p2} \le
 C_p \hat R_{\al,\be,p}(T, \| B\|_{C_T\shc^{-\al}})|t-s|^{\frac p2},
	\end{equation}
the latter inequality holding because $R_{\al,\be,p}$ is increasing.

For what concerns $I^2$,
we apply Proposition \ref{pr:phi_tight0} 
with $f= u^i, g=-{}^tB^i$, $ i=1,\ldots,d$, where $u=(u^1,\ldots,u^d)^\top$
and $T=t$, to get 
	\begin{align}\label{eq:U}
      I^2 &\le
       C_p (\tilde R_{\al,\beta}(t,\|{}^tB\|_{C_t\shc^{-\alpha}},\|{}^tB\|_{C_t\shc^{-\alpha}}))^p |t-s|^{\lambda p} \\
                                               &\le   C_p (\tilde R_{\al,\beta}(T,\|B\|_{C_T\shc^{-\alpha}},
                \|B\|_{C_T\shc^{-\alpha}}))^p |t-s|^{\lambda p}. 
	\end{align}
	Now adding together \eqref{eq:Phi} and \eqref{eq:U} and substituting the estimates obtained therein in \eqref{eq:PhiU} we get 
        \begin{align*}
          \E^\P[|Z_t-Z_s|^p]&\le C_p \hat R_{\al,\be,p}(T, \| B\|_{C_T\shc^{-\al}})|t-s|^{\frac p2}+
                C_p(\tilde R_{\al,\beta}(T,\|B\|_{C_T\shc^{-\alpha}},\|B\|_{C_T\shc^{-\alpha}}))^p |t-s|^{\lambda p}\\
	&\le C_p \left( \hat R_{\al,\be,p}(T, \|B\|_{C_T\shc^{-\al}})T^{\frac p2-\lambda p}+ \tilde R_{\al,\beta}(T,\|B\|_{C_T\shc^{-\alpha}},\|B\|_{C_T\shc^{-\alpha}}))^p\right) |t-s|^{\lambda p},
\end{align*}
where we grouped $|t-s|^{\lambda p}$, since $\lambda<\frac\be 2<\frac12$. Thus, defining now 
$$\bar R_{\al,\be,p}(T,x):=C_p \left(\hat R_{\al,\be,p}(T,x)T^{\frac p2-\lambda p}+(\tilde R_{\al,\beta}(T,x,x))^p\right),$$ 
we conclude.
\end{proof}

\section{The conditional rough martingale problem}
\label{sec:CRMP}


Let Assumption \ref{ass:beta} and \ref{ass:B} hold and  let
$\P$ be the solution of rMP($B,\nu$), where $\nu$ is a given Borel probability measure in $\R^d$.
Let us denote by $\P^{\bar Z}$ the law of $\bar Z:=  (Z^2,\ldots,Z^d)$.

Then, disintegration $\P$ with respect to
$\bar Z:=  (Z^2=\g^2,\ldots,Z^d=\g^d)$, where $ \bar \gamma = (\g^2,\ldots,\g^d) \in \Omega_T^{d-1}$, see e.g.  Chapter III, nos. 70–72 in \cite{dellacherie},
there is a 
random kernel  $(\P_{\bar \gamma}(d\gamma^1), \bar \gamma \in \Omega_T^{d-1})$
of Borel probabilities on $\Omega^1_T$ 
such that
\begin{equation} \label{eq:516}
  \P(d\g^1,d\bar\g)=
	:\int_{\Omega_T^{d-1}}\P^{\bar Z}(d\bar\g)\int_{\Omega_T}\P_{\bar \g}(d\g^1).
\end{equation}
Roughly speaking $\P_{\bar \gamma}(d\gamma^1)$ is the conditional
law of $Z^1$ given $(Z^2,\ldots,Z^d) = (\gamma^2, \ldots, \gamma^d)$.

The objective of this section is to show that, whenever the drift $B$ is of the specific form $\shb$ given in
\eqref{eq:shbMixt} below,
 that the conditional law  of $Z = (Z^1,\ldots,Z^d)$,
 with respect to $Z^2 = \gamma^2, \ldots, Z^d= \gamma^d$,
 which is a probability Borel law on $\Omega_T^1$,
 is the unique solution of a rough martingale problem
 with distributional drift, formally given by $(t,x) \mapsto B^{\bar \gamma}(t,x,\bar \gamma(t))$,
 i.e. {\it freezing} the drift $B(t,Z^1,\bar Z)$ with respect to $\bar Z = \bar \gamma$.
  We start with two preparatory lemmata.
  But before this we introduce two useful notations.
  For $\phi \in \shd_\shl$ we denote
\begin{equation*}
  M^{\phi}_t(Z):= \phi(t,Z_t)-\int_0^t \shl \phi(r,Z_r)dr,  \\
\end{equation*}
and
\begin{equation}\label{eq:DLb}
M^{\phi}_t(Z^1,\bar\g):=\phi(t,Z_t^1,\bar \g_t)-\int_0^t \shl \phi(r,Z^1_r,\bar \g_r)dr, t \in [0,T].
	\end{equation}

\begin{lemma}\label{lm:DLb}
  Let $\P$ be the solution to the rMP($B,\nu$). Then there is a
  $\P^{\bar Z}$-null set  $\shn\subseteq \Omega^{d-1}_T$,  such that for
  every $\bar\g\not\in N$ and for every  $f\in\shd_\shl$, 
  the processes   $M^f_t(Z^1,\bar\g)$
and  $M^{f^2}_t(Z^1,\bar\g)$
are $\P_{\bar \g}$-martingales.
        In particular
	\begin{equation}
          [M^f_t(Z^1,\bar\g)]
          =\int_0^t |\nabla f|^2(r,Z^1_r,\bar\g_r)dr.
\end{equation}
\end{lemma}
\begin{proof}
  In this proof we will indicate by  $E^0$  a countable dense in $\shd_\shl$.
  $E^{\Omega_s}$ (resp. $E^{\Omega_s^{d-1}}$)
will stand for a generic countable  subset of
the Banach space of continuous bounded
real-valued functionals $\Omega_s \rightarrow \R$
(resp. $\Omega_s^{d-1} \rightarrow \R$)
fulfilling the following sequential density property.
For every bounded continuous functional 
$F_s: \Omega_s \rightarrow \R $ (resp. $ \Phi_s: \Omega^{d-1}_s \rightarrow \R$,
there exists a sequence $(F^n_s)$, (resp. $ (\Phi^n_s)$) converging to $F_s$ (resp. $\Phi_s$), in a pointwise uniformly bounded way, see Lemma 4.6 in \cite{issoglio_russoMPb} for details.
By Theorem \ref{thm:rMPext}, and restricting ourselves to $\shd_\shl\subseteq\hat\shd_\shl$, we can state the following.

{\bf Claim C:} for every $f\in\shd_\shl$
\begin{align}\label{eq:claimC}
	&f(t,Z_t)-f(0,Z_0)-\int_0^t g(r,Z_r)dr\\
\notag	&f^2(t,Z_t)-f^2(0,Z_0)-\int_0^t(2fg+|\nabla f|^2)(r,Z_r)dr
	\end{align}
	are $\P$--martingales.

Arguing as in Fact 1 of Theorem 4.5 of
  in \cite{issoglio_russoMPb}, together with a functional
  monotone class
  theorem
  we can prove that
{\bf Claim (C)} holds if and only if
for every $0\leq s\leq t\leq T  $  with $s,t\in\Q$
  there exist    $ E^0$, $E^{\Omega_s}$ (resp. $E^{\Omega_s^{d-1}}$)
\begin{equation}\label{eq:mart_equiv}
   \E^\P[(M^\phi_t-M^\phi_s) F_s(Z^1_r;r\leq s)\Phi_s(\bar Z_r;r\leq s)]=0,
 \end{equation}
 for $\phi = f, f^2$,
      for every $f\in E^0$,
      $\Phi_s\in E^{\Omega_s^{d-1}}$, $F_s \in E^{\Omega_s}$.

      We can rewrite \eqref{eq:mart_equiv} making use of \eqref{eq:516}
we have 
\begin{equation*}
	\int_{\Omega_T^{d-1}}\P^{\bar Z}(d\bar \g)	\E^{\P_{\bar\g}}[(M^\phi_t(Z^1,\bar \g)-M^\phi_s(Z^1,\bar\g)) F_s(Z^1_r;r\leq s)\Phi_s(\bar \g_r;r\leq s)] = 0.
\end{equation*}
This implies that, for $\bar \g\in\Omega_T^{d-1}$ outside a $\P^{\bar Z}(d\bar \g)$-null set $\shn$, for every $f\in E^0$, $F_s\in E^{\Omega_T}$, and $\Phi_s\in E^{\Omega_T^{d-1}}$
\begin{equation}\label{eq:mart_equiv1Bis}
\E^{\P_{\bar \g}}[(M^\phi_t(Z^1,\bar \g)-M^\phi_s(Z^1,\bar \g)) F_s(Z^1_r;r\leq s)\Phi_s(\bar\g_r;r\leq s)]=0.
\end{equation}
Consequently being $\bar \gamma$ frozen,
\begin{equation}\label{eq:mart_equiv1}
\E^{\P_{\bar \g}}[(M^\phi_t(Z^1,\bar \g)-M^\phi_s(Z^1,\bar \g)) F_s(Z^1_r;r\leq s)]=0.
\end{equation}
This, by the usual density arguments, as those used for
establishing Claim (C),
we have  that,
$\P^{\bar Z}(d\bar \g)$-a.e.,
the processes
$M^\phi_t(Z^1,\bar \g)$ are  martingale under $\P_{\bar \g}$
for $\phi = f, f^2$, for every $f \in \shd_\shl$.
This concludes the proof of the lemma.
\end{proof}

\begin{lemma}\label{lm:L}
  Let $\P$ be the solution to rMP$(B,\nu)$.
 Then there is a null set $\shn$ with respect to $\P^{\bar Z}$,
  such that, for all $\gamma \in \shn$, the following holds.
  Let $(f^n_T)$ be a sequence in $\shc_c^{(1+\beta)+}(\R^d)$ such that $f^n_T\to f_T$ in $\shc^{(1+\beta)+}$, and let $(g^n)$ be a sequence in $C_T\shc_c^{0+}(\R^d)$ such that $g^n\to g$ in $C_T\shc^{(-\beta)+}(\R^d)$. Let also $(f^n)$ be the sequence of solutions to $\shl f^n=g^n,\ f^n(T)=f^n_T$.
  $M^{f^n}(Z^1,\bar\g)$
	converges u.c.p. under $\P^{\bar\g}$ to a $\P^{\bar\g}$-local martingale $	M^{f}(Z^1,\bar\g)$ so that 
	\begin{equation}  \label{MfQvar}
	[M^{f}(Z^1,\bar\g)] =\int_0^\cdot |\nabla f|^2(r,Z^1_r,\bar\g_r)dr.
	\end{equation}
\end{lemma}
\begin{proof}
  Let us choose the $\P$-null set $\shn$ provided by Lemma \ref{lm:DLb}.
  Let us fix now  $\bar \gamma \in \Omega_T^{d-1}$ not belonging to  $\shn$
until the end of the proof. 
In particular for all $n$
$$ M^{f^n}_t(Z^1,\bar\g) = f^n(t,Z^1_t,\bar \gamma_t) - \int_0^r g^n(r,Z^1_r,\bar \gamma_r) dr,$$
is a $\P^{\bar\g}$-local martingale with quadratic variation given by
\begin{equation} \label{eq:bracketGamma}
  [ M^{f^n}(Z^1,\bar\g)] =\int_0^\cdot |\nabla f^n|^2(r,Z^1_r,\bar\g_r)dr.
\end{equation}
By a standard bilinearity argument  we have
	\begin{equation}\label{eq:bilQvar}
          [M^{f^n}(Z^1,\bar \g)-M^{f^m}(Z^1,\bar \g)] =
          \int_0^\cdot |\nabla f^n-\nabla f^m|^2(r,Z^1_r,\bar\g_r)dr.
	\end{equation}
	Since $g^n\to g$ in $C_T\shc^{(-\beta)+}(\R^d)$ and $f^n_T\to f_T$ in $\shc^{(1+\beta)+}(\R^d)$,
 by Remark 4.18 following Lemma 4.17 in \cite{issoglio_russoMPb} we have that $f^n\to f$ in $C_T\shc^{(1+\beta)}(\R^d)$, 
        hence in $C^{0,1}([0,T]\times \R^d)$.
        Given $\phi \in C_T\shc^{(1+\beta)}(\R^d)$,  we denote
$\tilde \phi(t,x_1):=\phi(t,x_1,\bar\g_t)$.
With this notation        the sequence of functions
$(\tilde f^n)$ converges in $C^{0,1}([0,T]\times \R)$ to $\tilde f$.
Therefore, by \eqref{eq:bilQvar}
        \begin{equation} \label{eq:bilQvarBis}
          [M^{f^n}(Z^1,\bar \g)-M^{f^m}(Z^1,\bar \g)] \rightarrow_{n,m \rightarrow \infty} 0.
\end{equation}          
Consequently, by (5.5) of Problem 5.25 of \cite{karatzasShreve}, 
$M^{f^n}(Z^1,\bar \g)-M^{f^m}(Z^1,\bar \g)$ converges u.c.p.
to $0$, when $m, n$ go to infinity.
In particular the sequence $(M^{f^n}(Z^1,\bar \g))$ is Cauchy for a
usual
         metric describing the u.c.p. topology.
         Since the linear space of continuous local martingales
         vanishing at zero is complete
         under the u.c.p. convergence topology,
         then $M^{f^n}(Z^1,\bar \g)$ converges to some local martingale
         $M^{f}(Z^1,\bar\g)$.

         By \eqref{eq:bracketGamma}
\begin{equation} \label{eq:bracketGamma_n}
  (M^{f^n})^2_t(Z^1,\bar \g) - \int_0^t |\nabla f^n|^2(r,Z^1_r,\bar\g_r)dr,
  \end{equation}
         are local martingales. Since $M^{f^n}_t(Z^1,\bar \g)$ converge to $M^{f}_t(Z^1,\bar \g)$ u.c.p.
         and $\nabla f^n$ converges in $C^{0,1}$ to $\nabla f$,
         we have that \eqref{eq:bracketGamma_n} converge u.c.p. to 
\begin{equation} \label{eq:bracketGamma_Bis}
  (M^{f})^2_t(Z^1,\bar \g) - \int_0^t |\nabla f|^2(r,Z^1_r,\bar\g_r)dr.
\end{equation}
         Again, since  the linear space of continuous local martingales
         vanishing at zero is closed by the u.c.p convergence
         then \eqref{eq:bracketGamma_Bis}
         is a local martingale.
         Finally this implies \eqref{MfQvar}.

\end{proof}

We now focus on a specific class of Besov distributions. We consider $h^i\in C_T\shc^{(-\beta)+}(\R), 1 \le i \le d$,
$\ell\in C_T \shc^{\beta}(\R^d)$.
We consider $\shb\in C_T\shc^{(-\beta)+}(\R^d;\R^d)$
such that $\shb=(\shb_1,\ldots,\shb_d)$ is
formally of the form $\shb_i(t,x)= \ell(t,x) h^i(t,x_i)$.
The rigorous meaning is possible via the injection map $T^i$, which 
has been provided in Definition \ref{def:T}.
In the application to this chapter $d$ (resp. $N$) becomes $1$
(resp. $d$)

In  particular $\shb$
is  the vector, whose components are the pointwise products 
\begin{equation} \label{eq:shbMixt}
  \shb_i(t,x):=\ell(t,x)\cdot T^i_{h^i(t)} (x), 1 \le i \le d.
  \end{equation}
	Thanks to Proposition \ref{pr:besovT} we know that $t \mapsto T^i_{h^i(t)}\in C_T\shc^{(-\beta)+}(\R^d)$
	and by \eqref{eq:bonyt}  $\ell(t,\cdot)\cdot T^i_{h^i(t)}\in C_T\shc^{(-\beta)+}(\R^d)$.

From now on, for all $f\in C_b^{1,2+}([0,T]\times\R^d)$, we will denote with $\shl^\shb f$ the pointwise product given by
$$
\shl^{1,\shb} f(t,x)=\sum_{i=1}^d
 \shb_i(t,x) \cdot \partial_{x_i} f(t,x),
 $$
taking into account \eqref{eq:shbMixt}. 
For every $\bar \gamma \in \Omega^{d-1}_T$
and for any $f^1\in C_b^{1,2+}([0,T]\times \R)$ we denote
\begin{equation}\label{eq:L1Bg}
  \shl^{1,\shb_1,\bar \gamma}f^1(t,x_1)= \shb_1(t,x_1,\bar \gamma_t)
  \cdot \partial_{x_1} f(t,x_1).
\end{equation}
where, naturally 
\begin{equation} \label{eq:shb1}
\shb_1(t,x_1,\bar \gamma_t) = \ell(t,x_1,\bar \gamma(t))\cdot h^1(t,x_1),
(t,x_1) \in [0,T] \times \R.
\end{equation}
\begin{remark} \label{rmk:shb1}
  For all $\bar\g\in\Omega_T^{d-1}$ we have that
  $(t,x_1) \mapsto \shb_1(t,x_1,\bar \gamma_t)
  \in C_T\shc^{((-\beta)+}(\R)$ thanks to \eqref{eq:bonyt}.
  \end{remark}

Below we will make use of the notation $Z^1$ (resp. $\bar Z$)
the first component (resp. the other $d-1$ components) of the canonical process
$Z$.
  \begin{proposition}\label{pr:extendA1}
  Let $\P$ be the solution to rMP($\shb,\nu$). Then there exists  a $\P^{\bar Z}$-null set $\shn$ such that for all $\bar \g\notin \shn$ the following hold.
    \begin{itemize}
  	\item[i)] The map $C_T\shc_c^{0+}([0,T]\times\R)) \ni g^1\mapsto \int_0^t g^1(r,Z_r^1)dr\in \shc^{\P_{\bar\g}}$ admits a continuous extension to $A^{\P_{\bar\g}}:C_T\bar\shc_c^{(-\beta)+}\to \shc^{\P_{\bar\g}}$. 
  	\item[ii)] $\forall f^1\in C_c^{1,2+}([0,T]\times\R)$, 
  	\begin{equation} \label{eq:ii}
          M^{f^1}_t(Z^1,\bar\g)=f^1(t,Z_t^1)
          -\int_0^t (\partial_tf^1+ \frac{1}{2} \partial_{xx}^2 f^1)(r,Z_r^1)dr-A^{\P_{\bar\g}}(\shl^{1,\shb^{1,\bar\g}}f^1)
  	\end{equation}
  	is a $\P_{\bar\g}$- local martingale.
        
      \end{itemize}

  \end{proposition}

   \begin{proof}[Proof of Proposition \ref{pr:extendA1}]
     The set $\shn$ is the one selected by Lemmata \ref{lm:DLb} and \ref{lm:DLb}.
     We fix $\bar \gamma \notin \shn$.
     We start by proving item {\it i)}. Let $(g^n)$ be a sequence in $C_T\shc_c^{0+}(\R)$ converging to $0$ in $C_T\shc^{(-\beta)+}(\R)$.
           Take now the sequence  of functions $(f^n)$
in $C_T\shc^{(1+\beta)}(\R^d)$ solutions to $$\shl f^n=g^n,\quad f^n(T)=0.$$
By Remark 4.18 following Lemma 4.17 of \cite{issoglio_russoPDEa},
$f^n$ converge to zero in $C_T\shc^{(1+\beta)}(\R^d)$,
in particular uniformly in $(t,x)$.
Therefore for every   $\bar \gamma \in \Omega^{d-1}_T$,
we have that $f^n(t,Z_t^1,\bar\g_t)\to 0$ u.c.p.
By Lemmata  \ref{lm:DLb} and  \ref{lm:L}, 
with $g=0$ and $f_T=0$,  reminding \eqref{eq:DLb}
$$M^{f^n}(Z^1,\bar\g) = f^n(t,Z^1_t,\bar \gamma_t) - \int_0^t g^n(r,Z^1_r) dr   \to 0,$$
$\P_{\bar \gamma}$-u.c.p.,  thus, by additivity,  $\int_0^t g^n(r,Z_r^1)dr\to 0$ $\P_{\bar \gamma}$-u.c.p.
Therefore, the map $A^{\P_{\bar\g}}$ is well-defined for every $\bar \gamma \in
\shn$.

  	For what concerns item {\it ii)}, let $f^1\in C_c^{1,2+}([0,T]\times \R)$. We first calculate 
  	\begin{equation} \label{eq:heat-eq}
          (\shl f^1)(t,x_1,\ldots,x_d)=\partial_t f^1(t,x^1)+
       \frac{1}{2}\partial_{x_1x_1}^2 f^1(t,x_1)+ (\partial_{x_1}f^1(t,x_1)\ell(t,x_1,\ldots,x_d)) \cdot T^1_{h^1(t)}(x_1,\ldots,x_d),
  	\end{equation}
taking into account the associativity property of the pointwise product, see Remark 2.7 in \cite{BIRMcKean}.
Since  $C_T\shc^{0+}(\R)$ is dense in $C_T\shc^{(-\beta)+}(\R)$ by 
Remark \ref{rmk:density},  we can consider
a sequence $(h^{1,n})$  in $C_T\shc_c^{0+}(\R)$ converging to $h^1$ in $C_T\shc^{(-\beta)+}(\R)$. We set 
\begin{align} \label{eq:g-reg} 
          g^n(t,x_1\ldots,x_d)&= \partial_t f^1(t, x_1)+
         \frac{1}{2}\partial_{x_1x_1}^2 f^1(t,x_1) + \partial_{x_1}f^1(t,x_1) \ell(t,x_1,\ldots,x_d) h^{1,n}(t,x_1)  \\
      \notag   &= \partial_t f^1(t, x_1)+ \frac{1}{2} \partial_{x_1x_1}^2 f^1(t,x_1)+
      \partial_{x_1}f^1(t,x_1) \ell(t,x_1,\ldots,x_d) T^1_{h^{1,n}(t)}(x_1,\ldots,x_d),
\end{align}
  	which is an element of $C_T\shc_c^{0+}(\R^d)$.
        Now, for every $n\in\N$,  we consider the unique solution $f^n$ to 
  	\begin{equation}\label{eq:PDEf1}
  		\shl f^n=g^n,\ f^n(T,x_1,\ldots,x_d)=f^1(T,x_1).
  	\end{equation}
Now we show that $g^n\to g:=\shl f^1$ in  $C_T\shc^{(-\beta)+}(\R^d)$. We  observe that since $h^{1,n}$ converges in $C_T\shc^{(-\beta)+}(\R)$ there exists an $\al<\be$ such that the convergence holds in $C_T\shc^{-\al}(\R).$ From this, by the linearity and isometry property of $T$, see Proposition \ref{pr:besovT}, together with \eqref{eq:bonyt}, we get
  	\begin{align*}
  		\|g^n-g\|_{C_T\shc^{-\al}(\R^d)}=&\|(T^1_{h^{1,n}}-T_{h^1})\cdot \ell\,\partial_{x_1}f^1\|_{C_T\shc^{-\al}(\R^d)}\leq \|(T^1_{h^{1,n}}-T^1_{h^1})\|_{C_T\shc^{-\al}(\R^d)}\| \ell\,\partial_{x_1}f^1\|_{C_T\shc^{\beta}(\R^d)}\\
  		=& \|h^{1,n}-h^1\|_{C_T\shc^{-\al}(\R)}\| \ell\,\partial_{x_1}f^1\|_{C_T\shc^{\beta}(\R^d)}\to 0.
  	\end{align*}

        Since $g^n\to g$ in $C_T\shc^{(-\beta)+}(\R^d)$ and $(f^n_T)$ is constantly $f^1_T$, 
        by Lemma \ref{lm:L}, we obtain that
$$M^{f^n}(Z^1,\bar\g) = f^n(t,Z^1_t,\bar \gamma_t) - \int_0^t g^n(r,Z^1_r, \bar \gamma_r) dr,$$    
 converges   $\P_{\bar\g}$-u.c.p. to a local martingale with quadratic variation given by 
         $\int_0^t |\nabla f^1|^2(r,Z^1_r)dr$.
By continuity on the coefficients for the PDE \eqref{eq:PDEf1}, see Remark 4.18 following Lemma 4.17 in \cite{issoglio_russoPDEa},
we have that $f^n\to f^1$ in $C_T\shc^{(1+\beta)}(\R^d)$, therefore in $C^ {0,1}([0,T]\times \R^d)$
and consequently $(f^n(t, Z^1_t,\bar \gamma_t)$  converges u.c.p. to  $f^1(t, Z^1_t)$.
  	Consequently by additivity we have that $\int_0^t g^n(t,Z_r^1,\bar\g_r)dr$ converges $\P_{\bar \g}$-u.c.p. to $A^{\P_{\bar \g}}(\tilde{g}^{\bar\g})$, where 
  	$$\tilde g^{\bar\g}(t,x_1)=\lim_{n\to\infty} g^n(t,x_1,\bar\g_t),$$
        in $C_T \shc^{(-\beta)+}$,
  	if it exists.
        Taking into account \eqref{eq:g-reg}, the convergence of $h^{1,n}$
        to $h$, the fact that $t \mapsto \ell(t,x_1,\bar \gamma_t)
        \in C_T\shc^\beta(\R^d)$
        and the continuity of the pointwise product
        previous limit exists and it equals
        $$\shl^{\shb^{1,\bar\g}}f^1(t,x_1):=
        \partial_t f^1(t, x_1)+ \frac{1}{2} \partial_{x_1x_1}^2 f^1(t,x_1)+
        h^1(t,x_1) \ell(t,x_1,\bar\g_t) \partial_{x_1}f^1(t,x_1),$$
        which is a well-defined element of $C_T{\bar \shc}^{(-\beta)+}(\R)$.
        Since $\partial_t f^1, \partial_{x_1x_1} f^1$ belong (trivially)
        to $C_T \shc_c^{0+}(\R^d)$, the result \eqref{eq:ii}
is established.
      \end{proof}

\begin{theorem}\label{thm:1drMP}
   Let $\P$ be the solution to rMP($\shb,\nu$).
   Then, for $\P^{\bar Z}(d\bar \gamma)$-a.e. 
   $\P_{\bar \gamma}$ defined in \eqref{eq:516}
   is the solution to rMP($\shb_1^{\bar \g} ,\nu^1$),
   where 
   \begin{equation}
   	\label{eq:nu1}
     \shb^{\bar \g}_1(t,x_1):= \shb_1(t,x_1,\bar \gamma_t),
   \end{equation}
   defined in \eqref{eq:shb1}.
 \end{theorem}
 
\begin{proof} We select the set $\P^{\bar Z}$-null set  $\shn$ introduced in  Proposition
  \ref{pr:extendA1}.
  For all $\bar\g\notin \shn$,  Proposition
  \ref{pr:extendA1}, allows to apply Theorem \ref{thm:extABis} to $\P_{\bar\g}$ with $A^{\P_{\bar \g}}$
  in place of $A^{\P}$. Since 3. and 1. are equivalent the result follows. 
\end{proof}

\section{The particle system}\label{sec:IPS}

\subsection{The well-posedness of the particle system}

In this section Assumption \ref{ass:beta} will be in force and we introduce some other assumptions which are structural for the purpose of this part.
\begin{assumption}\label{ass:K}
	$K:\R^d\to \R$ is a probability density function such that $K\in\shc^{\beta+}$.
\end{assumption}
\begin{assumption}\label{ass:F}
	The function $ F\in C^1(\R)$ such that $ F'$ is bounded and Lipschitz.
\end{assumption}
\begin{assumption}\label{ass:b}
	$b = (b_1, \ldots, b_d) \in C_T\shc^{(-\beta)+}(\R^d;\R^d)$.
\end{assumption}
 The first objective is to establish the well-posedness of our system of $N$ $d$-dimensional interacting particles, described by the  system of stochastic differential equations
\begin{equation}\label{eq:MK_mod}
	\left\{
		\begin{array}{l}
			Z_t^{i,N} = Z_0^{i,N} + \displaystyle{\int_0^t F\left( \frac{1}{N}\sum_{j=1}^N K(Z_s^{i,N}-Z_s^{j,N})\right) T^i_{b(s)}(Z_s^N) \, ds + W^i_t,} \\
			 Z_0^{i,N} \sim \nu,
		\end{array}
	\right.
\end{equation}
where $i=1\ldots,N$, $t \in [0, T]$, and $W^i_t$ represents a $d$-dimensional Brownian motion for each $i \in \{1, \dots, N\}$. The processes $\{W^i_t\}_{i=1}^N$ are
formally intended to be mutually independent. The initial positions $Z_0^{i,N}$ are formally independent distributed according to a probability measure $\nu \in \mathcal{P}(\R^d)$. The operators
$T^i, 1 \le i \le N$, are the ones defined in Definition \ref{def:T}, and by Proposition \ref{pr:besovT}, they act as an isometrical injection
from $C_T\shc^{(-\beta)+}(\R^d)$ to $C_T\shc^{(-\beta)+}(\R^{dN})$.
An important issue is that the system \eqref{eq:MK_mod} has to be considered in law, 
as $dN$-dimensional rough martingale problem,
rMP($\shb^N, \nu^{\otimes N}$),
where 
\begin{equation}\label{eq:BN}
	\shb^N(t,x):=\left(\begin{array}{c}F\left(\frac{1}{N}\sum_{j=1}^NK(x_1-x_j)\right)T^1_{b(t)}(x)\\ \vdots\\ F\left(\frac{1}{N}\sum_{j=1}^NK(x_N-x_j)\right)T^N_{b(t)}(x)\end{array}\right),
\end{equation}
where $x=(x_1,\dots,x_N)\in\R^{dN}$. In Proposition \ref{pr:BNK} below, we observe that this is a well-defined element of $C_T\shc^{(-\beta)+}(\R^{dN};\R^{dN})$ for every $N\in\N$. We will denote the $i,k$ component of $\shb^N$ by $\shb^N_{i,k}$, namely by $\shb^N_{i,k}(t,x)=F\left(\frac{1}{N}\sum_{j=1}^NK(x_i-x_j)\right)T^i_{b_k(t)}(x)$ for all $i=1,\ldots,N$ and $k=1,\ldots,d$.
 In fact the maps 
\begin{equation} \label{eq:FnK}
	F^K_i: (x_1, \dots, x_N) \mapsto F\left(\frac{1}{N} \sum_{j=1}^N K(x_i - x_j)\right),
\end{equation}
belong  to $\shc^\beta(\R^{dN})$ for all $k \in \{1, \dots, N\}$, see Lemma  \ref{lm:FK} below.

\begin{lemma} \label{lm:FK}
	For each $ 1 \le k \le N $, the function $F^K_k$ belongs to $\shc^{\beta}(\R^{dN})$ and
        and there is a constant $C_{F,K}$, not depending on $N$ such that
	$$
	\Vert F^K_k \Vert_{\shc^{\beta}}  \le
	C_{F,K}: = \vert F(0) \vert + \sup_{x\in\R^{dN}} \vert F'(x) \vert \Vert K \Vert_{\shc^\beta}.
	$$
	
\end{lemma}
\begin{proof}
	We use the definition of $\|\cdot\|_{\shc^\beta}$ to determine the upper bound in \eqref{eq:FnK}. We have from \eqref{eq:holdnorm} that
	\begin{equation*}
		\left\|F^K_k
		\right\|_{\shc^\beta}
		=\left\|F^K_k
\right\|_{\infty}+\left\|F^K_k
\right\|_{\beta}=:I_1+I_2.
	\end{equation*}
	We estimate the two separately. For $I_1$ we get
	\begin{align*}
		\sup_{x\in \R^{dN}}	\left\vert F\left(\frac{1}{N} \sum_{j=1}^N K(x_k - x_j)\right)\right\vert\leq |F(0)|+\sup_{x\in\R^{dN}} |F'(x)|\left\|K\right\|_{\infty}.
	\end{align*}
	For what concerns $I_2$,  we have
	\begin{align*}
		I_2=& \sup_{x,y\in\R^{dN}}\frac{\left|F\left(\frac{1}{N} \sum_{j=1}^N K(x_k - x_j)\right)-F\left(\frac{1}{N} \sum_{j=1}^N K(y_k - y_j)\right)\right|}{|x-y|^{\beta}}\\
		\leq& \sup_{x\in\R^{dN}} |F'(x)| \frac{1}{N} \sum_{j=1}^N\sup_{x,y\in\R^{dN}}\frac{\left| K(x_k - x_j)- K(y_k - y_j)\right|}{|x-y|^{\beta}}\\
		=& \sup_{x\in\R^{dN}} |F'(x)| \| K\|_{\beta}.
	\end{align*}
	Adding up the two bounds obtained for $I_1,I_2$ we conclude.
\end{proof}

\begin{proposition}\label{pr:BNK} Let $\al < \be$ such that
	$b_k \in C_T\shc^{-\al}(\R^d), 1 \le k \le d$.
	
	There is a constant $C_{F,K}$ such that
	\begin{equation}\label{eq:BNK}
		\max_{1 \le i \le N,k=1,\dots,d} \|\shb_{i,k}^N\|_{C_T\shc^{-\al}(\R^{dN})}\leq C_{F,K}\max_{k =1,\dots,d}\|b_k \|_{C_T\shc^{-\al}(\R^d)}.
	\end{equation}
\end{proposition}

\begin{proof}
	Let $ 1 \le k \le d, 1 \le i \le N$. By \eqref{eq:bonyt} and Lemma \ref{lm:FK}
	we get
	\begin{align*}
		\Vert \shb^N_{i,k} \Vert_{C_T\shc^{-\al}(\R^{dN})} &=
		\sup_{0 \le t \le T}  \Vert F^K_i T^i_{b_k(t)} \Vert_{\shc^{-\al}(\R^{dN})} 
		\le
		c \Vert F^K_i \Vert_{\shc^{\beta}(\R^{dN})} \Vert
		T^i_{b_k}\Vert_{C_T\shc^{-\al}(\R^{dN})}\\ 
		&=  c \Vert F^K_i \Vert_{\shc^{\beta}(\R^{dN})} \Vert  b_k\Vert_{C_T\shc^{-\al}(\R^d) },
	\end{align*}
	where for the equality we have used Proposition \ref{pr:besovT}.
	This concludes the proof.
\end{proof}
Proposition \ref{pr:BNK} is used for proving 
well-posedness of  \eqref{eq:MK_mod} but it will be also useful in  Section
\ref{sec:tightness}
to establish the tightness of the particle system empirical measures, see \eqref{eq:EMmuN}.

 %
As expected from the preliminary lines of this section, well-posedness of \eqref{eq:vecMK} will be now an easy consequence.
 
 \begin{theorem} \label{thm:WPPS}
 	For every $N\in\N$ there exists a unique  solution $\P^N$ to the rMP($\shb^N,\nu^{\otimes N}$).
 \end{theorem}
 \begin{proof}
   Since $\shb^N\in C_T\shc^{(-\beta)+}(\R^{dN};\R^{dN})$ for every $N\in\N$, we use
   Proposition \ref{pr:rMP}  applied to \eqref{eq:vecMK} with $dN$ replacing $d$,  and we conclude.
 \end{proof}
Denoting the canonical process
on the canonical space $\Omega_T^{dN}$ by $Z^N:=(Z^{1,N},\ldots,Z^{N,N})$, we can rewrite the $N$-particle system in \eqref{eq:MK_mod} as the formal $dN$-dimensional SDE given by
 \begin{equation}\label{eq:vecMK}
	dZ_t^N =\shb^N(t,Z_t^N)dt+dW^N_t,\quad Z_0^N\sim\nu^{\otimes N}.
\end{equation}
       The next corollary, that formulates rigorously \eqref{eq:vecMK},  follows directly from Theorem \ref{thm:WPPS} and Corollary \ref{cr:Dynamics}.
\begin{corollary}\label{cr:Dynamicspart}
	For every $N\in\N$, let $\P^N$ be the solution to rMP($\shb^N,\nu^{\otimes N}$). Then the canonical process $Z^N$ satisfies the equation
	\begin{equation}
		Z^N_t=Z^N_0+A^{\P^N}(\shb^N)+W^N_t,
              \end{equation}
              $A^{\P^N}:C_T\shc^{(-\beta)+}(\R^{dN}) \rightarrow \shc^{\P^N}$ is the extended operator given by Proposition \ref{pr:extendA}
              replacing $d$ with $dN$.
\end{corollary}

The next result concerns the exchangeability
of the particle system.

\medskip

\begin{theorem} \label{th:exch}
Let $N \in \N^*$ be fixed and  $\P^N$ be the solution to rMP($\shb^N,\nu^{\otimes N}$).
If $(i_1,\ldots,i_N)$ is a permutation of $(1,\ldots,N)$, then
the law of $\hat P^N$ of $(Z^{i_1},\ldots,Z^{i_N})$ under $\P^N$ is equal to $\P^N$.
In other words $\P^N$ is exchangeable.
  \end{theorem}
  \begin{remark} \label{rmk:exch}
By definition of rough martingale problem, $\hat \P^N$ is solution of the system 
	\begin{equation}\label{eq:IPSexch}
		\left\{
		\begin{array}{l}
			dZ_t^{i_j}=F^K(Z^{i_1},\ldots,Z^{i_N}) \cdot T^{i_j}_{b(t)}(Z^N_t)dt+dW^{i_j}_t\\
			Z_0^{i_j}\sim \nu, \ j=1,\dots,N,
		\end{array}
		\right.
	\end{equation}
	and denote ${\hat\shb}^N$ the drift associated with the system \eqref{eq:IPSexch}.
        By the same argument of Proposition \ref{pr:BNK}, ${\hat \shb}^N \in C_T\shc^{(-\beta)+}(R^{dN})$, thus rMP$(\hat\shb^N,\nu^{\otimes N})$ admits a unique solution $\hat \P^N$, again
        by Theorem \ref{thm:WPPS}.
\end{remark}
We now prove an auxiliary result before proving Theorem \ref{th:exch}.
\begin{proposition}\label{pr:exch}
	$\hat \P^N$ is a solution to rMP$(\shb^N,\nu^{\otimes N})$.
\end{proposition}
\begin{proof}
  We show that the parabolic generator $\hat \shl$ associated to
  rMP$({\hat \shb}^N,\nu^{\otimes N})$ is equal to $\shl$,
  which is associated to
  rMP$(\shb^N,\nu^{\otimes N})$. Indeed, taking $d=1$ for simplicity,
  \begin{align*}
    \hat \shl f(t,x) =&\partial_t f(t,x) +\frac12 \Delta f(t,x)+\sum_{j=1}^N \partial_{x_{i_j}}f(t,x)T^{i_j}_{b(t)}(x)F\left(\frac1N\sum_{k=1}^N K(x_{i_j}-x_{i_k})\right)\\
	=&\partial_t f(t,x) +\frac12 \Delta f(t,x)+\sum_{j=1}^N \partial_{x_{i_j}}f(t,x) T^{i_j}_{b(t)}(x) F\left(\frac1N\sum_{k=1}^N K(x_{i_j}-x_k)\right)\\
 =   &\partial_t f(t,x) +\frac12 \Delta f(t,x)+\sum_{i=1}^N \partial_{x_i}f(t,x) T^{i}_{b(t)}(x) F\left(\frac1N\sum_{k=1}^N K(x_i-x_k)\right)\\
					=&\shl f(t,x).
	\end{align*}
        We can easily see that $\shd_{\hat\shl}=\shd_\shl$: indeed  given $f_T\in \bar\shc_c^{(1+\beta)+}$ and $g\in C_T\bar\shc^{0+}(\R^d)$, $f$ is a solution to $\shl f=g,\ f(T)=f_T$ if and only if $\hat\shl \hat f=\hat g, \hat f(T)=\hat f_T$, where
$\hat f(t,x) = f(t,x_{i_1}, \ldots, x_{i_d})$ and $\hat f_T$ is defined similarly.
        So the result follows taking into account Definition \ref{def:rMP}.
\end{proof}
\begin{proof}[Proof of Theorem \ref{th:exch}]
  By Proposition \ref{pr:exch}
  $\hat \P^N$ is a solution to rMP$(\hat\shb^N,\nu^{\otimes N})$.
Since uniqueness holds for that rough martingale problem 
then $\P^N=\hat\P^N$.
\end{proof}

\subsection{Tightness of the empirical measures}

\label{sec:tightness}

In this section the number $N$ of particles will change.
We remark that the canonical spaces $\Omega_T^{dN}$ and 
the canonical process $Z^N = (Z^{1,N}, \ldots, Z^{N,N})$
vary.
For each strictly positive integer $N$, let $\P^N$ be the solution of rMP($\shb^N,\nu^{\otimes N}$)
provided by Theorem \ref{thm:WPPS}.

We are interested a priori in the convergence
in law  of the empirical measures $\mu^N$, where
\begin{equation}\label{eq:EMmuN}
\mu^N_t:=\frac1N\sum_{i=0}^N\delta_{Z_t^{i,N}}.
\end{equation}
The global model, including all the infinite
number of particles, is defined
as an infinite product of $\Omega_T^d$.
Since we are interested in convergence in law of \eqref{eq:EMmuN},
we do not need
to specify the environment of the all all the particles, since the laws of $\mu^N$ are defined
on the same probability space $\shp(\Omega_T^d)$.
Indeed,  each particle $Z^{j,N}$ is defined on $\Omega_T^{dN}$ with
values in $\Omega_T^{d}$. Consequently $\delta_{Z^{j,N}}$ is
still defined  on $\Omega_T^{dN}$, 
taking values
in $\shp(\Omega_T^d)$, as well as its empirical mean.

In this paper, we only prove the tightness of the laws of the empirical measures $(\mu^N)$
as probability measures on $\shp(\Omega_T^d)$ and we leave the rest to future research.
A crucial result for this is Proposition 2.2 of \cite{Sznitman},
which states that, the sequence of laws of $(\mu^N)$ is tight if and only of
the sequence of their intensities is tight, see (2.4) of \cite{Sznitman}.
For a given $N$, the intensity of each law of $\mu^N$ is the measure
\begin{equation}\label{eq:Intensity}
  G \mapsto  \frac{1}{N} \E^{\P^N}\left[\sum_{j= 1}^N G(Z^{j,N})\right],
  \end{equation}
where $G$ is a bounded Borel functional on $\Omega_T^d$.
By Theorem \ref{th:exch}, all the components $Z^{j,N}$
are identically distributed as $Z^{1,N}$,
so that \eqref{eq:Intensity} becomes
\begin{equation}\label{eq:IntensityBis}
  G \mapsto   \E^{\P^N}\left[G(Z^{1,N})\right],
  \end{equation}
which is the law of $Z^{1,N}$ under $\P^N$.
Clearly, the disintegration formula \eqref{eq:516} applies in this context too,
by substituting $1$ (resp. $d$) with $d$ (resp. $N$). So, we get
\begin{equation} \label{eq:516Bis} 
  \P^N(d\g^1,d\bar\g)
	= \int_{\Omega_T^{(N-1)d}}\P^{\bar Z}(d\bar\g)\int_{\Omega^d_T}\P_{\bar \g}(d\g^1),
\end{equation}
where $\g^1\in\Omega^d_T$ and  $\bar \g:=(\g^2,\ldots,\g^N)\in\Omega_T^{(N-1)d}$.
$(\P_{\bar \g}(d\g^1), \bar \gamma \in \Omega_T^{(N-1)d})$ is the usual random kernel.
Formally speaking, $\P_{\bar \gamma}(d\gamma^1)$ is the conditional
law of the $d$-dimensional block $Z^1$ given $(Z^2,\ldots,Z^d) = (\gamma^2, \ldots, \gamma^d)$.
In \eqref{eq:516}, for simplicity, our $N$ was $d$ and our $d$ was $1$.

We denote by $\bar Z^N$  the vector $(\bar Z^{2,N}, \ldots, \bar Z^{N,N})$ and $\P^{N, \bar Z}$  the law of $\bar Z^N$ under $\P^N$.
For $\bar \g=(\g^2,\ldots,\g^N)$,
we denote
\begin{equation}\label{eq:b11}
  \shb^{N,\bar \g}_1(t,x_1)=F_1^K(x_1,\bar\g(t))\cdot b(t,x_1)\in C_T\shc^{(-\beta)+}(\R^{d};\R^{d}).
  \end{equation}
  This according to Remark \ref{rmk:shb1} applied substituting $1$ with $d$, 
$\ell = F^K_1$, $h^i = b$.
\begin{proposition}\label{pr:1drMPpart}
  Let $\P^N$ be the solution to rMP($\shb^N,\nu^{\otimes N}$). Then there is a $\P^{\bar Z^N}$
  null set $\shn$, such that
for all $\bar \gamma \notin \shn$,
$\P^N_{\bar\g}$ is the solution to rMP($\shb^{N,\bar\g}_1,\nu$).
\end{proposition}
\begin{proof}

  It is a direct adaptation of  Theorem \ref{thm:1drMP}
taking into account that there $N$ was $d$ and our $d$ was $1$,
with $\shb_1^{\bar\g}=\shb^{N,\bar\g}_{1}$, 
 defined in \eqref{eq:b11}.
\end{proof}

\begin{theorem} \label{thm:MomentsCond}
  Let $(\P^N)$ be the solution to rMP($\shb^N,\nu^{\otimes N} $). Then, under $\P^N$,
  the sequence of laws of $Z^{1,N}$ is tight.
\end{theorem}
\begin{proof}
  Since all the laws have the same initial
  probability law $\nu$, Theorem 4.10, Chapter 2 in \cite{karatzasShreve},
  taking into account the usual arguments such as Garsia-Rudemich-Rumsey
  (see e.g. Lemma 3.12 in \cite{issoglio_russoMPb}),
  says that we need to find $\ell, \delta >0$, so that
  \begin{equation}\label{eq:momentsT}
    \E^{\P^N}[|Z^{1,N}_t-Z^{1,N}_s|^{\ell}] \le {\rm const}|t-s|^{1+\delta},
  \end{equation}
  where we insist on the fact that ${\rm const}$ does not  depend on $N$,
see also Problem 4.11, Chapter 2 of \cite{karatzasShreve}. 

  Let us consider $p\geq 1$ and estimate
\[
\E^{\P^N}[|Z^{1,N}_t-Z^{1,N}_s|^p].
\]
Now we take $\g\not\in \shn$ a null set under $\P^{\bar Z^N}$, given by Lemma \ref{lm:DLb}, and by disintegration, using \eqref{eq:516Bis} 
we get
\[
\E^{\P^N}[|Z^{1,N}_t-Z^{1,N}_s|^p]=	
	\int_{\Omega_T^{d-1}}\P^{\bar Z}(d\bar \g)	\E^{\P_{\bar\g}}[|Z^{1,N}_t-Z^{1,N}_s|^p]
\]
By Proposition \ref{pr:1drMPpart} $\P^N_{\bar\g}$ solves the rMP($B^{1,N}_\g,\nu$), then thanks to Theorem \ref{thm:LPestim} applied to  $\P^N_{\bar\g}$, together with Proposition \ref{pr:BNK}, for every $0 < \lam < \frac{\beta}{2}$ we have
\begin{equation}\label{eq:LPIPS}
\E^{\P_{\bar\g}^N}[|Z^{1,N}_t-Z^{1,N}_s|^p]
\leq R_{\al,\beta}(d,T,\|\shb^{1,N}_{\bar\g}\|_{C_T\shc^{-\al}})|t-s|^{\lam p}\leq R_{\al,\beta}(d,T,C_{F,K}\|b\|_{C_T\shc^{-\al}})|t-s|^{\lam p},
\end{equation}
Consequently \eqref{eq:momentsT} is proved
choosing $p > \frac{2}{\beta}$.

\end{proof}
\begin{corollary}\label{cor:MomentsCond}
Let $\P^N$ be the sequence of solutions to rMP($\shb^N,\nu^{\otimes N}$). Then the sequence of empirical measures $(\mu^N_t)$ is tight.
\end{corollary}
\begin{proof}
By Theorem \ref{pr:1drMPpart} the sequence $\shl$aw($Z^{1,N}$) is tight. The result follows by the considerations before \eqref{eq:Intensity}.
\end{proof}

\appendix

	\section*{Appendices}

\section{Auxiliary analytical results.}

\label{App:A}

First we show a a natural, however not obvious, result that extends a function $f$ in $C_b^{1,2+}([0,T]\times \R^d)$ to a function $\hat f$ which belong to $C_b^{1,2+}(\R_+\times \R^d)$.
\begin{lemma}\label{lm:LA}
	Let $f:[0,T]\times \R^d \to \R$ be of class $C_b^{1,2+}([0,T]\times \R^d)$. Then there exists an extension $\hat f:\R_+\times \R\to\R$ which belongs to $C_b^{1,2+}(\R_+\times \R^d)$.
\end{lemma}
\begin{remark} \label{rmk:C12+}
  We highlight that the space $C_b^{1,2+}([0,T]\times \R^d)$
was  defined in Section \ref{sc:basic-prelim}.
\end{remark}
\begin{proof}[Proof of Lemma \ref{lm:LA}]
	Without restriction of generality we can assume $T=1$. Now we show we can extend $f$ up to $[0,\frac32]\times \R^d$ and from there one iterates the procedure. We define $\hat f:[0,\frac32]\times \R^d\to \R$ as 
	\[
	\hat f(t,x):=\begin{cases}
		f(t,x) &\text{ if } (t,x)\in[0,1]\times\R^d\\
		3f(1,x)-3f(2-t,x)+f(3-2t,x)& \text{ if }(t,x)\in (1,\frac32]\times \R^d.
	\end{cases}
	\]
	Obviously $\hat f$ is continuous.
   Clearly $\hat f$ is of class
$C_b^{1,2+}([0,1)\times \R^d)\cap C_b^{1,2+}((1,\frac32]\times \R^d)$,
Moreover,  concerning, we get
	\[
	\partial_t \hat f (t,x)=\begin{cases}
		\partial_t f(t,x)&\text{ if } (t,x)\in[0,1)\times\R^d\\
		3\partial_t f(2-t,x)-2\partial_tf(3-2t,x)& \text{ if }(t,x)\in (1,\frac32]\times \R^d.
	\end{cases}
	\]
	Computing the limit as $t$ approaches $1$ from both left and right 
		\[
	\lim_{t\to1^-}\partial_t\hat f=\partial_t f(1-,x),\quad
	\lim_{t\to1^+}\partial_t\hat f=\partial_t f(1-,x).
	\]
	Since those limits coincide and $\hat f$ is continuous we conclude that
        $\partial_t \hat f$ exists at on the axis $\{1\} \times \R^d$ and
      it is continuous.
For the same reasons that $\hat f$ is continuous, we can show that $\nabla \hat f$ and ${\rm Hess}(\hat f)$ extend continuously to $\{1\} \times \R^d$ 
    and the extension is H\"older continuous.  
\end{proof}

	\begin{lemma}\label{lm:chi_n}
				Let $\chi\in C^{\infty}(\R) $  be a compactly supported real-valued function  such that 
			\begin{equation}
				\chi(\rho)=\begin{cases}
					1\ \text{if } \rho\leq 0\\
					0\ \text {if }\rho \geq 1
				\end{cases}
			\end{equation}
			and for every $n \ge 1$, let us define $\chi_n(x):=\chi\displaystyle{\left( \frac{|x|-n}{\sqrt{n}} \right)}$.
			Then
			\begin{align}
			\label{eq:11}	& (\|\chi_n\|_{\shc^{1+\beta}})  \ {\rm is \ a \ bounded \ sequence}, \\
			\label{eq:12}	&\nabla \chi_n\to 0 \text{ in } \shc^{\beta+},\\
			\label{eq:13}	&\Delta\chi_n\to 0 \text{ in } \shc^{(-\beta)+}.
			\end{align}
		\end{lemma}
		\begin{proof}
                  We start by calculating both $\nabla \chi_n$ and ${\rm Hess} \chi_n$.
                  Obviously, for $\vert x \vert \le n$, those functions do vanish.
                  For $x > n$, we get
			\begin{align*}
				&\nabla \chi_n(x)=\frac{x}{|x|\sqrt{n}}\chi'\left(\frac{|x|-n)}{\sqrt{n}}\right),\\
				&{\rm Hess}\chi_n(x)=\frac{1}{|x|\sqrt{n}}\left({\rm Id}-\frac{x\cdot x^\top}{|x|^2}\right)\chi'\left(\frac{|x|-n}{\sqrt{n}}\right)+\frac{1}{n}\frac{x\cdot x^\top}{|x|^2}\chi''\left(\frac{|x|-n}{\sqrt{n}}\right).
			\end{align*}
		We observe that 
			\begin{equation} \label{eq:nablan}
				\|\nabla \chi_n\|_{\infty}=\frac{1}{\sqrt{n}}\|\chi'\|_{\infty}\to 0,\ \text{ as }n\to\infty
                              \end{equation}
                              and
\begin{equation} \label{eq:Hessn}
                         \|{\rm Hess}\chi_n \|_{\infty}  
                         \leq \frac{\|\chi'\|_{\infty}}{\sqrt{n}}+ 2 \frac{\|\chi'\|_{\infty}}{n \sqrt{n}}+\frac{\|\chi''\|_\infty}{n}\to 0, \text{ as }n\to \infty,
                         	\end{equation}
	where the latter inequality comes from triangle inequality, from the fact that $\|{\rm Id}-\frac{x\cdot x^\top}{|x|^2}\|_{\infty}$ and $\|\frac{x\cdot x^\top}{|x|^2}\|_{\infty}$
                        are bounded by $1$ together with the fact that  $\chi''$ is possibly non-zero for $|x|\geq n$, hence $\displaystyle{\frac{1}{|x|}}\leq \frac{1}{n}$.

                                Let us consider $\al>\be$. Owing to Remark \ref{rm:equivnorm}, we calculate the $\shc^{\al}$-norm of $\nabla \chi_n$, and one gets 
			\begin{align} \label{eq:chialpha}
                          \|\nabla \chi_n\|_{\shc^\al}&=\|\nabla \chi_n\|_{\infty}+\sup_{|x-y|<1,x\not=y}\frac{|\nabla \chi_n(x)-\nabla\chi_n(y)|}{|x-y|^\al}\leq \|\nabla \chi_n\|_{\infty}
                            + \|{\rm Hess}\chi_n \|_{\infty} \sup_{|x-y|<1}  |x-y|^{1-\al} \\
                                \notag                      &\leq \|\nabla \chi_n\|_{\infty}+  \|{\rm Hess}\chi_n\|_{\infty} \to  0, \text{ as }n\to \infty,
			\end{align}
because of \eqref{eq:nablan} and \eqref{eq:Hessn}. This proves \eqref{eq:12}.
                        
                        
			Since the norm of $\shc^{\alpha}$ is dominated by the one in $\shc^{\beta}$, one has
			\begin{equation*}
                          \|\chi_n\|_{\shc^{1+\beta}} =\|\chi_n\|_{\infty}+ \|\nabla \chi_n\|_{\shc^\beta} \le
\|\chi_n\|_{\infty}+ \|\nabla \chi_n\|_{\shc^\alpha} \le
                          1 + \sup_n (\|\nabla \chi_n\|_{\shc^\alpha}),
                        \end{equation*}
                        which is bounded because of \eqref{eq:chialpha}. This proves \eqref{eq:11}.

			Concerning \eqref{eq:13}, we recall that
			Proposition 2.1 with $p=\infty$ and Proposition 2.3 with $s=-\g>-\be$, $\delta=\g$ and $p=q_1=q_2=\infty$  of Section 2.1.2 in \cite{sawano} imply that $L^{\infty}\hookrightarrow \shc^{-\g}$.
                       Consequently, it is enough to prove that $\|\Delta \chi_n\|_{\infty}\to 0.$
                        Now
                        $$ \Vert \Delta \chi_n \Vert_\infty = \Vert {\rm Tr}({\rm Hess}\chi_n) \Vert_\infty
                        \le \Vert {\rm Hess}\chi_n) \Vert_\infty \to 0,$$
because of \eqref{eq:Hessn}. This concludes the proof.

		\end{proof}

                \begin{remark} \label{rm:kApprox}
                  \begin{enumerate}
\item         We remark that the
         constant function $1$, which is an element of $\shc^{(1+\beta)+}$,
         does not belong to the $\bar \shc^{(1+\beta)+}$.
         Indeed, it cannot be a uniform limit of continuous functions with compact support.
  \item        Nevertheless, an immediate consequence of  Lemma \ref{lm:chi_n} is
          that $1 \in \bar D \shc^{\beta+}$.
\item Analogously to  $ \bar \shc^{(1+\beta)+}$, we define  $ \bar D \shc^{\beta+}$ as the closure
  of the space of functions in $  D \shc^{\beta+}$ with compact support.
\end{enumerate}
        \end{remark}

\section{Technical proofs of Section \ref{sc:prelim}}

\label{app:B}

        \subsection{Proof of Lemma \ref{lm:schauderDC}}
\label{app:B1}
         

Let $f\in D\shc^\g$.
We start with a preliminary estimate.
If $W$ is a  standard Brownian motion, we obtain
\begin{align}\label{3.4}
 \notag |P_tf(0) -  f(0)| &= |\E[f(W_t) - f(0)]|  
\le \E[|f(W_t) - f(0)|] \le \|\nabla f\|_{\infty} \E[|W_t |] \\
\notag           &\le \|\nabla f\|_{ \g } \E[|W_t|] = \sqrt{\frac{2t}{\pi}} \|\nabla f\|_{\g}  \\
  &\le     \sqrt{\frac{2t}{\pi}} \| f\|_{D \shc^{\g}}.                
\end{align}
We  prove now \eqref{eq:schaud1DC}. \eqref{3.4} yields
\begin{equation} \label{3.4Bis}
  |P_tf(0)| \le  |P_tf(0) -  f(0)| + \vert f(0) \vert \le \left(\sqrt{\frac{2t}{\pi}}+1\right) \|
  f\|_{D \shc^{\g}}.
\end{equation}
By definition we have
\begin{equation} \label{eq:eval-normGamma}
\|P_tf\|_{D\shc^{2\theta+\g}}=|P_tf(0)|+\|\nabla P_tf\|_{\shc^{2\theta+\g}}.
\end{equation}
The first term is estimated by \eqref{3.4Bis}.
For what concerns the second term of \eqref{eq:eval-normGamma},
we recall that $\nabla$ and $P_t$ commute, by Remark 2.3 of \cite{issoglio_russoPDEa}.
By  \eqref{eq:schauder1} in Lemma \ref{lm:schauderDC}, we get
\begin{equation}\label{3.4Ter}
\Vert \nabla P_tf\|_{\shc^{2\theta+\g}}=\| P_t\nabla f\|_{\shc^{2\theta+\g}}\le c t^{-\theta}\|\nabla f\|_{\shc^\g}
\le c t^{-\theta}\|f\|_{D \shc^\g}.
\end{equation}
Adding \eqref{3.4Bis} and \eqref{3.4Ter} one can establish \eqref{eq:schaud1DC},
where $C_T:= \max \left\{c, 1+\sqrt{\frac{2T}{\pi}}\right\}$.

We go on with the proof of \eqref{eq:schaud2DC}. Let now 
$f \in D\shc^{\g + 2\theta}$.
Again, by definition we write
 \begin{equation} \label{eq:Quater}
   \|P_tf-f\|_{D\shc^\g} =  |  P_tf(0) -  f(0)|+	
   \|\nabla P_tf - \nabla f\|_{\shc^\g}, t \in [0,T].
 \end{equation}
Replacing $\gamma$ with $2\theta + \gamma$ in \eqref{3.4}, we get
 \begin{equation}\label{3.4Quinque}
   |P_tf(0) -  f(0)|
  \le     \sqrt{\frac{2t}{\pi}} \|  f\|_{D \shc^{2 \theta +\g}}.                
\end{equation}
Taking into account \eqref{eq:schauder2} in Lemma \ref{lm:schauderDC}, the second term in \eqref{eq:Quater} gives
\begin{equation}\label{3.5}
  \|\nabla P_tf - \nabla  f\|_{\shc^\g}  =
  \| P_t\nabla f - \nabla  f\|_{\shc^\g}    \le c t^{\theta} \|\nabla f \|_{\shc^{\g+2\theta}}
 \le c t^{\theta} \|f \|_{D \shc^{\g+2\theta}}.
\end{equation}

Adding together \eqref{3.4Quinque} and \eqref{3.5}, one gets \eqref{eq:schaud2DC}.


\subsection{Proof of Lemma \ref{lm:issoglio}}
\label{app:B2}


\medskip
We observe that for all $\forall\  0\le s < t\le T$ we have that  $P_{s-t}h(s)\in\shc^{1+\beta}$ thanks
to \eqref{eq:schauder1} in Lemma \ref{lm:schauder}. Let $\eps>0$.
Let $\alpha < \beta$ such that $h \in C_T\shc^{-\alpha}$ 
and $\nu>0$ such that $\theta:=\frac{1+\al+\be+2\nu}2<1$. We have

        	\begin{align*}
		\left\| \int_{t+\eps}^T P_{s-(t+\eps)}h(s)ds\right.&\left.-\int_t^T P_{s-t}h(s)ds\right\|_{\shc^{1+\beta}}\\
		&\le \left\| \int_{t+\eps}^T \left( P_{s-(t+\eps)} h(s) - P_{s-t} h(s) \right) ds \right\|_{\shc^{1+\beta}} + \left\| \int_t^{t+\eps} P_{s-t} h(s) \, ds \right\|_{\shc^{1+\beta}} \\
           &\le \int_{t+\eps}^T \| P_{s-(t+\eps)} h(s) - P_{s-t} h(s) \|_{\shc^{1+\beta}} \, ds + \int_t^{t+\eps} \| P_{s-t} h(s) \|_{\shc^{1+\beta}} \,  \\
          &= \int_{t+\eps}^T \| P_{s-(t+\eps)}(I-P_\eps)h(s) \|_{\shc^{1+\beta}} \, ds + \int_t^{t+\eps} \| P_{s-t} h(s) \|_{\shc^{1+\beta}} \, ds \\       
		&\le c\left( \int_{t+\eps}^T (s-(t+\eps))^{-\theta}  \| h(s)-P_\eps h(s) \|_{\shc^{-(\alpha+2\nu)}}  ds + \int_t^{t+\eps} (s-t)^{-\frac{1+\al+\beta}2} \| h(s) \|_{\shc^{-\alpha}} \, ds\right) \\
		&\le c\left( \int_{t+\eps}^T (s-(t+\eps))^{-\theta} |\eps|^\nu \| h(s) \|_{\shc^{-\alpha}} \, ds + \int_t^{t+\eps} (s-t)^{-\frac{1+\al+\beta}2} \, ds \, \| h \|_{B_T \shc^{-\alpha}} \right)\\
     &\le \| h \|_{B_T \shc^{-\alpha}} c |\eps|^\nu \left( \int_{t+\eps}^T (s-(t+\eps))^{-\theta} \, ds +
   \int_t^{t+\eps} (s-t)^{-\frac{1+\al+\beta}2} \, ds \right)\\
        &= \| h \|_{B_T \shc^{-\alpha}} c  \left(|\eps|^\nu \frac{(T - t - \eps)^{1-\theta}}{1-\theta} + \frac{2\eps^{\frac{1-\al-\beta}2}}{1-\al-\beta} \right) \to 0 \quad \ \text{ as } \eps \downarrow 0. 
	\end{align*}
        At the fifth line we have used \eqref{eq:schauder1} with $\gamma = -(\alpha + 2 \nu)$ so that $\gamma + 2\theta = 1+\beta$.
        At the sixth line, we have performed \eqref{eq:schauder2}
        with $\theta = \nu, \gamma = - \alpha - 2 \nu$
        so that $\gamma + 2 \theta = -\alpha$.

        This concludes the proof.

        
        \section{Technical complements about the Kolmogorov PDE}

        \label{app:C}

This section focuses on some technical aspects  of the
        Kolmogorov PDE \eqref{eq:KolmPDE}, keeping in ùind
the parabolic generator  $\mathcal L$ 
given by \eqref{eq:shl}.
        
\subsection{Proof of Theorem \ref{thm:unDC}}

\label{app:C1}

We make use of \eqref{eq:mild}.
 we define  $ v(t):= u(t) - P_{T-t} u_T, t \in [0,T].$ Of course we have
              	\begin{equation}\label{eq:vmild}
		v(t) = \int_t^T P_{s-t} \left(\nabla u(s)\cdot B(s) \right) d s - \int_t^T P_{s-t}  g(s)d s, \ \forall t\in[0,T],
	\end{equation}
	thus we get immediately that $v\in C_T\shc^{1+\beta}$, since both 
\begin{equation} \label{eq:1+beta}
        \int_t^T P_{s-t}(\nabla u(s)\cdot B(s))ds, \  \int_t^T P_{s-t}g(s) ds \in
	C_T\shc^{1+\beta},
\end{equation}
        thanks to Lemma \ref{lm:issoglio} applied to $\nabla u\cdot B$ and $g$.
        
We proceed proving {\it item i)}. We set $\eta:=\frac{1+\al+\beta}{2}<1$.
        Lemma \ref{lm:schauder}, with $\theta = \eta$, and \eqref{eq:bony} yield
	\begin{align}\label{eq:Auxil}
	\notag	\| v(t)\|_{\shc^{1+\beta}} 
		\leq& \left\|\int_t^T   P_{s-t}(\nabla u(s)\cdot B(s))ds+\int_t^T  P_{s-t}g(s)ds\right\|_{\shc^{1+\beta}}\\
	\notag	\leq &  c \int_t^T (s-t)^{-\eta}\|\nabla u(s)\|_{\shc^{\beta}}\|B(s)\|_{\shc^{-\al}}ds+c\int_t^T (s-t)^{-\eta}\|g(s)\|_{\shc^{-\al}}ds\\
	\notag	\leq &  c \int_t^T (s-t)^{-\eta}\|\nabla u(s)\|_{\shc^{\beta}}\|B(s)\|_{\shc^{-\al}}ds+c\|g\|_{C_T\shc^{-\al}}\frac{T^{1-\eta}}{1-\eta} \\
	\notag	\leq &  c \|B\|_{C_T\shc^{-\al}}\int_t^T (s-t)^{-\eta}\|\nabla u(s)\|_{\shc^{\beta}}ds+c\frac{T^{1-\eta}}{1-\eta}    \|g\|_{C_T\shc^{-\al}}\\
		\leq &  c \|B\|_{C_T\shc^{-\al}}\int_t^T (s-t)^{-\eta}\| u(s)\|_{ D\shc^{\beta}}ds+c \frac{T^{1-\eta}}{1-\eta}   \|g\|_{C_T\shc^{-\al}}.
	\end{align}
By  \eqref{eq:schaud1DC} of
Lemma \ref{lm:schauderDC} with $\gamma = \beta, \theta  = 0$,
recalling $\|\cdot\|_{ D\shc^{\beta}}\leq\|\cdot\|_{\shc^{1+\beta}}$ and
taking into account
 \eqref{eq:Auxil},
we obtain
\begin{align}
  \label{eq:uDCal}
	\notag	\| u(t)\|_{D\shc^\beta}
         	& \le \| P_{T-t}u_T\|_{D\shc^\beta}+ \Vert v(t)\Vert_{D\shc^\beta} \\
             		& \le C_T\| u_T\|_{ D\shc^\beta}+ \Vert v(t)\Vert_{{\shc^{1+\beta}}} \\
  \notag	& \leq  C_T\| u_T\|_{ D\shc^\beta}+ c \|B\|_{C_T\shc^{-\al}}
      \int_t^T (s-t)^{-\eta}\| u(s)\|_{ D\shc^{\beta}}ds+c \frac{T^{1-\eta}}{1-\eta}   \|g\|_{C_T\shc^{-\al}}.
	\end{align}
	Then,  by using a generalised Gronwall's inequality \cite{ye} and taking the $\sup_{t \in[0,T]}$, we obtain
	\begin{align*}
		\notag\| u\|_{C_TD\shc^\beta}\leq& \left(c \frac{T^{1-\eta}}{1-\eta}  \|g\|_{C_T\shc^{-\al}}+C_T\| u_T\|_{ D\shc^\beta}\right)E_{1-\eta}(\Gamma(1-\eta)T^{1-\eta} c \|B\|_{C_T\shc^{-\al}})\\
		=&R_{\al,\beta}(T,\|B\|_{C_T\shc^{-\al}},\|g\|_{C_T\shc^{-\al}},\|u_T\|_{ D\shc^\beta}),
	\end{align*}
        where
        \begin{equation} \label{eq:Ralpha}
          R_{\al,\beta}(T, x, y, z):= \left(c \frac{T^{1-\eta}}{1-\eta} y + C_T z\right)
          E_{1-\eta}(\Gamma(1-\eta)
          T^{1-\eta} c x),
          \end{equation}
	in which $\Gamma(x)$ is the Euler Gamma function and $E_{1-\eta}(x)$ is the Mittag-Leffler function of parameter $1-\eta$. Thus, we have proven \eqref{eq:nablaunDC}, i.e. item i).

  Concerning items (ii) and (iii),   we proceed now in two steps.
  \begin{itemize}
  \item[1)]
Let $0 < \eps <  \beta$.
    We first prove, for $0 \le s \le t \le T$, that
    \begin{equation}\label{eq:holdvn}
          \|v (t,\cdot)-v  (s,\cdot)\|_{\shc^{1+\eps}}\leq
(t-s)^{\frac{\be-\eps}2}\big(R_{\al,\beta}(T,\|B\|_{C_T\shc^{-\al}},\|g\|_{C_T\shc^{-\al}},\|u_T\|_{ D\shc^\beta})\|B\|_{C_T\shc^{-\al}}+\|g\|_{C_T\shc^{-\al}}\big)
\frac{4cT^{1-\eta}}{1-\be-\al}.
              \end{equation}
By \eqref{eq:vmild} we have
              \begin{align*}
		\| v(t)- v(s)\|_{\shc^{1+\eps}} \leq &  \left\|\int_t^T  P_{r-t}(\nabla u(r)\cdot B(r))dr-\int_s^T  P_{r-s}(\nabla u(r)\cdot B(r))dr\right\|_{\shc^{1+\eps}}\\ \notag
		&+\left\|\int_t^T  P_{r-t}g(r)dr-\int_s^T  P_{r-s}g(r)dr\right\|_{\shc^{1+\eps}}  \\        
		=&   \left\|\int_t^T  (I-P_{t-s})P_{r-t}(\nabla u(r)\cdot B(r))dr-\int_s^t  P_{r-s}(\nabla u(r)\cdot B(r))dr\right\|_{\shc^{1+\eps}}\\
		&+\left\|\int_t^T  (I-P_{t-s})P_{r-t}g(r)dr-\int_s^t  P_{r-s}g(r)dr\right\|_{\shc^{1+\eps}}\\
		\leq &\int_t^T \|( I-P_{t-s})P_{r-t}(\nabla u(r)\cdot B(r))\|_{\shc^{1+\eps}}dr+\int_s^t\| P_{r-s}(\nabla u(r)\cdot B(r))\|_{\shc^{1+\eps}}dr\\
		&+\int_t^T \left\| (I-P_{t-s})P_{r-t}g(r)\right\|_{\shc^{1+\eps}}dr+\int_s^t \left\| P_{r-s}g(r)\right\|_{\shc^{1+\eps}}dr\\
		=:&J^1_{t,s}+J^2_{t,s}.
	\end{align*}
 Now using 	\eqref{eq:schauder2},
	\eqref{eq:schauder1}, 
	\eqref{eq:bony}
	and we use the estimates \eqref{eq:nablaunDC}  on $u$, we obtain that
	\begin{align*}
		J^1_{t,s}\leq &  c(t-s)^{\frac{\be-\eps}2}\int_t^T \|P_{r-t}(\nabla u(r)\cdot B(r))\|_{\shc^{1+\beta}}dr+ c \int_s^t   (r-s)^{-\frac{1+\eps+\al}{2}}\| \nabla u(r)\cdot B(r) \|_{\shc^{-\al}}dr\\
		\leq & c(t-s)^{\frac{\be-\eps}2}\int_t^T (r-t)^{-\eta}\|\nabla u(r)\cdot B(r)\|_{\shc^{-\al}}dr+c\int_s^t (r-s)^{-\frac{1+\eps+\al}{2}}\|\nabla u(r)\cdot B(r)\|_{\shc^{-\al}}dr\\
		\leq & c(t-s)^{\frac{\be-\eps}2}\int_t^T (r-t)^{-\eta}\|\nabla u(r)\|_{\shc^\beta}\|B(r)\|_{\shc^{-\al}}dr+c\int_s^t (r-s)^{-\frac{1+\eps+\al}{2}}\|\nabla u(r)\|_{\shc^\beta}\|B(r)\|_{\shc^{-\al}}dr\\
		\leq & c\|\nabla u\|_{C_T\shc^\beta}\|B\|_{C_T\shc^{-\al}}\left((t-s)^{\frac{\be-\eps}2}\int_t^T (r-t)^{-\eta}dr+c\int_s^t (r-s)^{-\frac{1+\eps+\al}{2}}dr\right)\\
		\leq & cR_{\al,\beta}(T,\|B\|_{C_T\shc^{-\al}},\|g\|_{C_T\shc^{-\al}},\|u_T\|_{ D\shc^\beta})\|B\|_{C_T\shc^{-\al}}\left((t-s)^{\frac{\be-\eps}2}\int_t^T (r-t)^{-\eta}dr+c\int_s^t (r-s)^{-\frac{1+\eps+\al}{2}}dr\right)\\
		\leq &  cR_{\al,\beta}(T,\|B\|_{C_T\shc^{-\al}},\|g\|_{C_T\shc^{-\al}},\|u_T\|_{ D\shc^\beta})\|B\|_{C_T\shc^{-\al}}\left((t-s)^{\frac{\be-\eps}2} \frac{(T-t)^{1-\eta}}{1-\eta}+\frac{2}{1-\eps-\al}(t-s)^{\frac{1-\eps-\al}{2}}\right),
	\end{align*}
	where, in the second to last line, we have used the fact that $\|\nabla u\|_{C_T\shc^\beta}\leq \| u\|_{C_TD\shc^\beta}$ and
        then applied \eqref{eq:nablaunDC}.

	Since $\al<\be$ and from Assumption \ref{ass:beta}, we deduce $\be<1-\be$ and $\displaystyle{\frac{\be-\eps}2<\frac{1-\eps-\al}2}$. We furthermore observe that $1-\eta=\displaystyle{\frac{1-\eps-\al}2-\frac{\be-\eps}2}$, then,
      for all $s\leq t\leq T$ 
	\begin{align}\label{eq:I1tsDC}
		\notag   J^1_{t,s} \leq & cR_{\al,\beta}(T,\|B\|_{C_T\shc^{-\al}},\|g\|_{C_T\shc^{-\al}},\|u_T\|_{ D\shc^\beta})\|B\|_{C_T\shc^{-\al}}(t-s)^{\frac{\be-\eps}2}\left( \frac{(T-t)^{1-\eta}}{1-\eta}+\frac{2}{1-\eps-\al}(t-s)^{1-\eta}\right)\\
		\leq & cR_{\al,\beta}(T,\|B\|_{C_T\shc^{-\al}},\|g\|_{C_T\shc^{-\al}},\|u_T\|_{ D\shc^\beta})\|B\|_{C_T\shc^{-\al}}(t-s)^{\frac{\be-\eps}2}T^{1-\eta}  \frac{4}{1-\be-\al}.
	\end{align} 
Concerning $J^2_{t,s}$, by similar (but simpler) calculations as for $J^1_{t,s}$, we get
\begin{equation}
  \label{eq:I2tsDC}
  J^2_{t,s}    	\leq   c \|g\|_{C_T\shc^{-\al}}T^{1-\eta}(t-s)^{\frac{\be-\eps}2}  \frac{4}{1-\be-\al}.
			\end{equation}
                        Therefore  adding \eqref{eq:I1tsDC}  and \eqref{eq:I2tsDC},
                        we get \eqref{eq:holdvn}.

                      \item[2)] Next we analyse the impact of the terminal condition $u_T$, proving that,
               for  $0 \le s \le t \le T$ we have
  \begin{equation}\label{eq:PuTDC}
   \| P_{T-s} u_T -  P_{T-t} u_T\|_{D \shc^\varepsilon}  \le cC_T (t-s)^{\frac{\be-\eps}2} \Vert u_T \Vert_{D \shc^\beta}
 \end{equation}
 and
  \begin{equation}\label{eq:PuTC}
   \| P_{T-s} u_T -  P_{T-t} u_T\|_{\shc^{1+\varepsilon}}  \le c (t-s)^{\frac{\be-\eps}2} \Vert u_T \Vert_{\shc^{1+\beta}}.
 \end{equation}
 In order to estimate \eqref{eq:PuTDC}, we
borrow  Lemma \ref{lm:schauderDC}.
 We first make use of  \eqref{eq:schaud2DC}, with $f=P_{T-t}u_T$, $t=t-s$,
$\gamma = \beta$,
$\theta=\frac{\be-\eps}{2}$, then \eqref{eq:schaud1DC}, with $\theta=0$ and $t=T-t$.
For all $s\leq t\leq T$, we have
 \begin{align*}
  \| P_{T-s} u_T -  P_{T-t} u_T\|_{D\shc^{\eps}}&\leq  \| (P_{t-s}-I)P_{T-t} u_T\|_{D\shc^{\eps}}\leq c(t-s)^{\frac{\be-\eps}{2}} \|P_{T-t}u_T\|_{D\shc^\beta}\\
  &\leq cC_T (t-s)^{\frac{\be-\eps}{2}}  \Vert u_T \Vert_{D \shc^\beta}.
 \end{align*}
 For what concerns \eqref{eq:PuTC} the procedure is the same using \eqref{eq:schauder1} and \eqref{eq:schauder2} from Lemma \ref{lm:schauder}.
\end{itemize}
We set now 
\begin{equation} \label{eq:R0}
  R^0_{\alpha, \beta}(T, x,y,z) := (R_{\al,\beta}(T,x,y,z) x + y)c T^{1-\eta} \frac{4}{1-\be-\al} + c(1+C_T)z.
  \end{equation}
iii) follows from \eqref{eq:holdvn}, \eqref{eq:PuTC} and the triangle inequality. ii) follows again from the triangle inequality, \eqref{eq:PuTDC} and
\eqref{eq:holdvn} observing that
$$     \|v (t,\cdot)-v  (s,\cdot)\|_{D\shc^{\eps}} \leq
  \|v (t,\cdot)-v  (s,\cdot)\|_{\shc^{1+\eps}}.$$ 
 This concludes the proof of the theorem.


      \subsection{Continuity of the Kolmogorov PDE with
        respect to unbounded Besov-H\"older spaces\\  and
        a generalized Zvonkin's transform}
        
        \label{ssc:propshl}

We extend below some analytical properties proven in Lemmata 3.5-3.7 of \cite{issoglio_russoMPb} for the space $\shc^{(1+\beta)+}$
to the function space $D\shc^{\beta+}$. 





\begin{lemma}\label{lm:gen3.5}
  Let $g: [0,T] \times \R^d\to \R$ (resp. $h: [0,T] \times \R^d\to \R^d$) of class $C_TD\shc^{\beta+}$
  (resp. $C_T D\shc^{\beta+}(\R^d;\R^d)$). Let $(g^n)\subseteq C_T D\shc^{\beta+}$
  converging to $g$.
  Then $g\circ h\in C_T D\shc^{\beta+}$ 
  and $g^n\circ h\to g\circ h$ in $C_T D\shc^{\beta+}$.
\end{lemma}
\begin{remark} The function $g\circ h:[0,T]\times \R^d \rightarrow \R$
  is defined by $(g\circ h)(t,x) := g(t,h(t,x), (t,x) \in [0,T] \times \R^d$. 
  \end{remark}
\begin{proof}[Proof of Lemma \ref{lm:gen3.5}]
  Obviously $g\circ h$ globally continuous and  differentiable in space because it is composition of
  differentiable functions in space.
  Therefore proving that $g\circ h \in C_T D\shc^{\beta+}$
  reduces to proving that $\nabla (g\circ h) \in C_T \shc^{\beta+}(\R^d;\R^d)$.
  Since $\nabla (g\circ h)=\nabla g\circ h \nabla h$ is a matrix product of two 
  elements (componentwise)  in $C_T \shc^{\beta+}$, there exists an $\al>\be$ such that $\nabla g\circ h$
  and $\nabla h$ are (componentwise)  in $C_T \shc^\al$. Now 
taking into account \eqref{eq:bonyt}, it is immediate to show that
$C_T \shc^\al$ is an algebra.
Thus $\nabla (g\circ h)\in\shc^\al(\R^d;\R^d)$, hence in $\shc^{\beta+}(\R^d;\R^d)$. 
	
Since $g^n\to g$ in $D\shc^{\beta+}$, there exists an $\al>\be$ such that $g^n\to g$ in
$D\shc^\al$. We bound the $D\shc^\al$-norm of $(g^n\circ h-g\circ h)(t,\cdot)$,
uniformly in $t$.
It is made of two following terms.
	\begin{align*}
          \|(g^n\circ h-g\circ h)(t,\cdot)\|_{D\shc^\al}&=|(g^n\circ h)(t,0)-(g\circ h)(t,0)|+
          \|\nabla(g^n\circ h-g\circ h)(t,\cdot)\|_{\shc^\al} \\
         &\le \sup_{t \in [0,T]}|(g^n\circ h)(t,0)-(g\circ h)(t,0)| 
          +
          \|\nabla(g^n\circ h-g\circ h)\|_{C_T \shc^\al}.
	\end{align*}
	The first sum converges clearly to $0$ since $g^n$ converges uniformly to $g$.
        For what concerns the term involving the $\shc^\al$-norm we proceed by direct calculations, i.e. 
	\begin{align*}
          \|\nabla(g^n\circ h-g\circ h)\|_{C_T \shc^\al}=&
          \sup_{t\in[0,T]}\|\nabla(g^n\circ h-g\circ h)(t,\cdot)\|_{\infty}+
                                  \sup_{t\in[0,T]}\|\nabla(g^n\circ h-g\circ h)(t,\cdot)\|_{\al}\\
          \leq &\|\nabla(g^n-g)\|_{\infty}\|\nabla h\|_{\infty}+
                 \sup_{t\in[0,T]} \|\nabla(g^n\circ h-g\circ h)(t,\cdot)\|_{\al}\\
          \leq &\|\nabla(g^n-g)\|_{C_T \shc^\al}\|\nabla h\|_{\infty}+
                 \sup_{t\in[0,T]} \|\nabla(g^n\circ h-g\circ h)(t,\cdot)\|_{\al}.\\
	\end{align*}
	The first summand converges to $0$ thanks to convergence of $g^n$ in $C_T D\shc^\al$, whilst we analyse the second one separately. Thanks to Remark \ref{rm:equivnorm} we can make use of the equivalent formulation of the $\shc^\al$-norm. Omitting $t$ we obtain
	\begin{align*}
		\|\nabla(g^n\circ h-g\circ h)\|_{\al}
		=&\sup_{\underset{|x-y|<1}{x,y\in\R^d}}\frac{\nabla(g^n\circ h-g\circ h)(x)-\nabla(g^n\circ h-g\circ h)(y)}{|x-y|^\al}\\
		=&\sup_{\underset{|x-y|<1}{x,y\in\R^d}}\frac{\nabla(g^n-g)(h(x))\nabla h(x)-\nabla(g^n-g)(h(y))\nabla h(y)}{|x-y|^\al}\\
		= &\sup_{\underset{|x-y|<1}{x,y\in\R^d}}\frac{(\nabla(g^n-g)(h(x))-\nabla(g^n-g)(h(y)))\nabla h(x)+\nabla(g^n-g)(h(y))(\nabla h(x)-\nabla h(y))}{|x-y|^\al}\\
		\leq &\|\nabla(g^n-g)\|_{\infty}\sup_{\underset{|x-y|<1}{x,y\in\R^d}}\frac{|h(x)-h(y)|}{|x-y|^\al}\|\nabla h\|_{\infty}+\|\nabla(g^n-g)\|_{\infty}\|\nabla h\|_{\al}\\
          \leq &\|\nabla(g^n-g)\|_{\infty}\|\nabla h\|^2_{\infty}+\|\nabla(g^n-g)\|_{\infty}\|\nabla h\|_{\al}
                 \le \|\nabla(g^n-g)\|_{\shc^\al}\left(\|\nabla h\|^2_{\infty}+\|\nabla h\|_{\al}\right).
	\end{align*}
        Consequently, taking the sup in $t$, we get
        \begin{equation*}
          \sup_{t\in[0,T]} \|\nabla(g^n\circ h-g\circ h)(t,\cdot)\|_{\al} \le
          \|\nabla(g^n-g)\|_{C_T \shc^\al}\left(\|\nabla h\|^2_{C_T \shc^{\alpha}}+\|\nabla h\|_{C_T \al}\right).
          \end{equation*}
        Since $g^n\to g$ in $C_T D\shc^\al$, letting $n$ go to infinity concludes the proof.
\end{proof}

Throughout the rest of the  section $\phi$ will be a function satisfying the following assumptions.
\begin{assumption}\label{ass:phiell}
	\begin{itemize}
        \item[(i)] There exists $\ell = (\ell^1,\ldots, \ell^d)^\top \in C_T\shc^{0+}(\R^d;\R^d)$ such that
          $\phi=(\phi^1,\ldots,\phi^d)^\top \in
          C_TD\shc^{(1-\beta)-}(\R^d;\R^d)$ is the unique solution to  $\shl \phi (t,z)=\ell(t,z), \ \phi(T,\cdot)=id$. 
		\item[(ii)] $u:=\phi-id$ is of class $C_T\shc^{(2-\beta)-}$.
		\item[(iii)] $\nabla \phi(t,z)\not=0$ for all $(t,z)\in[0,T]\times \R^d$, $\phi(t,\cdot)$ is bijective for all $t\in[0,T]$.
                 
	\end{itemize}
\end{assumption}
\begin{remark}\label{rmk:philam}
  \begin{enumerate}
    \item  In the sequel $\psi:[0,T]\times \R^d\to\R$
      will denote, the function such that, for each $t$, 
      $\phi^{-1}(t,\cdot)=\psi(t,\cdot)$.
    \item In the special case of $\ell=\lam (\phi_\lam-id)$ as in (4.30) of \cite{issoglio_russoPDEa}
      this $\phi:=\phi_\lam$ exists.
\end{enumerate}
    \end{remark}
\begin{remark} \label{rmk:phi}
  According to Theorem 4.7 in \cite{issoglio_russoPDEa}, since
  $\phi_T\in \shc^{(2-\beta)-}(\R^d;\R^d)$, we have $\phi\in C_T\shc^{(2-\beta)-}(\R^d,\R^d)$.
\end{remark}

\begin{lemma}\label{lm:gen3.6}
  Let $\tilde f \in C_{ubuc}^{1,2}([0,T]\times \R^d)$ and let $\phi$ satisfy
  Assumption \ref{ass:phiell}. Then  $f:=\tilde f \circ \phi \in C_TD\shc^{\beta+}$.
\end{lemma}
\begin{proof}
  The idea consists in making use of Lemma \ref{lm:gen3.5} with $g = \tilde f$
  and $h = \phi$.
  Since $\phi$ fulfills Assumption \ref{ass:phiell},
  it remains to prove that $\tilde f \in C_T D\shc^{\beta+}$.
  Since $\tilde f$ is continuous, we need to prove that
  $\nabla \tilde f \in C_T\shc^{\beta+}$.
  Setting $H:= \partial_{x_k} \tilde f$ for a given $1\le k \le d$,
  we prove that $H \in C_T\shc^{\beta+}$.
  Let $\alpha > \beta$ and $0 \le t \le T$.
  For $s \in [0,T]$ we have
	\begin{equation}\label{eq:1gen3.6}
          \left\|H(s,\cdot)-H(t,\cdot) \right\|_{\shc^\al}
          =\left\|H(s,\cdot)-H(t,\cdot) \right\|_{\infty}
          +\left\|H(s,\cdot)-H(t,\cdot) \right\|_{\al}.
	\end{equation}
        Since $H$ is uniformly continuous, the  first summand converges to zero when $s$ goes $t$.
We now come to the $\al$-seminorm and we get
	\begin{align}\label{eq:3gen3.6}
          \notag \left\|H(t,\cdot)-H(s,\cdot) \right\|_{\al}=
          &\sup_{\underset{|x-y|<1}{x,y\in\R^d}}\frac{|H(t,x)-H(t,y)-H(s,x)+H(s,y)|}{|x-y|^\al}\\
          \notag \leq & \sup_{\underset{|x-y|<1}{x,y\in\R^d}}
                        \frac{\int_0^1|\nabla H(t,ax+(1-a)y)-\nabla H(s,ax+(1-a)y)|da|x-y|}{|x-y|^\al}\\
		\leq & \sup_{\underset{|x-y|<1}{x,y\in\R^d}}\int_0^1|\nabla H(t,ax+(1-a)y)-\nabla H(s,ax+(1-a)y)|da.
	\end{align}
        The latter converges to zero when $s$ goes to $t$ since $\nabla H$
        is uniformly continuous in time uniformly in space.
        Indeed ${\rm Hess} \tilde f$ is uniformly continuous.

      \end{proof}

\begin{lemma}\label{lm:gen3.7}
  Let $\phi$ satisfy Assumption \ref{ass:phiell}, $\tilde f\in C_{ubuc}^{1,2}$ and $\tilde \shl$ be the operator given by
  \begin{equation}\label{eq:tildeL}
		\tilde \shl \tilde f:= \partial_t \tilde f+\frac12 {\rm Tr}\left[(\nabla \phi \circ \psi)^\top\cdot{\rm Hess}\tilde f\cdot( \nabla \phi \circ \psi)\right]+ \nabla \tilde f \cdot \ell.
	\end{equation}
	Then
\begin{equation} \label{eq:dl0}
  f:=\tilde f\circ \phi \in C_TD\shc^{\beta}\cap C^1([0,T];\shs')
\end{equation}
  and 
	\begin{equation}\label{eq:gen3.7}
		(\tilde \shl \tilde f)\circ \phi= \shl f,
	\end{equation}
	in $C_T\shc^{0+}$, or equivalently
	\begin{equation}\label{eq:gen3.7bis}
		\tilde \shl \tilde f= (\shl f) \circ \psi.
	\end{equation}
	
\end{lemma}

\begin{proof}
  To show that  $f\in C^1([0,T], \mathcal S')$ we compute the time
  partial derivative  $\partial_t  f $. 
	Recall that  $\tilde f \in C^{1,2}$ by assumption, and that $\phi:[0,T]\times \R^d \to \R^d$ and $f, \tilde f: [0,T]\times \R^d \to \R$. We have
	\begin{equation}\label{eq:fdot}
          t \mapsto \partial_t f (t, \cdot) =  {\partial_t \tilde f }(t, \phi(t, \cdot)) + \sum_{k=1}^d \partial_{x_k} \tilde f(t, \phi(t, \cdot))
          \cdot \partial_t \phi^k (t, \cdot).
	\end{equation}
	We show that  the right-hand side of equation \eqref{eq:fdot} is  in $C_T\mathcal S'$. 

	For the first term in  \eqref{eq:fdot}  we clearly have the claim because   $ \partial_t {\tilde f} \circ \phi$ is uniformly continuous  in $t,x$. 
	The second term in  \eqref{eq:fdot}
	has products of the form  $(\partial_{x_k} \tilde f \circ \phi)
       \cdot \partial_t \phi^k$, where $\partial_t\phi^k \in C_T \mathcal C^{(-\beta)-}$, thanks to Proposition 4.15 of \cite{issoglio_russoPDEa},  and
        $\partial_{x_k} \tilde f \circ \phi \in C_T \mathcal C^{\beta+}$. Hence, the product is well-defined and continuous by  \eqref{eq:bonyt}.
	This  shows that $f\in C^1([0,T]; \mathcal S')$
        and by Lemma \ref{lm:gen3.6},   \eqref{eq:dl0} is now established.

	We now apply $\mathcal L$ to $f$, and to do so we need to calculate the derivatives of $f$. 
	The first space derivative $\partial_{x_i} f$ of $f$ is
	\begin{align*}
		\partial_{x_i} f (t, \cdot) = \sum_{k=1}^d  \partial_{x_k}  \tilde f (t, \phi(t, \cdot)) \partial_{x_i} \phi^k (t, \cdot) , t\in[0,T]
	\end{align*}
	and the second derivative is given by
	\begin{align*}
          \partial_{x_ix_i} f (t, \cdot) &=  \sum_{k=1}^d\left [ \sum_{l=1}^d ( \partial_{x_lx_k}  \tilde f (t, \phi(t, \cdot)) \partial_{x_i} \phi^l (t, \cdot) ) \partial_{x_i} \phi^k (t, \cdot)  + \partial_{x_k}
         \tilde f (t, \phi(t, \cdot)) \cdot \partial_{x_ix_i} \phi^k(t, \cdot) \right]\\
		&=   \left ((\nabla \phi)^T (\text{Hess}(\tilde f)\circ \phi)  \nabla \phi  \right)_{ii}(t, \cdot)+ \sum_{k=1}^d  \partial_{x_k} \tilde f (t, \phi(t, \cdot)) \cdot \partial_{x_ix_i} \phi^k(t, \cdot) , t\in[0,T].
	\end{align*}
	Note that $\partial_{x_i} f (t, \cdot)$ for all $ t\in[0,T]$ is a function in $ \mathcal C^{\beta+}$ by  Lemma \ref{lm:gen3.6}, thus it is a well-defined object in $\mathcal C^{(-\beta)-}$.
	The second derivative   $\partial_{x_ix_i} f (t, \cdot) $ is made of two terms: the first one  is  a bounded function by assumptions on $\tilde f$ and $\phi$,  and the second one is well-defined in $\mathcal C^{(-\beta)-}$ again by means of the pointwise product \eqref{eq:bony}, where for all  $ t\in[0,T]$  the distributional term $\partial_{x_ix_i}\phi^k (t, \cdot)$ is in $\mathcal C^{ (-\beta)-}$  since $\partial_{x_i}\phi^k (t, \cdot) \in \mathcal C^{ (1-\beta)-} $, by Assumption \ref{ass:phiell}. Using these we determine $\shl f$, i.e.
	\begin{align}\label{eq:Lf} 
		\nonumber
		(\mathcal L f) (t, \cdot)
		=& \partial_t  {\tilde f} (t, \phi(t, \cdot)) + \frac12 \sum_{i=1}^d \left ( (\nabla \phi)^T (\text{Hess}(\tilde f)\circ \phi) \nabla \phi \right)_{ii}(t,\cdot)\\
		&+ \sum_{k=1}^d  \partial_{x_k} \tilde f (t, \phi(t, \cdot)) \left[ \partial_t \phi^k (t, \cdot) + \frac12  \sum_{i=1}^d\partial_{x_ix_i} \phi^k(t, \cdot) +\partial_{x_i} \phi^k (t, \cdot) \cdot B^i(t, \cdot) \right], t\in [0,T],
	\end{align}
	where the last term $(\partial_{x_k} \tilde f (t, \phi(t, \cdot))
        \partial_{x_i} \phi^k (t, \cdot)) \cdot  B^i(t, \cdot)$ is well-defined in $\mathcal C^{(-\beta)+}$  by \eqref{eq:bony}  used twice and Remark 2.7 of \cite{BIRMcKean}. Thus  equality \eqref{eq:Lf} holds  in the space $\mathcal C^{(-\beta)-}$.
	Now we observe that $\mathcal L \phi^k   = \ell^k$ 
	because again of Assumption \ref{ass:phiell}.
	Thus the equality above becomes
	\begin{align}\label{eq:Lf2}
		\nonumber
		(\mathcal L f) (t, \cdot)
          =& \partial_t  {\tilde f} (t, \phi(t, \cdot)) + \frac12 \sum_{i=1}^d \left ((\nabla \phi)^T (\text{Hess}(\tilde f)\circ \phi)  \nabla \phi \right)_{ii}(t, \cdot)+ \sum_{k=1}^d  \partial_{x_k} \tilde f (t, \phi(t, \cdot))
             \ell^k(t,\cdot)\\ 
		=& \partial_t  {\tilde f} (t, \phi(t, \cdot)) + \frac12 \text{Tr} \left (  (\nabla \phi)^T (\text{Hess}(\tilde f)\circ \phi) \nabla \phi\right)(t, \cdot)+  \nabla \tilde f (t, \phi(t, \cdot)) \ell(t,\cdot) , t \in [0,T].
	\end{align}
	On the other hand, applying $\tilde \shl$ to $\tilde f \in C^{1,2}_{ubuc}$ and then composed with $\phi$ together with the identity $\psi(t, \phi(t, \cdot)) = id$,  one easily gets 
	\begin{align}\label{eq:tLf}
		\nonumber
		(\tilde {\mathcal L} \tilde f )(t, \phi(t, \cdot))
		=& \partial_t{  \tilde f} (t, \phi(t, \cdot)) + \frac12 \text{Tr} \left ( (\nabla \phi)^T (\text{Hess}(\tilde f)\circ \phi)  \nabla \phi \right)(t, \cdot)\\
		& + \nabla \tilde f (t, \phi(t, \cdot)) \ell(t,\cdot), t\in[0,T].
	\end{align}
	Now using \eqref{eq:Lf2} and \eqref{eq:tLf} we get
        $t\mapsto \mathcal L f(t,\cdot) =(\tilde {\mathcal L} \tilde f )(t, \phi(t, \cdot)) $ in $C([0,T];\mathcal S')$. 
	We observe that the right-hand side of \eqref{eq:tLf} belongs to $C_T \mathcal C^{0+}$. 
	Setting $g:=(\tilde {\mathcal L} \tilde f )\circ \phi$ we can conclude that  $\mathcal L f =g \in C_T \mathcal C^{0+}$. Given that both sides are functions, we can compose them with $\psi $ to get $\tilde {\mathcal L} \tilde f=  ({\mathcal L}  f  ) \circ \psi $.
\end{proof}

   {\bf Acknowledgements.} 
The author L.\ Bondi was partially supported by GNAMPA Progetti di Ricerca 2026 CUP: E53C25002010001.
The author E.\ Issoglio was partially supported by
PRIN-PNRR2022 (P20224TM7Z) CUP: D53D23018780001, from EU - Next Generation
EU - PRIN2022 (202277N5H9) CUP: D53D23005670006 and  GNAMPA Progetti di Ricerca 2026 CUP: E53C25002010001.
   The author F.\ Russo was  supported by the  ANR-22-CE40-0015-01 (SDAIM). 

\bibliographystyle{plain}
\bibliography{../../BIBLIO_FILE/Biblio_Tesi_Luca}

\end{document}